\documentclass[11pt]{article}
\usepackage[utf8]{inputenc}

\usepackage{amsmath,amsthm,amssymb}
\usepackage{graphicx}
\usepackage{wrapfig}
\usepackage{enumitem}
\usepackage{tikz-cd}
\usepackage{esvect}
\usepackage{bm}
\usepackage{tocloft}
\usepackage{lipsum} 
\usepackage{mathrsfs}
\usepackage{float}
\usepackage{graphicx}
\usepackage{commath}
\usepackage{mathtools}
\usepackage{ulem}
\usepackage[nottoc]{tocbibind}

\usepackage{caption}
\usepackage{subcaption}

\usetikzlibrary{backgrounds, matrix,fit,matrix,decorations.pathreplacing, calc, positioning}

\makeatletter
\newcommand\setItemnumber[1]{\setcounter{enum\romannumeral\@enumdepth}{\numexpr#1-1\relax}}
\makeatother

\newcommand{\R}{\mathbb{R}}
\newcommand{\C}{\mathbb{C}}

\newcommand{\e}{\varepsilon}

\newcommand{\beq}{\begin{equation}}
\newcommand{\eeq}{\end{equation}}
\newcommand{\bdef}{\begin{definition}}
\newcommand{\eedef}{\end{definition}}
\newcommand{\bpf}{\begin{proof}}
\newcommand{\epf}{\end{proof}}

\newtheorem{theorem}{Theorem}[section]
\newtheorem{lemma}[theorem]{Lemma}
\newtheorem{proposition}[theorem]{Proposition}

\newtheorem{corollary}[theorem]{Corollary}
\newtheorem{definition}[theorem]{Definition}
\newtheorem{claim}[theorem]{Claim}
\newtheorem{conjecture}[theorem]{Conjecture}

\theoremstyle{definition}

\newtheorem{remark}[theorem]{Remark}

\makeatletter
\newtheorem*{rep@theorem}{\rep@title}
\newcommand{\newreptheorem}[2]{%
\newenvironment{rep#1}[1]{%
 \def\rep@title{#2 \ref{##1}}%
 \begin{rep@theorem}}%
 {\end{rep@theorem}}}
\makeatother

\newreptheorem{theorem}{Theorem}
\newreptheorem{lemma}{Lemma}
\newreptheorem{proposition}{Proposition}

\PassOptionsToPackage{hyphens}{url}\usepackage{hyperref} 

\usepackage{breakurl}
\hypersetup{
  breaklinks=true,
  colorlinks=true,
  citecolor=blue,
  linkcolor=blue, 
  urlcolor=cyan
}

\def\?{{\bf ****???????****}}
\def\M{{\mathcal M}}

\newcommand{\defeq}{\vcentcolon=}

\long\def\frame#1#2#3#4{\hbox{\vbox{\hrule height#1pt
 \hbox{\vrule width#1pt\kern #2pt
 \vbox{\kern #2pt
 \vbox{\hsize #3\noindent #4}
\kern#2pt}
 \kern#2pt\vrule width #1pt}
 \hrule height0pt depth#1pt}}}

\def\calC{\mathcal{C}}
\def\calP{\mathcal{P}}

\def\?{{\bf ****???????****}}
\def\M{{\mathcal M}}
\def\B{{\mathcal{B}}}
\def\q{\quad}

\newcommand\blfootnote[1]{%
	\begingroup
	\renewcommand\thefootnote{}\footnote{#1}%
	\addtocounter{footnote}{-1}%
	\endgroup
}

\usepackage{hyperref}
\usepackage[backend=biber, maxbibnames=99, citestyle=numeric, bibstyle=numeric, sorting=nyt, backref = false]{biblatex}

\renewbibmacro{in:}{} 

\DeclareFieldFormat{journaltitle}{#1\isdot} 

\DeclareFieldFormat*{title}{\mkbibitalic{#1}} 

\DefineBibliographyExtras{english}{%
  }

\DeclareFieldFormat[article]{volume}{\mkbibbold{#1}} 

\usepackage{csquotes}
\usepackage{xcolor}
\usepackage{xpatch}
\colorlet{partnumbercolour}{blue}
\makeatletter
\AtBeginDocument{
\xpatchcmd{\@part}{\partname\nobreakspace\thepart}{\textcolor{partnumbercolour}{\thepart}}{}{}
}
\makeatother

\title{Two-dimensional B\l{}ocki, $L^p$-Mahler, and Bourgain conjectures}
\author{Vlassis Mastrantonis, Yanir A. Rubinstein }

\begin{document}
\let\i\undefined
\newcommand{\i}{\sqrt{-1}}
\def\LpK{K^{\circ,p}}
\def\LpS{S^{\circ,p}}
\def\Cov{\mathrm{Cov}}
\def\co{\hbox{co}\,}
\numberwithin{equation}{section}

\maketitle
\date

\begin{abstract}
We confirm, in dimension two, 
B\l{}ocki's conjectures
on sharp lower bounds for Bergman kernels of tube domains. 
To that end, we verify a broader class of $L^p$-Mahler conjectures  due to Berndtsson and the authors, where $p=1$ 
are B\l{}ocki's conjecture, and $p=\infty$ are Mahler's conjectures.
The proofs are technically challenging as the $L^p$-Mahler
volume is considerably harder to deal with analytically compared to Mahler's volume,
and furthermore duality is lost. In addition, unlike in the classical
Mahler setting,
the non-symmetric setting is considerably more involved than the symmetric one.
The proofs involve studying the effect of
Mahler's classical sliding of vertices on 
two-dimensional polytopes on the $L^p$-polar body (no longer a polytope).
Some arguments are inspired by works of Campi--Gronchi and Meyer--Reisner on volumes of classical polar bodies of shadow systems. 
In passing, we also explore how Mahler's sliding affects the isotropic constant. This leads to an elementary proof of Bourgain's strong hyperplane conjectures in dimension two, originally due to 
Bisztriczky--B\"or\"oczky, 
Campi--Colesanti--Gronchi and Meckes. Specifically, we show that, as a function of the sliding parameter, the isotropic constant raised to an appropriate power is a convex quadratic polynomial. 
\end{abstract}

\blfootnote{Research supported by NSF grants 
DMS-1906370,2204347, and a UMD Graduate School Summer Research Fellowship. Thanks to the referee for a very careful reading.}

\tableofcontents

\section{Introduction}

This article is a continuation of our study of 
interactions between convex analysis and complex analysis in
the context of Mahler conjectures in convexity \cite{MR22,BMR23}. The main
theorem is a verification of conjectures of B\l{}ocki in dimension
$n=2$---purely
complex analytic statements concerning Bergman kernels of tube domains in $\C^2$---using
methods of {\it convex} analysis.
The bridge between convex and complex is provided by 
the theory of $L^p$-polarity,
introduced in our previous work, 
where B\l{}ocki's conjectures are precisely captured
by the $L^1$-Mahler volume (while Mahler's original conjectures
correspond to $p=\infty$).
With this in mind we have strived to make the article accessible to both convex and complex analysts.
Also, despite its title, this article also continues to develop
a few aspects of the theory of $L^p$-polarity in all dimensions.

\subsection{B\l{}ocki conjectures}
A famous conjecture of Mahler from the 1930's 
\cite{Mahler39a} stipulates that the \textit{Mahler volume},  
\begin{equation*}
    \M(K)\defeq n! |K||K^{\circ}|, 
\end{equation*}
is minimized by the cube $[-1,1]^n$ among symmetric convex bodies (i.e.,
$K=-K$),
where
\begin{equation*}
    K^\circ\defeq \{y\in \R^n: \langle x,y\rangle\leq 1 \text{ for all } x\in K\}, 
\end{equation*}
is the polar of $K$.
A `non-symmetric' analog of this conjecture, also
due to Mahler \cite{Mahler39b},
states that $\M$ is minimized among all convex bodies by the simplex 
\begin{equation*}
    \Delta_{n,0}\defeq \co\big\{e_1,\ldots,e_n,-{\textstyle\sum_{i=1}^n}e_i\big\},
\end{equation*}
where $e_1, \ldots, e_n$ is the standard basis in $\R^n$,
and $\co\mskip1mu A$ denote
the convex hull of $A$.

In a beautiful paper, Nazarov established a connection between the Mahler volume and the Bergman kernels of tube domains \cite{Nazarov12}.
For a domain $\Omega\subset \C^n$, denote by 
\begin{equation*}
    \mathcal{B}_{\Omega}: \Omega\times \Omega\to \C
\end{equation*}
the \textit{Bergman kernel} of $\Omega$. 
The Bergman kernel is holomorphic in the first variable, anti-holomorphic in the second, and is uniquely characterized as the reproducing kernel of the evaluation functional. In other words, for any $w\in \Omega$ and holomorphic $f\in L^2(\Omega)$, 
\begin{equation*}
    \langle f, \B_\Omega(\cdot,w)\rangle_{L^2}\defeq \int_{\Omega} f(z) \overline{\B_\Omega(z,w)}\, d\lambda(z) = f(w).
\end{equation*}
Nazarov's idea was to exploit the explicit formula for the Bergman kernel of tube domains
\begin{equation*}
    T_K\defeq \R^n+ \i\, K, 
\end{equation*}
to obtain a lower bound on the Mahler volume of a symmetric convex body $K$,
\begin{equation}
\label{MBeq}
    \M(K)\geq \pi^n |K|^2 \mathcal{B}_{T_K}(0,0). 
\end{equation}
It was later noted by Hultgren \cite{Hultgren13}, Berndtsson \cite{Bern22} and us \cite{MR22}, that \eqref{MBeq} holds in general for convex bodies whose barycenter
\begin{equation*}
    b(K)\defeq \int_{\R^n} x\frac{d x}{|K|}
\end{equation*}
lies at the origin.
To obtain a lower bound on the Mahler volume, it is therefore enough to find a lower bound on the Bergman kernel $\mathcal{B}_{T_K}$. Nazarov specifically focuses on symmetric convex bodies, but his approach generalizes to non-symmetric
bodies with more work \cite{MR22}. He constructs a holomorphic function on $T_K$ that evaluates to one at the origin with $L^2$-norm bounded from above by $(16/\pi^2)^n |K|^2$ multiplied by terms exhibiting sub-exponential dependence on dimension. To do this he uses H\"ormander's $\overline{\partial}$-theorem \cite[Theorem 2.2.1]{Horm65}. 
B\l{}ocki revisited Nazarov's approach and reproved the same bound on the Bergman kernel of tube domains of symmetric convex bodies
\cite[(2)]{Blocki15}
\begin{equation}
\label{BTKboundEq}
    |K|^2\mathcal{B}_{T_K}(0,0)\geq  ({\pi}^2/{16})^{n}.
\end{equation}
An alternative proof of \eqref{BTKboundEq}, due to Lempert \cite[p. 464]{Blocki14b}, can be given by invoking Berndtsson's plurisubharmonicity theorem for Bergman kernels \cite[Theorem 1.1]{Bern06}. B\l{}ocki then conjectured that \eqref{BTKboundEq} is suboptimal, and in fact, the sharp lower bound should be obtained by the cube \cite[(7)]{Blocki14a}:

\begin{conjecture}
\label{BlockiConjecture}
{\rm (B\l{}ocki 2014)}    For a symmetric convex body $K\subset \R^n$, 
    \begin{equation*}
       |K|^2\mathcal{B}_{T_K}(0,0)\geq |[-1,1]^n|^2 \,\mathcal{B}_{T_{[-1,1]^n}}(0,0)= (\pi/4)^{n}. 
    \end{equation*}
\end{conjecture}
A natural `non-symmetric' analog has also been conjectured \cite[Conjecture 10]{MR22}:
\begin{conjecture}
    \label{MRConjecture}
    For a convex body $K\subset \R^n$, 
    \begin{equation*}
        |K|^2\mathcal{B}_{T_K}(0,0)\geq \inf_{x\in \R^n} \Big( |\Delta_{n,0}|^2 \mathcal{B}_{T_{\Delta_{n,0}-x}}(0,0)\Big). 
    \end{equation*}
\end{conjecture}

Our main result is a verification of Conjectures \ref{BlockiConjecture}--\ref{MRConjecture}
in dimension two. 

\begin{theorem}
\label{blockiConj}
    For a symmetric convex body $K\subset \R^2$, $|K|^2 \mathcal{B}_{T_K}(0,0)\geq \pi^2/16$. 
\end{theorem}
\begin{theorem}
\label{blockinonsymThm}
    For a convex body $K\subset \R^2$, 
    $|K|^2\mathcal{B}_{T_K}(0,0)\geq 
\inf_{x\in \Delta_{2,0}} \Big( |\Delta_{2,0}|^2 \mathcal{B}_{T_{\Delta_{2,0}-x}}(0,0)\Big). 
    $
\end{theorem}

In fact, Theorems \ref{blockiConj}-\ref{blockinonsymThm}
follow (see \eqref{BMEq}) from more general results, 
Theorems \ref{2DimpMahlerSym}--\ref{2DimpMahler},
that we next turn to describe.

\begin{remark}
Determining when equality occurs in 
Conjectures \ref{BlockiConjecture}--\ref{MRConjecture}  is likely a subtle issue, and was
not part of B\l{}ocki's original conjecture---see
Remark \ref{MinRemark} for a discussion.
\end{remark}

\subsection{\texorpdfstring{The $L^p$-Mahler conjectures}{The Lᵖ-Mahler conjectures}}
\label{lpMahlerIntro}
Motivated by Nazarov's proof of the Bourgain--Milman theorem \cite{Nazarov12},
we introduced $L^p$ versions of the support function 
\cite[Remark 36]{MR22},\cite[(1.8)]{BMR23} 
\begin{equation}
\label{hpKDef}
    h_{p,K}(y)\defeq \log\left( \int_K e^{p\langle x,y\rangle}\frac{dx}{|K|}\right)^{\frac1p}. 
\end{equation}
These functions are smooth and convex
\cite[Lemma 2.1]{BMR23}
and are related to the classical support function
\cite[Corollary 2.7]{BMR23}
\begin{equation*}
    h_{\infty, K}(y)\defeq \lim_{p\to\infty} h_{p,K}(y)= \sup_{x\in K} \langle x,y\rangle =\vcentcolon h_{K}(y).
\end{equation*}

However, the \textit{$L^p$-support functions} lack a hallmark property of the classical support function: homogeneity. Consequently, they cannot be directly associated with a geometric object. This is in contrast to $h_K$, which, owing to its homogeneity, allows for the definition of the polar body through its sublevel sets,
\begin{equation*}
    K^{\circ}\defeq \{y\in \R^n: h_K(y)\leq 1\}.
\end{equation*}
Nonetheless, as demonstrated by Ball \cite[Theorem 5]{Ball88} (and later by Klartag who removed the symmetry assumption \cite[Theorem 2.2]{Klartag06}), any convex function can be used to define a `near' norm (in the sense that it is only positively homogeneous (and actually a norm in the symmetric case)) and hence a convex body. That is,
\begin{equation}
\label{KcircpNormEq}
    \|y\|_{K^{\circ,p}}\defeq \left( \frac{1}{(n-1)!}\int_0^\infty r^{n-1} e^{-h_{p,K}(ry)} dr\right)^{-\frac1n} 
\end{equation}
is non-negative, vanishes only when $y=0$, satisfies the triangle inequality, and is positively $1$-homogeneous, i.e., $\|\lambda y\|_{K^{\circ,p}}= \lambda \|y\|_{K^{\circ,p}}$ for $\lambda>0$ (for a detailled exposition see \cite[Proposition A.1]{BMR23}). Defining the \textit{$L^p$-polar body} via
\begin{equation*}
    K^{\circ,p}\defeq \{y\in \R^n: \|y\|_{K^{\circ,p}}\leq 1\}, 
\end{equation*}
one has \cite[Theorem 1.2]{BMR23}
\begin{equation}
\label{hpKKcircpEq}
    \int_{\R^n} e^{-h_{p,K}(y)} dy= n! |K^{\circ,p}|. 
\end{equation}
Motivated by the relation between $h_K$ and the classical 
Mahler volume 
\begin{equation}
\label{Meq}
    \M(K)\defeq n! |K| |K^{\circ}|= |K| \int_{\R^n}e^{-h_{K}(y)}dy 
\end{equation}
we introduced the \textit{$L^p$-Mahler volume}
\begin{equation}
\label{MpDefEq}
    \M_p(K)\defeq |K|\int_{\R^n} e^{-h_{p,K}(y)} dy= n! |K| |K^{\circ,p}|. 
\end{equation}
Clearly, \eqref{MpDefEq} is a much more complicated
quantity to deal with compared to \eqref{Meq} as
it involves a double integral. Moreover,
duality is lost in the sense that
neither the polar nor the $L^p$-polar of $K^{\circ,p}$
equals $K$.
Yet, the $L^p$-Mahler volumes exhibit many of the signature properties of the Mahler volume: invariance under the action of $GL(n,\R)$ \cite[Lemma 4.6]{BMR23}, tensoriality \cite[Remark 2.3]{BMR23}, existence and uniqueness of a Santal\'o point \cite[Proposition 1.5]{BMR23}.  Additionally, a Santal\'o inequality holds for the $L^p$-Mahler volume
\cite[Theorem 1.6]{BMR23} \cite[Theorem 1.2]{Mastr23}, 
i.e., ellipsoids are maximizers of the $L^p$-Mahler volumes for all $p>0$. 

The motivation for the present paper and its overarching approach
is the key observation \cite{MR22,BMR23} that $\M_1$ is closely related to the Bergman kernels of
the previous section. By a formula
of Rothaus \cite[Theorem 2.6]{Rothaus60}, Kor\'anyi \cite[Theorem 2]{Koranyi62}, and Hsin \cite[(1.2)]{Hsin05} (see \cite[Lemma 32]{MR22} for a self-contained proof using the Paley--Wiener Theorem), 
\begin{equation}
\label{BTKeq}
    \mathcal{B}_{T_K}(z,w)= \frac{1}{(4\pi)^n |K|} \int_{\R^n} e^{-\i\langle \frac{z-\overline{w}}{2}, x\rangle- h_{1,K}(x)}d x.
\end{equation}
Hence, \eqref{BTKeq} on the diagonal gives, 
\begin{equation}
\label{BMEq}
\begin{aligned}
   |K|^2 \mathcal{B}_{T_K}(z,z)&= \frac{|K|}{(4\pi)^n} \int_{\R^n} e^{-\i \langle \frac{z-\overline{z}}{2}, x\rangle- h_{1,K}(x)} dx\\
   &= \frac{|K|}{(4\pi)^n} \int_{\R^n} e^{\langle \mathrm{Im}\,z, x\rangle-h_{1,K}(x)} dx= \frac{\M_1(K-\mathrm{Im}\,z)}{(4\pi)^n}, 
\end{aligned}
\end{equation}
because $h_{1,K}(x)-\langle \mathrm{Im}\,z, x\rangle= h_{1,K-\mathrm{Im}\,z}(x)$ \cite[Lemma 2.2 (ii)]{BMR23}.
In conclusion, Conjectures \ref{BlockiConjecture}--\ref{MRConjecture}
are special cases of the following statements
\cite[Conjectures 1.3--1.4]{BMR23}:

\begin{conjecture} 
\label{pMahlerSym}
    Let $p\in (0,\infty)$. For a symmetric convex body $K\subset \R^n$, 
    $$
    \M_p([-1,1]^n)\leq \M_p(K).
    $$ 
\end{conjecture}
\begin{conjecture} 
\label{pMahler}
    Let $p\in (0,\infty)$. For a convex body $K\subset \R^n$, 
    $$
    \inf_{x\in\Delta_{n,0}}\M_p(\Delta_{n,0}-x) \le \M_p(K).
    $$
\end{conjecture}

 \begin{remark}
 \label{MinRemark}
Conjecture \ref{pMahlerSym} implies the symmetric Mahler conjecture \cite[Lemma 3.12]{BMR23} but has a nuanced difference. Indeed,
    for finite $p$, the conjectured minimizers of $\M_p$ among symmetric convex bodies are the 
$GL(n,\R)$-orbit of the cube $[-1,1]^n$. 
This is a notable distinction between the symmetric Mahler conjecture and 
 Conjecture \ref{pMahlerSym} \cite[\S1.3]{BMR23}.
Among all convex bodies, the conjectured minimizers (both for the non-symmetric Mahler Conjecture and for Conjecture \ref{pMahler}) are the $GL(n,\R)$-orbit of the translation of the 
simplex
$   \Delta_{n,0}$
by its $L^p$-Santal\'o point $s_p(\Delta_{n,0})$
\cite[Conjecture 1.4]{BMR23}, where
for $p\in (0,\infty)$ and convex body $K\subset \R^n$, 
there exists a unique point 
$s_p(K)\in\R^n$ 
that lies in the interior of $K$, called the \textit{$L^p$-Santal\'o} point of $K$ with
\cite[Proposition 1.5]{BMR23}
\begin{equation}
\label{LpPointDef}
   \M_p(K-s_p(K))= \inf_{x\in \R^n} \M_p(K-x). 
\end{equation}
These expectations are currently only supported by numerical data in dimension $n=3$ \cite[p. 8]{BMR23}. 
\end{remark}

Our main results are:
\begin{theorem}
\label{2DimpMahlerSym}
Conjecture \ref{pMahlerSym} holds  
when $n=2$. 
\end{theorem}

\begin{theorem}
\label{2DimpMahler}
    Conjecture \ref{pMahler} holds
    when $n=2$.
\end{theorem}

\subsection{Ideas
from the proofs}

The basic strategy for the proof of Theorems \ref{2DimpMahlerSym} and \ref{2DimpMahler} relies on an argument of Mahler dating as far back (1938) as his eponymous conjecture: approximate a body by a polytope and repeatedly slide a vertex to reduce the total number of vertices while not increasing the Mahler volume
(concluding the general case by continuity of $\M$ in
the Hausdorff topology).
This argument was successfully used by Mahler to prove his conjectures in dimension $n=2$ \cite{Mahler39b}, i.e., Conjectures \ref{pMahlerSym} and \ref{pMahler} when $p=\infty$ and $n=2$. 

The case of finite $p$ is significantly more
involved,
ultimately stemming from
 the double integral nature
(alternatively, the `Bergman kernel nature')
of 
$\M_p$ and loss of duality (the double $L^p$-polar is not the original body). In fact, the loss
of duality can be seen as the precise reason
the double integral does not reduce to a single
integral as in the case $p=\infty$.

It is worthwhile to compare our setting
to that of the related, but 
quite  different, setting of $L^p$-centroid
bodies
introduced by Lutwak--Zhang \cite{LZ97}. This body is the symmetric convex body with support function 
\begin{equation}
\label{LpCentroidSupport}
    h_{\Gamma_pK}(y)\defeq \left( \int_K |\langle x,y\rangle|^p \frac{dx}{|K|}\right)^{\frac1p},
\end{equation}
i.e., $\Gamma_pK:= \{y\in \R^n: h_{\Gamma_pK}(y)\leq 1\}^\circ$.
The associated Mahler-type volume is 
$    |K||(\Gamma_p K)^\circ|$.
There are important,
and not just cosmetic,
differences between 
$K^{\circ,p}$ and
$\Gamma_pK$ \cite[Remark 1.7]{BMR23}. Most
fundamentally,
$h_{p,K}$
is not homogeneous while
$h_{\Gamma_pK}$ is.
This lack of homogeneity is
apparent in all stages of developing the theory
of $L^p$-polarity.

Despite these differences, we draw important
inspiration from the work  of
Campi--Gronchi \cite{CampGron06} and Meyer--Reisner \cite{MR06}
who study the volume of the polar bodies of a shadow system
(see also \cite{MR19,EMilmanYehudayoff23}). 
One important idea of the latter is the
``coupled" Mahler sliding alluded to above, which
is crucial for our non-symmetric result (Theorem \ref{2DimpMahler}).

Campi--Gronchi demonstrated that for a shadow system $\{K_t\}_{t\in [0,1]}$ of symmetric convex bodies, the reciprocal of the volume of the polar $t\mapsto |K_t^\circ|^{-1}$ is a convex function \cite[Theorem 1]{CampGron06}. Shadow systems are 1-parameter families of convex bodies that arise as projections from a higher dimensional body. Both Steiner symmetrization (used
in conjuction with $L^p$-polarity in \cite{BMR23}) and coupled Mahler's sliding are examples of shadow systems. Combined with their earlier result that $L^p$-centroid bodies of a shadow system constitute a shadow system \cite[Theorem 2.1]{CampGron02}, they establish sharp lower bounds on the volume of the polar of the $L^p$-centroid body in dimension $n=2$ \cite[Theorem 3]{CampGron06}. This is attained by the simplex in general and by the square among symmetric bodies \cite[Theorem 4]{CampGron06}.  

In their proof, Campi--Gronchi apply the Borell--Brascamp--Lieb inequality to the volume of the polar body of the shadow system, expressed in terms of the support function of the overlying body \cite[p. 2398]{CampGron06}. In the absence of a similar formula for the $L^p$-support functions, we start with \eqref{hpKKcircpEq} and use polar coordinates and an arctan parametrization of the hemispheres of $\partial B_2^2$ (Lemma \ref{suitableVolumeLemma}). An additional layer of complexity in our work stems from the fact that the $L^p$-support function does not coincide with the norm of the $L^p$-polar body. Consequently, we initially establish results at the $L^p$-support level (Claims \ref{TriangleClaim} and \ref{TriangleSliding}, and Lemmas \ref{hpSymCor} and \ref{Px2Slidinghp}), and subsequently, through a theorem of Ball (Theorem \ref{BallIneq}), obtain the desired results on the norms (Lemmas \ref{symNormConvexity} and \ref{normTripleConvexity}).

In the argument of Campi--Gronchi, a notable technical challenge arises in the non-symmetric setting.  
Specifically, the volume of the polar body needs to be split in two parts: the intersection with upper and lower half-spaces,
\begin{equation*}
    |K_t^\circ|= |K_t^\circ\cap \{x_n\leq 0\}| + |K_t^\circ\cap \{x_n\geq 0\}|.
\end{equation*}
The reciprocal of each part is convex in $t$, 
and this suffices in the symmetric case (as the terms are equal).
Meyer--Reisner partly address this hurdle by claiming that the ratio $|K_t^{\circ}\cap \{x_n\leq 0\}|/ |K_t^\circ\cap \{x_n\geq 0\}|$ remains constant in $t$ \cite[Lemmas 4--5]{MR06}. 
It is noteworthy that Meyer--Reisner follow a different argument than Campi--Gronchi by using a different formula for the volume of the polar and, instead of employing the Borell--Brascamp--Lieb inequality, they use Ball's theorem (Theorem \ref{BallIneq}). In the present article we adopt a blend of arguments from both papers combined
with facts on $L^p$-polarity
developed here and in our previous work \cite{MR22,BMR23,Mastr23}.
We employ our own formula for the volume (Lemma \ref{suitableVolumeLemma}), establish convexity through Ball's theorem (Theorem \ref{BallIneq}) and the Borell--Brascamp--Lieb inequality (Lemma \ref{-1concaveLemma}), and, in the non-symmetric case, balance the two integrals through translation of the polytope (Proposition \ref{BalancingProp}). 
Since many elements of the theory of $L^p$-polarity are
not developed yet, we take the opportunity in a few to
prove some basic results on $K^{\circ,p}$ in all dimensions
(see \S\ref{DecompSection}), and this gives,
in passing, an alternative proof 
of the key `balancing' step (described above) in the Meyer--Reisner approach 
(see Remark \ref{MRRemark}).

\subsection{Bourgain Conjectures}

In the final section of this article, we provide a rather elementary 
and hands-on proof of the strong version of Bourgain's conjectures (symmetric and non-symmetric cases) in dimension two using Mahler sliding.
These results are not new, as we explain in a moment.
Yet, since Bourgain's conjecture is of central importance in many areas of mathematics it is perhaps of some interest to have several proofs of it. The proof we give fits well with the theme of this article of using Mahler sliding for a range of ``best constant" extremal problems concerning volume.
Thus, it serves to unify our treatment of $\M, \M_p$ and $\calC$ (see \eqref{CLKEq}),
and we believe it is self-contained and  accessible to readers
from various backgrounds.
Also, it has the nice feature of demonstrating that $\calC$ can be explicitly expressed as a convex quadratic polynomial with respect to the sliding parameter. Also, we verify
the continuity of $\calC$ in the Haussdorf metric (Corollary
\ref{CHcontinuity}), a simple fact that we
could not find in the literature (cf. \cite[p. 91]{CCG99}).
In passing we also derive 
a formula for the isotropic constant for a body consisting of different pieces (Lemma \ref{iso2_lem2}), that might be of some independent interest, in any dimension.

Before describing our approach let us describe earlier proofs.
The earliest results are due to
Bisztriczky--B\"or\"oczky \cite[Theorems 2.1 and 2.3]{BB01} and, independently, by 
Campi--Colesanti--Gronchi \cite[(11)]{CCG99} although neither of them seem to 
mention the relation of their work to Bourgain's conjectures. The former 
determined the extremizers for the volume of the so-called ellipsoid of inertia, and this rather immediately implies both of Bourgain's 
conjectures in $n=2$. The latter resolved the probabilistic conjecture of Sylvester in $n=2$  for non-symmetric bodies.
Meckes then treated the symmetric case 
and was the first to note (without proof) the implication for Bourgain's conjectures
\cite[Corollary 10]{Meckes05}.
For a proof of the standard fact that Sylvester's conjecture implies Bourgain's conjecture see \cite[Theorem 3.5.7]{BGV+14}.
Both Campi--Colesanti--Gronchi and Meckes used in their work
shadow systems (a far-reaching generalization of Mahler sliding) to prove Sylvester's conjecture in dimension two. In fact, analyzing their proofs
reveals that in dimension two they reduce to Mahler sliding.

We describe all of these three works in more detail below, but first
let us state Bourgain's conjectures.
For $u\in \partial B_2^n$, denote by 
\begin{equation*}
    u^\perp\defeq \{x\in \R^n: \langle x,u\rangle=0\}
\end{equation*}
the hyperplane through the origin that is normal to $u$.
Bourgain's hyperplane conjecture states  \cite[Remark p. 1470]{Bourg86a} \cite[(1.9)]{Bourg91}:
\begin{conjecture}
\label{SlicingConjecture}
    There is a universal constant $c>0$ such that for any $n\geq 2$ and convex body $K\subset \R^n$ with $b(K)=0$ and $|K|=1$, 
    \begin{equation*}
        \max_{u\in\partial B_2^{n}} |K\cap u^\perp|\geq c. 
    \end{equation*}
\end{conjecture}
Conjecture \ref{SlicingConjecture} has been reformulated by Ball \cite[(5.7)]{Ball86} and Milman--Pajor \cite[p. 82]{MilmanPajor89}, by linking the volume of a hyperplane section of a convex body to its second moments via 
\begin{equation}
\label{SMsectionEQ}
    \int_{K} \langle x,u\rangle^2 \frac{dx}{|K|}\approx \frac{1}{|K\cap u^\perp|^2}, \quad u\in \partial B_2^n. 
\end{equation}
Additionally, any convex body can be transformed via an affine transformation so that the left-hand side of \eqref{SMsectionEQ} is constant, independent of $u$. If it is also of unit volume and has its barycenter at the origin, it is said to be in \textit{isotropic position}. In essence,
for any convex body $K$ there exists an affine transformation $T$ such that $b(TK)=0$, $|TK|=1$ and 
\begin{equation}
\label{LKDef}
    L_K^2\defeq \int_{TK} \langle x,u\rangle^2  dx \quad  \text{ is constant for all } u\in\partial B_2^n
\end{equation}
\cite[Proposition 2.3.3]{BGV+14}.
It is evident from \eqref{SMsectionEQ} and \eqref{LKDef} that Conjecture \ref{SlicingConjecture} is equivalent to an upper bound on the isotropic constant. 
\begin{conjecture}
\label{IsotropicConjecture}
    There is a universal constant $c>0$ such that for any $n\geq 2$ and any convex body $K\subset \R^n$, $L_K\leq c$. 
\end{conjecture}

Although Conjecture \ref{IsotropicConjecture} remains unresolved, several bounds have been obtained. Specifically, Bourgain demonstrated $L_K\leq Cn^{1/4}\log n$ \cite[Theorem 1.6]{Bourg91}, Klartag established $L_K\leq C n^{1/4}$ \cite[Corollary 1.2]{Klartag06}, and Chen improved it to $L_K\leq C_1 e^{C_2\sqrt{\log(n)}\sqrt{\log\log(3n)}}$ \cite{Chen21} \cite[(1)]{KL22}. Subsequent enhancements by various authors resulted in $L_K\leq C (\log n)^q$ for different values of $q$ \cite{JLV22,,Klartag23,KL22}. 

Another equivalent way of defining the isotropic constant involves the \textit{covariance matrix} of the convex body $K$,
\begin{equation*}
    \mathrm{Cov}(K),
\end{equation*}
an $n$-by-$n$ matrix given by 
\begin{equation*}
    \mathrm{Cov}(K)_{ij}\defeq \int_{K}x_ix_j\frac{dx}{|K|} - \int_{K}x_i\frac{dx}{|K|} \int_{K} x_j\frac{dx}{|K|}. 
\end{equation*}
Define
\begin{equation}
\label{CLKEq}
    \calC(K)\defeq \frac{|K|^2}{\det\mathrm{Cov}(K)}= \frac{1}{L_K^{2n}}.
\end{equation}
\cite[(1.9)]{MilmanPajor89} (see \S \ref{CAffineInvSection} for a proof).
As in \cite{BMR23}, we prefer to work with $\calC$ instead of $L_K$ because it behaves similarly to the Mahler volume $\M$. Namely, it is both affine invariant (Corollary \ref{CaffineInvariance})
\begin{equation*}
    \calC(AK+b)= \calC(K), \quad A\in GL(n,\R), b\in \R^n,
\end{equation*}
and tensorial
\begin{equation}
\label{Ctensor}
    \calC(K\times L)= \calC(K)\,\calC(L). 
\end{equation}
In addition, it is maximized by $B_2^n$ \cite[4.1 Lemma]{MilmanPajor89} \cite[Proposition 3.3.1]{BGV+14}, 
\begin{equation}
\label{CUpperBound}
    \calC(K)\leq \calC(B_2^n), \quad \text{ for all convex bodies } K\subset \R^n.
\end{equation}

A strong version of Conjecture \ref{IsotropicConjecture}, first suggested by Ball \cite[p. 85]{Ball86} \cite[p. 85]{Ball88} for symmetric bodies, and appeared in print in \cite{Meckes05},  states that $L_K$ is maximized by ${\Delta_{n,0}}$ in general, or among symmetric convex bodies, by ${[-1,1]^n}$. Equivalently:
\begin{conjecture}
\label{strongIso}
    For a convex body $K\subset \R^n$, $\calC(K)\geq \calC(\Delta_{n,0})= \frac{(n+2)^n(n+1)^{n+1}}{(n!)^2}$. 
\end{conjecture}
\begin{conjecture}
\label{strongIsoSym}
    For a symmetric convex body $K\subset \R^n$, $\calC(K)\geq \calC([-1,1]^n)=12^n$. 
\end{conjecture}

As a way of unifying the treatment for
$\M, \M_p, \B$ and $\calC$, we show
that 
 Conjectures \ref{strongIso} and \ref{strongIsoSym} in dimension two can also be derived
 by using Mahler's sliding (with no translation necessary).
 \begin{theorem}
\label{plane_slicing_nonsym}
    For a convex body $K\subset \R^2$, $
    \calC(K)\ge \calC(\Delta_{2,0})=108.
    $
\end{theorem}
\begin{theorem}
\label{plane_slicing_sym}
    For a symmetric convex body $K\subset \R^2$, $
    \calC(K)\geq\calC([-1,1]^2)=144.
    $
\end{theorem}

\begin{remark}
Interestingly, Klartag demonstrated that Conjecture \ref{strongIso} implies the non-symmetric Mahler conjecture (Conjecture \ref{pMahler} for $p=\infty$)  \cite[Corollary 1.2]{Klartag18}. An alternative proof was recently provided by Balacheff--Solanes--Tzanev \cite{BST23}. While not giving a converse, using Mahler's sliding
to prove Conjecture \ref{strongIso} indicates that for $n=2$ the conjectures are loosely equivalent. 
\end{remark}

As already described earlier, Theorems \ref{plane_slicing_nonsym}--\ref{plane_slicing_sym}  have been proved through the resolution of related conjectures, and we refer to Gardner \cite[p. 17]{GardnerNotes} \cite{CCG99} for a survey. The first proof, by Bisztriczky--B\"or\"oczky,  
demonstrates that the volume of the \textit{ellipsoid of inertia} (or \textit{Legendre ellipsoid} \cite[\S 1.1]{MilmanPajor89}, or \textit{$L^2$-centroid body})
\begin{equation*}
    \Gamma_2K,
\end{equation*}
that is the ellipsoid
with support function given by \eqref{LpCentroidSupport} for $p=2$,
is maximized among symmetric convex bodies of unit volume by squares \cite[Theorem 2.1]{BB01}, and, in general, by triangles among all convex bodies of unit volume in $\R^2$ \cite[Theorem 2.3]{BB01}. 
Note that, by \eqref{LKDef}, a convex body $K$ is in isotropic position if and only if $h_{\Gamma_2 K}(u)= L_K^2$ is constant for $u\in\partial B_2^n$, i.e., $\Gamma_2 K$ is a ball of radius $L_K^2$. Therefore, for convex bodies in isotropic position $|\Gamma_2K|= L_K^{2n} |B_2^n|= |B_2^n|/\calC(K)$. Hence, it is evident that $\calC(K)$ is minimized by the simplex, or by the cube among symmetric bodies, if and only if $|\Gamma_2K|$ is maximized by the simplex, or by the cube among symmetric convex bodies. This is precisely what Bisztriczky--B\"or\"oczky establish in dimension two \cite[Theorems 2.1 and 2.3]{BB01}. In addition, they characterize the equality cases as parallelograms in the symmetric case, and triangles with the origin as a vertex in general.

Another proof, using shadow systems, was given by Meckes, by resolving the generalized Sylvester conjecture in dimension two for symmetric convex bodies \cite[Corollary 10]{Meckes05}. The non-symmetric case had already been resolved by Campi--Colesanti--Gronchi a few years before \cite[(11)]{CCG99}. This states that the mean value 
\begin{equation*}
    \mathbb{E}\,U_{K,N}
\end{equation*}
of the volume of a random polytope with $N\geq n+1$ vertices uniformly distributed in a convex body $K$, is maximized by the simplex \cite[p. 308]{Meckes05}. Similarly, considering the symmetric analog $V_{N,K}$, where one also chooses the $N$ antipodal points, then $\mathbb{E}\,V_{K, N}$ is conjectured to be maximized by the cube $[-1,1]^n$ among symmetric convex bodies. 
The key link to the slicing problem is an equality associating the isotropic constant of a convex body $K\subset \R^n$ of volume $|K|=1$ essentially due to Kingman \cite[Theorem 1]{Kingman69}, 
\begin{equation*}
    L_K^{2n} =\frac{n!}{n+1} \mathbb{E}\,U_{K,N}= \frac{n!}{4^n} \mathbb{E}\,V_{K, N}, 
\end{equation*}
with the second equality for symmetric convex bodies due to Meckes \cite[(1)]{Meckes04} (also studied by Ball \cite[Lemma 8]{Ball88}). 

For more in-depth information on the slicing problem, we refer to \cite{BGV+14, GardnerNotes, KM21, AmbrusBoroczky14, Rademacher16, MilmanPajor89}. 

\text{ }

\noindent
\textbf{Organization.}
In \S \ref{LpPolarVolumeSection}, we derive a suitable formula for the volume of the $L^p$-polar (Lemma \ref{suitableVolumeLemma}).
Sections \ref{SymSlidingSection} and \ref{NonSymSlidingSection}, introduce the classes of polytopes most convenient for our analysis (Figures \ref{SymSlidingFig} and \ref{SlidingFig}), and the main propositions on which the proofs of Theorems \ref{2DimpMahlerSym} and \ref{2DimpMahler} are based upon.

Section \ref{MpSymSection} is dedicated to proving Theorem \ref{2DimpMahlerSym}. A technical part of the proof concerning the $L^p$-support function of triangles under sliding is deferred until \S \ref{SymTriangleSection}.
`Convexity' of the $L^p$-support function under sliding for triangles is proved in \S \ref{SymTriangleSection} (Claim \ref{TriangleClaim}). This is then generalized for symmetric polytopes in \S \ref{SymEntirePolytopeSection} (Lemma \ref{hpSymCor}), where the convexity of the norm of a symmetric convex polytope under sliding is also established
(Lemma \ref{symNormConvexity}).
Theorem \ref{2DimpMahlerSym} is proved in \S \ref{SymFinishingSection}. 
\ref{BallIneq}). 

In Section \ref{MpNonSymSection}, we prove Theorem \ref{2DimpMahler}. Due to the necessity of translating the polytope while sliding a vertex, we revisit triangles in \S \ref{TrianglesRevisitSection}, to incorporate the translating parameter (Claim \ref{TriangleSliding}). Convexity of the near norm of a polytope under sliding and translation is established in \S \ref{ConvNormSection} (Lemma \ref{normTripleConvexity}). In the absence of symmetry, the volume of the $L^p$-polar is expressed as the sum of two integrals (see \eqref{volumeLpPolar}). For our proof to work, the ratio of the two must remain constant under sliding. This is achieved in \S \ref{DecompSection} by carefully translating the body while sliding a vertex (Proposition \ref{BalancingProp}). Theorem \ref{2DimpMahler} is derived in \S \ref{nonSymFinishSection}. 

Theorems \ref{plane_slicing_nonsym} and \ref{plane_slicing_sym} are established in Section \ref{BourgainSection}. \S \ref{CAffineInvSection} is dedicated to the affine invariance of $\calC$ (Corollary \ref{CaffineInvariance}). The continuity of $\calC$ in the Hausdorff topology is established in \S \ref{CHausdorffSection}. In \S\ref{UnionSection}, we study the isotropic constant of a disjoint union (Lemma \ref{iso2_lem1}), and derive a convenient formula in dimension two (Corollary \ref{2dCFormula}). Following the pattern of the previous sections, \S \ref{Ctriangle} is dedicated to the study of the isotropic constant of triangles under sliding. Building on that, in \S \ref{isoPolytopeSlidingSection}, we investigate the isotropic constant of polytopes under sliding, whether symmetric (Lemma \ref{CsymConvexity}) or not (Lemma \ref{Cconvexity}). We conclude the proof of Theorems \ref{plane_slicing_nonsym} and \ref{plane_slicing_sym} in \S \ref{isoFinishingSection}.

In Appendix \ref{appendixA}, we revisit the Borell--Brascamp--Lieb inequality (Lemma \ref{BrascampLiebineq}) and its corollaries (Lemma \ref{-1concaveLemma} and Corollary \ref{BLcorollary}). Finally, in Appendix \ref{BallSection} we state a theorem of Ball (Theorem \ref{BallIneq}) that will be very useful throughout. 

\section{\texorpdfstring{A formula for the volume of planar $L^p$-polars}{A formula for the volume of planar Lp-polars}}
\label{LpPolarVolumeSection}

For the proof of Lemma \ref{SymSlidingLemma}, and later for its non-symmetric counterpart (Lemma \ref{SlidingLemma}), we derive a convenient formula for the volume of the $L^p$-polar.  
\begin{lemma}
\label{suitableVolumeLemma}
    Let $p>0$. For a convex body $K\subset \R^2$ with $0\in\mathrm{int}\,K$, 
    \begin{equation}
    \label{volumeLpPolar}
        |K^{\circ,p}|= \frac12\int_{\R}\frac{dt}{\|(1,t)\|^2_{K^{\circ,p}}}+ \frac12\int_{\R} \frac{dt}{\|(-1,-t)\|^2_{K^{\circ,p}}}. 
    \end{equation}
    If $K$ is symmetric, 
    \begin{equation}
    \label{volumeLpPolarSym}
        |K^{\circ,p}|= \int_{\R}\frac{dt}{\|(1,t)\|^2_{K^{\circ,p}}}. 
    \end{equation}
\end{lemma}

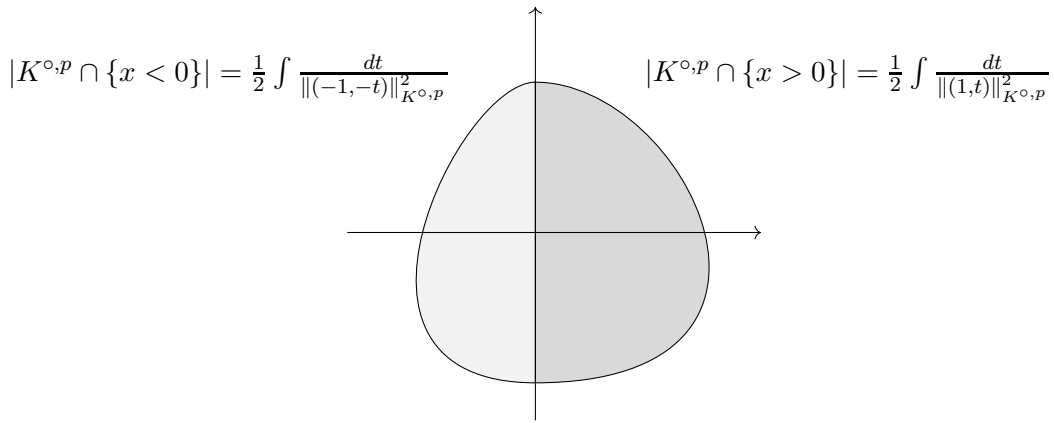
\begin{figure}[H]
    \centering
    \begin{tikzpicture}
\draw[fill=gray!30] 
    (0,2) .. controls +(0,0) and + (0,0) .. 
    (0,-2) .. controls +(4,0) and + (2,0) .. 
          cycle;
\draw[fill=gray!10] 
    (0,2) .. controls +(0,0) and + (0,4) .. 
    (0,-2) .. controls +(-3,0) and + (-1,0) .. 
          cycle;
\draw[->] (-2.5,0) -- (3,0) node[anchor=north west] {};
\draw[->] (0,-2.5) -- (0,3) node[anchor=north west] {};
\draw[] (7,2.5) node[below left]{$|K^{\circ,p}\cap\{x> 0\}|= \frac12\int\frac{dt}{\|(1,t)\|_{K^{\circ,p}}^2}$};
\draw[] (-1, 2.5) node[below left]{$|K^{\circ,p}\cap\{x< 0\}|= \frac12\int\frac{dt}{\|(-1,-t)\|_{K^{\circ,p}}^2}$};
\end{tikzpicture}
\caption{$|K^{\circ,p}|= |K^{\circ,p}\cap\{x\geq 0\}| + |K^{\circ,p}\cap \{x\leq 0\}|.$ Observe these two areas
are equal in the symmetric case.}
\end{figure}

\begin{proof}
    Starting with \eqref{MpDefEq}, in polar coordinates,
    \begin{equation*}
    \begin{aligned}
        |K^{\circ,p}|&= \frac12\int_{\R^2} e^{-h_{p,K}(x,y)} dx dy= \frac12\int_{\partial B_2^2} \left(\int_{0}^\infty e^{-h_{p,K}(ru)}r dr \right) du= \frac12\int_{\partial B_2^2}\frac{du}{\|u\|^2_{K^{\circ,p}}}, 
    \end{aligned}
    \end{equation*}
    by definition of the near norm \eqref{KcircpNormEq}. We may now split $\partial B_2^2$ into two hemispheres: 
    \begin{equation}
    \label{KcircpSplitInt}
        |K^{\circ,p}|= \frac12\int_{\partial B_2^2\cap \{x>0\}}\frac{du}{\|u\|^2_{K^{\circ,p}}}+ \frac12\int_{\partial B_2^2\cap \{x<0\}} \frac{du}{\|u\|^2_{K^{\circ,p}}}.
    \end{equation}
Such a formula holds in all dimensions (replacing 
$\|u\|^2_{K^{\circ,p}}$ with $\|u\|^n_{K^{\circ,p}}$). 
    Next, we will change variables twice, making use of $n=2$. First, parameterizing $\partial B_2^n\cap \{x>0\}$ via
    \begin{equation*}
        (-\pi/2, \pi/2) \ni \theta \mapsto (\cos\theta, \sin\theta) \in \partial B_2^n\cap \{x>0\},
    \end{equation*}
    gives
    \begin{equation*}
        \int_{\partial B_2^{n}\cap \{x>0\}} \frac{du}{\|u\|_{K^{\circ,p}}^2}= \int_{-\pi/2}^{\pi/2}\frac{d\theta}{\|(\cos\theta, \sin\theta)\|_{K^{\circ,p}}^2}.
    \end{equation*}
    Since for $\theta\in (-\pi/2, \pi/2)$, $\cos\theta>0$, and $\|\cdot\|_{K^{\circ,p}}$ is positively homogeneous
    $
        \|(\cos\theta, \sin\theta)\|_{K^{\circ,p}}
        = \cos\theta \|(1, \tan\theta)\|_{K^{\circ,p}}. 
    $
    Therefore, using the change of variables $t= \tan\theta$, $dt= \frac{d\theta}{(\cos\theta)^2}$, 
    \begin{equation}
    \label{KcircInt1}
        \int_{{\partial B_2^n}\cap \{x>0\}}\frac{du}{\|u\|^2_{K^{\circ,p}}}= \int_{-\pi/2}^{\pi/2} \frac{d\theta}{(\cos\theta)^2 \|(\cos\theta, \sin\theta)\|^2_{K^{\circ,p}}}= \int_{\R}\frac{dt}{\|(1,t)\|^2_{K^{\circ,p}}}.
    \end{equation}
    Similarly, since for $\theta\in (\pi/2, 3\pi/2)$, $\cos\theta<0$, and $\|\cdot\|_{K^{\circ,p}}$ is (only) positively homogeneous, $\|(\cos\theta,\sin\theta)\|_{K^{\circ,p}}= -\cos\theta \|(-1, -\tan\theta)\|_{K^{\circ,p}}$,
    \begin{equation}
    \label{KcircpInt2}
        \int_{\partial B_2^n\cap \{x<0\}}\frac{du}{\|u\|_{K^{\circ,p}}^2}= \int_{\pi/2}^{3\pi/2}\frac{d\theta}{(\cos\theta)^2\|(-1, -\tan\theta)\|_{K^{\circ,p}}^2}= \int_{\R} \frac{dt}{\|(-1,-t)\|_{K^{\circ,p}}^2}. 
    \end{equation}
    Combining \eqref{KcircpSplitInt}, \eqref{KcircInt1}, and \eqref{KcircpInt2} proves \eqref{volumeLpPolar}.

    If $K$ is symmetric, then so is $K^{\circ,p}$ \cite[Theorem 1.2]{BMR23}, hence $\|\cdot\|_{K^{\circ,p}}$ is a norm. Therefore, $\|(-1, -t)\|_{K^{\circ,p}}= \|(1,t)\|_{K^{\circ,p}}$ and \eqref{volumeLpPolarSym} follows from \eqref{volumeLpPolar}. 
\end{proof}

Campi--Gronchi \cite[p. 2398]{CampGron06}
used different parametrizations in
their computation of  $|\Gamma_pK|$. In dimension two both our parametrization and
theirs are equivalent via a two-step reparametrization (from the hemisphere to a disk, and from a disk to a plane).

\section{Coupled Mahler sliding}

In \S\ref{SymSlidingSection}
we describe classical Mahler sliding, with Proposition 
\ref{symSlidingProp} stating that it decreases $L^p$-Mahler volume for symmetric
polytopes.
In \S\ref{NonSymSlidingSection}
we describe Mahler sliding coupled with translation, with Proposition 
\ref{SlidingProp} the analogous result for general
polytopes.

\subsection{Symmetric Mahler sliding}
\label{SymSlidingSection}

As described in the introduction, the aim is to start with a symmetric polytope and `slide' two antipodal vertices
(see Figure \ref{SymSlidingFig})
in a manner that throughout the motion

\begin{itemize}
\item retains the symmetry of the polytope, 
\item leaves the area invariant,
\end{itemize}
and, in addition, ultimately (but perhaps not throughout) 
\begin{itemize}
\item decreases the number of vertices by two. 
\item reduces the $L^p$-Mahler volume, 
\end{itemize}

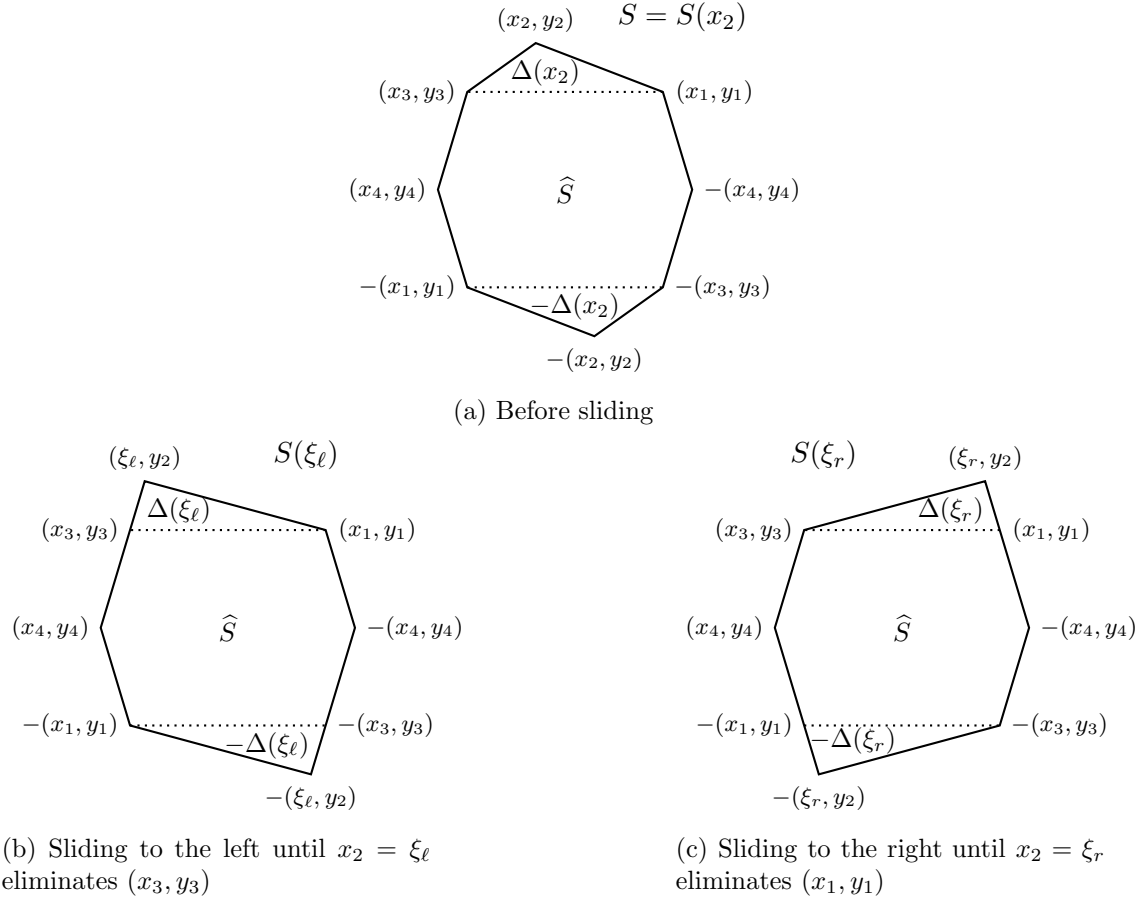
\begin{figure}[ht]
     \centering
     \begin{subfigure}[b]{0.35\textwidth}
         \centering
         \begin{tikzpicture}[scale=1.3]
\draw[thick] (-1,1) -- (-1.3, 0) -- (-1,-1) -- (.3,-1.5) -- (1,-1) -- (1.3, 0) -- (1, 1) -- (-.3,1.5) -- cycle;
\draw[thick, dotted] (1,1) -- (-1,1);
\draw[thick, dotted] (-1,-1) -- (1,-1);
\node[right=1pt] at (1, 1){\footnotesize $(x_1, y_1)$};
\node[above=.5pt] at (-.3, 1.5){\footnotesize $(x_2,y_2)$};
\node[above=.5pt] at (1.2, 1.5){$S= S(x_2)$};
\node[] at (-.2, 1.2){\small $\Delta(x_2)$};
\node[] at (0, 0){\small $\widehat{S}$};
\node[] at (.1, -1.2){\small $-\Delta(x_2)$};
\node[left=.5pt] at (-1, 1){\footnotesize $(x_3, y_3)$};
\node[left=.5pt] at (-1.3, 0){\footnotesize $(x_4, y_4)$};
\node[left=.5pt] at (-1, -1){\footnotesize $-(x_1, y_1)$};
\node[below=.5pt] at (.3, -1.5){\footnotesize $-(x_2, y_2)$};
\node[right=.5pt] at (1, -1){\footnotesize $-(x_3, y_3)$};
\node[right=.5pt] at (1.3, 0){\footnotesize $-(x_4, y_4)$};
\end{tikzpicture}
         \caption{\small Before sliding}
     \end{subfigure}
     \\
    \begin{subfigure}[b]{0.35\textwidth}
    \centering
\begin{tikzpicture}[scale=1.3]
\draw[thick] (-1,1) -- (-1.3, 0) -- (-1,-1) -- (.85, -1.5) -- (1,-1) -- (1.3, 0) -- (1, 1) -- (-.85,1.5) -- cycle;
\draw[thick, dotted] (1,1) -- (-1,1);
\draw[thick, dotted] (-1,-1) -- (1,-1);
\node[above=.5pt] at (.8, 1.5){$S(\xi_\ell)$};
\node[] at (-.5, 1.2){\small $\Delta(\xi_\ell)$};
\node[] at (0, 0){\small $\widehat{S}$};
\node[] at (.4, -1.2){\small $-\Delta(\xi_\ell)$};
\node[right=1pt] at (1, 1){\footnotesize $(x_1, y_1)$};
\node[right=.5pt] at (1, -1){\footnotesize $-(x_3, y_3)$};
\node[left=.5pt] at (-1, 1){\footnotesize $(x_3, y_3)$};
\node[above=.5pt] at (-.85, 1.5){\footnotesize $(\xi_\ell, y_2)$};
\node[left=.5pt] at (-1.3, 0){\footnotesize $(x_4, y_4)$};
\node[left=.5pt] at (-1, -1){\footnotesize $-(x_1, y_1)$};
\node[below=.5pt] at (.85, -1.5){\footnotesize $-(\xi_\ell, y_2)$};
\node[right=.5pt] at (1.3, 0){\footnotesize $-(x_4, y_4)$};
\end{tikzpicture}
    \caption{\small Sliding to the left until $x_2=\xi_\ell$ eliminates $(x_3, y_3)$}
     \end{subfigure}
     \hspace{3cm}
              \begin{subfigure}[b]{0.35\textwidth}
         \centering
\begin{tikzpicture}[scale=1.3]
\draw[thick] (-1,1) -- (-1.3, 0) -- (-1,-1) -- (-.85, -1.5) -- (1,-1) -- (1.3, 0) -- (1, 1) -- (.85,1.5) -- cycle;
\draw[thick, dotted] (1,1) -- (-1,1);
\draw[thick, dotted] (-1,-1) -- (1,-1);
\node[above=.5pt] at (-.8, 1.5){$S(\xi_r)$};
\node[] at (.5, 1.2){\small $\Delta(\xi_r)$};
\node[] at (0, 0){\small $\widehat{S}$};
\node[] at (-.5, -1.15){\small $-\Delta(\xi_r)$};
\node[above=.5pt] at (.85, 1.5){\footnotesize $(\xi_r, y_2)$};
\node[right=1pt] at (1, 1){\footnotesize $(x_1, y_1)$};
\node[left=.5pt] at (-1, -1){\footnotesize $-(x_1, y_1)$};
\node[left=.5pt] at (-1, 1){\footnotesize $(x_3, y_3)$};
\node[right=.5pt] at (1, -1){\footnotesize $-(x_3, y_3)$};
\node[left=.5pt] at (-1.3, 0){\footnotesize $(x_4, y_4)$};
\node[below=.5pt] at (-.85, -1.5){\footnotesize $-(\xi_r, y_2)$};
\node[right=.5pt] at (1.3, 0){\footnotesize $-(x_4, y_4)$};
\end{tikzpicture}
    \caption{\small Sliding to the right until $x_2=\xi_r$ eliminates
         $(x_1, y_1)$}
     \end{subfigure}
        \caption{\small Symmetric Mahler sliding eliminating two antipodal vertices simultaneously.}
        \label{SymSlidingFig}
\end{figure}

Before proceeding, let us fix some notation. 
Let 
\begin{equation*}
    S\defeq\co\{(x_1, y_1), \ldots, (x_m, y_m), -(x_1, y_1), \ldots, -(x_m, y_m)\}
\end{equation*}
denote a symmetric polytope with $2m$ vertices.
After rotating $S$, and possibly renaming the vertices, we can assume that $(x_1, y_1),\ldots, (x_m, y_m)$ are counter-clockwise oriented, $y_2>0$, and $y_1= y_3$. Sliding the $p_2$ vertex in a manner parallel to $[(x_1, y_1), (x_3, y_3)]$ amounts to changing the $x_2$ coordinate, as $[(x_1, y_1), (x_3, y_3)]$ is parallel to $x$-axis. Write $S(x_2)$ to emphasize the dependence on $x_2$. 
Decompose $S(x_2)$ into the triangles 
\begin{equation}
\label{Deltax2Def}
    \Delta(x_2) \defeq \co \{(x_1, y_1), (x_2, y_2), (x_3, y_1)\}, 
\end{equation}
and $-\Delta(x_2)$,
and the `sandwich' part
\begin{equation}
\label{hatS}
    \widehat{S}\defeq 
    \co\{
    (x_1, y_1), (x_3, y_3), \cdots, (x_m, y_m),
    -(x_1, y_1), -(x_3, y_3), \cdots, -(x_m, y_m)
    \},
\end{equation}

Sliding $(x_2, y_2)$ in a manner parallel to $[(x_1,y_1), (x_3,y_3)]$,
i.e., only changing $x_2$ while fixing $y_2$ does not change the volume as 
\begin{equation}
\label{Sx2constEq}
|S(x_2)|= |\widehat{S}|+ 2|\Delta(x_2)|
\end{equation}
is {\it constant
in $x_2$}.
To emphasize the sole dependence of the `sliding' on the value of $x_2$, write
\begin{equation*}
   x_2\mapsto S(x_2)
\end{equation*}
for the one-parameter family of symmetric polytopes.
For this family it is clear that the first three bullet points hold.
One of our main results is that also the fourth bullet point holds:

\begin{proposition}
\label{symSlidingProp}
    Let $p>0$, and $S\subset \R^2$ be a symmetric polytope with $2m>4$ vertices. There is a symmetric polytope $S'$ with $2m-2$ vertices obtained from $S$ by Mahler sliding such that
    \begin{equation*}
        \M_p(S')\leq \M_p(S).
    \end{equation*}
\end{proposition}

\subsection{Coupled Mahler sliding}
\label{NonSymSlidingSection}

For a general polytope $P$, we wish to perform a similar operation that
does not change the area of the polytope and ultimately results in a
polytope with one less vertex and small $L^p$-Mahler volume.
It turns out that Mahler sliding is not sufficient in this case, but
a simultaneous (carefully chosen) translation of the whole polytope is also required.
We use the same notation as in \S\ref{SymSlidingSection}. For
$P\in\mathcal{P}$, write
\begin{equation*}
    P(x_2)
\end{equation*}
to emphasize the dependence of the sliding (without translation) on the $x_2$ variable. As before, $P(x_2)$ decomposes into a moving part, 
a triangle $\Delta(x_2)$, and a fixed part 
\begin{equation}
\label{hatP}
    \widehat{P}\defeq \co\{(x_1, y_1), (x_3, y_3), \cdots, (x_m, y_m)\},
\end{equation}
so that $|P(x_2)|= |\widehat{P}|+ |\Delta(x_2)|$ is {\it constant
in $x_2$}.

\begin{figure}[H]
\centering
\begin{subfigure}[b]{0.35\textwidth}
\centering
\begin{tikzpicture}[scale=1.3]
\draw[thick] (-1,1) -- (-1.5, 0) -- (-1,-1) -- (1,-1.2) -- (1.3, 0.3) -- (1, 1) -- (-.3,1.5) -- cycle;
\draw[thick, dotted] (1,1) -- (-1,1);
\node[above=.5pt] at (1.2, 1.5){$P= P(x_2)$};
\node[] at (-.2, 1.2){\small $\Delta(x_2)$};
\node[] at (0, .2){\small $\widehat{P}$};
\node[right=1pt] at (1, 1){\footnotesize $(x_1, y_1)$};
\node[above=.5pt] at (-.3, 1.5){\footnotesize $(x_2, y_2)$};
\node[left=.5pt] at (-1, 1){\footnotesize $(x_3, y_3)$};
\node[left=.5pt] at (-1.5, 0){\footnotesize $(x_4, y_4)$};
\node[left=.5pt] at (-1, -1){\footnotesize $(x_5, y_5)$};
\node[right=.5pt] at (1, -1.2){\footnotesize $(x_6, y_6)$};
\node[right=.5pt] at (1.3, 0.3){\footnotesize $(x_7, y_7)$};
\end{tikzpicture}
\caption{\small Before the sliding}
\end{subfigure}
\\
\begin{subfigure}[b]{0.35\textwidth}
\centering
\begin{tikzpicture}[scale=1.3]
\draw[thick] (-1,1) -- (-1.5, 0) -- (-1,-1) -- (1,-1.2) -- (1.3, 0.3) -- (1, 1) -- (-.75,1.5) -- cycle;
\draw[thick, dotted] (1,1) -- (-1,1);
\node[above=.5pt] at (.8, 1.5){$P(\xi_\ell)$};
\node[] at (-.55, 1.2){\small $\Delta(\xi_\ell)$};
\node[] at (0, .2){\small $\widehat{P}$};
\node[right=1pt] at (1, 1){\footnotesize $(x_1, y_1)$};
\node[above=.5pt] at (-.85, 1.5){\footnotesize $(\xi_\ell, y_2)$};
\node[left=.5pt] at (-1, 1){\footnotesize $(x_3, y_3)$};
\node[left=.5pt] at (-1.5, 0){\footnotesize $(x_4, y_4)$};
\node[left=.5pt] at (-1, -1){\footnotesize $(x_5, y_5)$};
\node[right=.5pt] at (1, -1.2){\footnotesize $(x_6, y_6)$};
\node[right=.5pt] at (1.3, 0.3){\footnotesize $(x_7, y_7)$};
\end{tikzpicture}
\caption{\small Sliding to the left, eliminating $(x_3, y_3)$}
\end{subfigure}
\hspace{2.5cm}
\begin{subfigure}[b]{0.35\textwidth}
\centering
\begin{tikzpicture}[scale=1.3]
\draw[thick] (-1,1) -- (-1.7, 0) -- (-1,-1) -- (1,-1.2) -- (1.3, 0.3) -- (1, 1) -- (.8,1.5) -- cycle;
\draw[thick, dotted] (1,1) -- (-1,1);
\node[above=.5pt] at (-.8, 1.5){$P(\xi_r)$};
\node[] at (.5, 1.2){\small $\Delta(\xi_r)$};
\node[] at (0, .2){\small $\widehat{P}$};
\node[right= 2pt] at (.9, 1){\footnotesize $(x_1, y_1)$};
\node[above=.5pt] at (.85, 1.5){\footnotesize $(\xi_r, y_2)$};
\node[left=.5pt] at (-1, 1){\footnotesize $(x_3, y_3)$};
\node[left=.5pt] at (-1.7, 0){\footnotesize $(x_4, y_4)$};
\node[left=.5pt] at (-1, -1){\footnotesize $(x_5, y_5)$};
\node[right=.5pt] at (1, -1.2){\footnotesize $(x_6, y_6)$};
\node[right=.5pt] at (1.3, 0.3){\footnotesize $(x_7, y_7)$};
\end{tikzpicture}
\caption{\small Sliding to the right, eliminating $(x_1, y_1)$}
\end{subfigure}
    \caption{\small Mahler's sliding eliminating a vertex}
    \label{SlidingFig}
\end{figure}
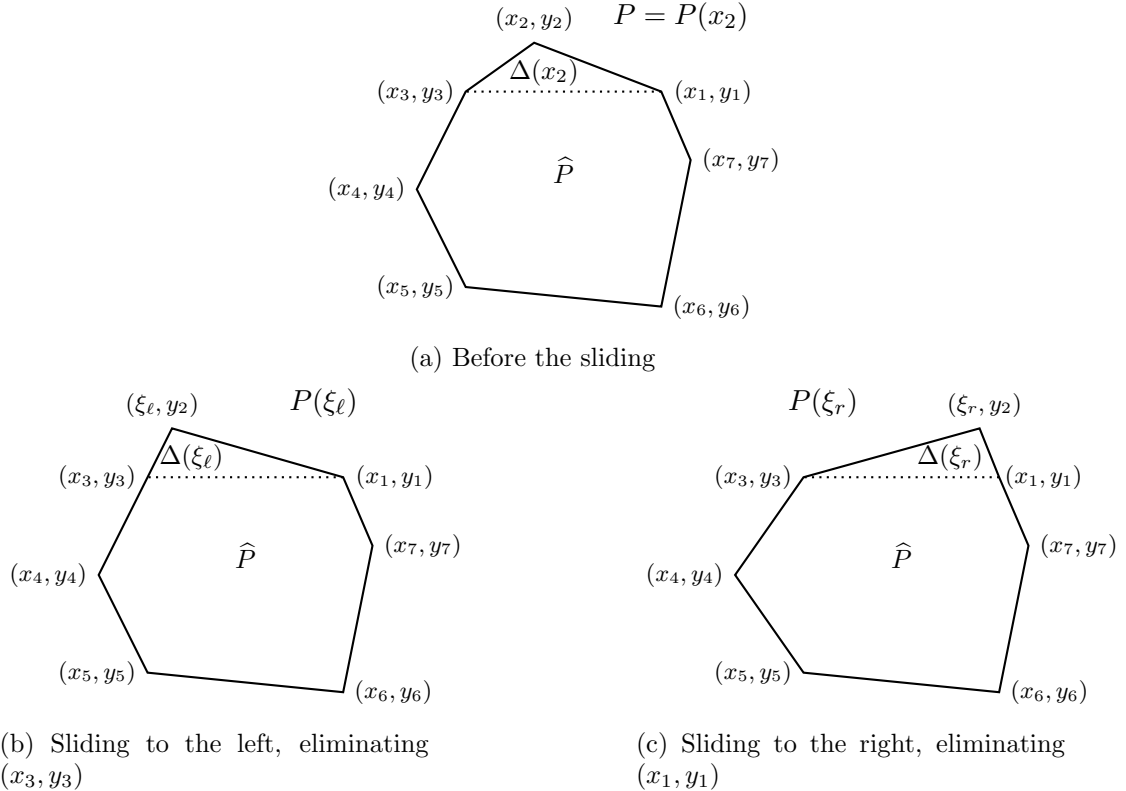

The following generalizes Proposition 
\ref{symSlidingProp}. In comparison to the symmetric case, the non-symmetric case requires a combination of translation and sliding (Figures
\ref{SlidingFig}--\ref{translatingFigure}). 

\begin{proposition}
\label{SlidingProp}
    Let $p>0$, and $P\subset \R^2$ be a polytope with $m>3$ vertices. There is a
    vector $(x_P,y_P)\in\R^2$ and a polytope $P'$ with $m-1$ vertices 
    obtained from $P$ by Mahler sliding
    such that
    \begin{equation*}
        \M_p(P'-(x_P,y_P))\leq \M_p(P).
    \end{equation*}
    In particular,
    \begin{equation*}
        \inf_{P_{m-1}}\inf_{(x,y)\in\R^2} \M_p(P_{m-1}-(x,y))\leq \inf_{P_m}\inf_{(x,y)\in\R^2}\M_p(P_{m}-(x,y)),
    \end{equation*}
    where $P_{m-1}$ and $P_m$ range over all polytopes with $m-1$ and $m$ vertices, respectively.
\end{proposition}
The translation $(x_P, y_P)$ in Proposition \ref{SlidingProp} depends only on the initial polytope
$P$ and the final value of the sliding parameter ($\xi_\ell$ or
$\xi_r$, that in turn are
also determined by $P$). It is selected to maintain a constant ratio of volumes for the {\it $L^p$-polar} of the polytope in the two half-spaces at the starting and ending positions (Proposition \ref{BalancingProp}). 
One can also think
of the sliding process
as a continuous
one-parameter family of polytopes and the translation 
$(x_P(t),y_P(t))$
as a continuous curve
with $(x_P(x_2),y_P(x_2))=(0,0)$
chosen so that the aforementioned ratio of volumes is constant. 
This is unnecessary for symmetric polytopes, since then the $L^p$-polar is symmetric and the ratio
is always constant equal to 1
(see the end of the proof
of Lemma \ref{suitableVolumeLemma}).

\section{\texorpdfstring{Lower bound for $\M_p$ with symmetry}{Lower bound for Mp with symmetry }}
\label{MpSymSection}

In this section we prove Theorem \ref{2DimpMahlerSym}
using Mahler sliding. Following the discussion in \S\ref{SymSlidingSection}
it suffices to restrict attention to symmetric polytopes 
in the class
\begin{equation}
\label{SClassEq}
\mathcal{S}=
\left\{
S= \co\{(x_1, y_1), \ldots, (x_m, y_m), -(x_1,y_1),\ldots,  -(x_m, y_m) \}
\;\middle|\;
\begin{alignedat}{2}
  & m\in\mathbb{N} \\
  & (x_i, y_i) \text{ are extreme points} \\
  & \text{counter-clockwise oriented}\\
  & x_3\leq x_2\leq x_1 \\
  & 0\leq y_1= y_3 < y_2
\end{alignedat}
\right\}.
\end{equation}

Proposition \ref{symSlidingProp} is an immediate corollary of the next lemma 
whose proof, in turn, occupies most of this section
(concluded in \S \ref{SymFinishingSection}). 
\begin{lemma}
\label{SymSlidingLemma}
    Let $p\in (0,\infty]$. For $S\in\mathcal{S}$, 
     $   x_2 \mapsto {|S(x_2)^{\circ,p}|}^{-1}
    $
    is convex.
\end{lemma}

\begin{proof}[Proof of Proposition \ref{symSlidingProp}]
    $\M_p$ is invariant under the action of $GL(2,\R)$ \cite[Lemma 4.6]{BMR23}, 
    and, in particular, under rotation.  We may therefore assume that $S\in\mathcal{S}$. Let $\xi_\ell, \xi_r\in \R$, be as in Figure \ref{SymSlidingFig},
    so there is a unique $\lambda\in (0,1)$ with $x_2= (1-\lambda)\xi_\ell+ \lambda\xi_r$. By Lemma \ref{SymSlidingLemma}, $x_2\mapsto |S(x_2)^{\circ,p}|^{-1}$ is convex, hence
    \begin{equation*}
        \frac{1}{|S(x_2)^{\circ,p}|}\leq \frac{1-\lambda}{|S(\xi_\ell)^{\circ, p}|}+ \frac{\lambda}{|S(\xi_r)^{\circ,p}|}\leq \max\left\{\frac{1}{|S(\xi_\ell)^{\circ,p}|}, \frac{1}{|S(\xi_r)^{\circ,p}|}\right\}, 
    \end{equation*}
    or $|S(x_2)^{\circ,p}|\geq \min\{|S(\xi_r)^{\circ,p}|, |S(\xi_r)^{\circ,p}|\}$. This concludes the proof as $S(\xi_\ell)$ and $S(\xi_r)$ are symmetric polytopes with $2m-2$ vertices.
    \end{proof}

\noindent
\begin{proof}[Proof of Lemma \ref{SymSlidingLemma}]
A well-known corollary (Lemma \ref{-1concaveLemma}) of the Borell--Brascamp--Lieb inequality
is that the one-variable function
$x_2\mapsto \bigg({\int_{\R}\frac{\dif y}{f(y,x_2)^2}}\bigg)^{-1}$
is convex if the two-variable function $(y,x_2)\mapsto f(y,x_2)$ is convex.
Thus, Lemma \ref{SymSlidingLemma} follows from
Lemma \ref{suitableVolumeLemma} 
combined with the crucial Lemma \ref{symNormConvexity} for $x=1$ below.
\end{proof}

\begin{lemma}
\label{symNormConvexity}
For $p>0$, $S\in \mathcal{S}$ and $x\in \R$,
$(y,x_2)\mapsto \|(x,y)\|_{S(x_2)^{\circ,p}}$
is convex. 
\end{lemma}

\noindent
The proof of Lemma \ref{symNormConvexity} occupies the rest of this section.
It is a joint convexity property for norms of $L^p$-polars (varying in a
specific 1-parameter family, coming from Mahler sliding).
We lean on an idea
used in \cite{BMR23}:
according to a classical theorem of Ball (Theorem \ref{BallIneq}), convexity-type
results for norms of the $L^p$-polars can be obtained by verifying a pointwise inequality on the $L^p$-support functions. This is first done for triangles (Claim \ref{TriangleClaim}), and then for $S(x_2)$ by decomposing it (Figure \ref{SymSlidingFig})
into its moving antipodal triangles and non-moving part (Lemma \ref{hpSymCor}).  

\subsection{\texorpdfstring{$L^p$-support function: reduction to the sliding triangles}{Lp-support function: reduction to the moving triangles}}
\label{LpsupportReductionSection}

It will be necessary to decompose the polytope into its moving and non-moving parts. Therefore, we require an understanding of how the $L^p$-support function behaves under disjoint unions. 
\begin{claim}
\label{hpDecomp}
    Let $p>0$ and $K\subset \R^n$ be a convex body. For compact bodies (possibly not connected) $L, C\subset K$ with $K= L\cup C$ and $\mathrm{int}\,L\cap \mathrm{int}\,C= \emptyset$, 
    \begin{equation*}
        e^{p h_{p,K}(y)}= \frac{|L|}{|K|} e^{p h_{p,L}(y)}+ \frac{|C|}{|K|} e^{ph_{p, C}(y)}. 
    \end{equation*}
\end{claim}
\begin{proof}
By definition, 
\begin{equation*}
    \begin{aligned}
        e^{ph_{p,K}(y)}&= \int_{K}e^{p\langle x,y\rangle}\frac{dx}{|K|} \\
        &= \int_L e^{p\langle x,y\rangle}\frac{dx}{|K|} + \int_C e^{p\langle x,y\rangle}\frac{dx}{|K|} \\
        &= \frac{|L|}{|K|} \int_{L} e^{p\langle x,y\rangle}\frac{dx}{|L|} + \frac{|C|}{|K|}\int_C e^{p\langle x,y\rangle} \frac{dx}{|C|} \\
        &= \frac{|L|}{|K|} e^{p h_{p, L}(y)}+ \frac{|C|}{|L|} e^{p h_{p, C}(y)}. 
    \end{aligned}
\end{equation*}
\end{proof}

Repeated applications of Claim \ref{hpDecomp} give the following:
\begin{corollary}
\label{hpMultipleDecomp}
        Let $p>0$ and $K\subset \R^n$ be a convex body. For 
$m\in\mathbb{N}$ and compact bodies (possibly not connected) $L_1, \ldots, L_m\subset K$ with $K= L_1\cup L_2\cup\ldots\cup L_m$ and $\mathrm{int}\,L_i\cap \mathrm{int}\,L_j= \emptyset$ for all $i \neq j$, 
    \begin{equation*}
        e^{p h_{p,K}(y)}= \frac{|L_1|}{|K|} e^{p h_{p,L_1}(y)}+ \frac{|L_2|}{|K|} e^{ph_{p, L_2}(y)}+ \cdots + \frac{|L_m|}{|K|} e^{p h_{p, L_m}(y)}. 
    \end{equation*}
\end{corollary}

\subsection{\texorpdfstring{Joint `convexity' of the $L^p$-support for sliding triangles}{Joint convexity of the Lp-support for sliding triangles}}
\label{SymTriangleSection}

The following `convexity' property is precisely what is needed for the application of Ball's theorem (Theorem \ref{BallIneq}). 
 It is not the same as the convexity of
 \begin{equation}
 \label{x2yhp}
 (x_2,y)\mapsto h_{p,\Delta(x_2)}(x,y)
 \end{equation} 
 but rather a {\it stronger} property (put $r=s=t=1$ below to get convexity of \eqref{x2yhp}) that
 leads, via Ball's theorem, to (joint in the sliding parameter) convexity of the associated norm, as stated in Lemma
\ref{symNormConvexity}. 
\begin{claim}
\label{TriangleClaim}
    Let $p>0$ and let 
    $\Delta(x_2)$ be the triangle as in \eqref{Deltax2Def}.
    For $r,t, s>0$ with $\frac{2}{r}= \frac{1}{t}+ \frac{1}{s}$, 
\begin{equation}
\label{Claim4.5Eq}
\begin{aligned}
    &h_{p,\Delta((x_2+x_2')/2)}\left(r\Big(x, \frac{y + y'}{2}\Big)\right)\leq \frac{s}{ t+ s} h_{p, \Delta(x_2)}\big(t(x,y)\big)+ \frac{t}{ t+ s} h_{p, \Delta(x_2')}\big(s(x,y')\big),
\end{aligned}
\end{equation}
for all $x_2, x_2', x, y,y'\in \R$. 
\end{claim}

\begin{proof}
We start by deriving a useful formula for $h_{p,\Delta(x_2)}$.
Referrring to Figure \ref{SymSlidingFig} (a), the sides of 
$\Delta(x_2)$ are 
\begin{equation*}
\begin{aligned}
&\hbox{(right side)}\q    &\xi_r(y)= x_1+\frac{x_2-x_1}{y_2-y_1} (y-y_1), \q y\in[y_1,y_2],\cr
&\hbox{(left side)}\q    &\xi_\ell(y)= x_3+ \frac{x_2-x_3}{y_2-y_1}(y-y_1),
\q y\in[y_1,y_2]. 
\end{aligned}
\end{equation*}
Denote by
$$
h:=y_2-y_1
$$
the height of the triangle. 
From \eqref{hpKDef} observe
that $|\Delta(x_2)|$ cancels from both sides
of \eqref{Claim4.5Eq}, so we can assume for the proof
that $|\Delta(x_2)|=1$.
Then,
\begin{equation*}
    \begin{aligned}
        &e^{ph_{p, \Delta(x_2)}(x,y)} \\
        &= \int_{\Delta(x_2)} e^{pxu+ pyv}
        {\dif u\dif v} \\
        &= \int_{v= y_1}^{y_2} e^{pyv} 
        \int_{u=\xi_\ell(v)}^{\xi_r(v)}
        e^{pxu}\dif u{\dif v} \\
        &= \int_{v=y_1}^{y_2} e^{pyv} \frac{e^{px(x_1+ \frac{x_2-x_1}{h}(v-y_1))}- e^{px(x_3+\frac{x_2-x_3}{h}(v-y_1))}}{px} {\dif v} \\
        &= \int_{v=y_1}^{y_2} e^{pyv} e^{px x_2\frac{v-y_1}{h}} \frac{e^{pxx_1 (1-\frac{v-y_1}{h})} - e^{pxx_3(1- \frac{v-y_1}{h})}}{px}{\dif v} \\
        &= \int_{v=y_1}^{y_2} e^{pyv} e^{px x_2\frac{v-y_1}{h}} e^{px\frac{x_1+x_3}{2}(1-\frac{v-y_1}{h})}\frac{e^{px\frac{x_1-x_3}{2} (1-\frac{v-y_1}{h})} - e^{-px\frac{x_1-x_3}{2}(1- \frac{v-y_1}{h})}}{px}{\dif v} \\
        &=  \int_{v=y_1}^{y_2} e^{pyv} e^{px x_2\frac{v-y_1}{h}} e^{px\frac{x_1+x_3}{2}(1-\frac{v-y_1}{h})} \frac{\sinh\left(px\frac{x_1-x_3}{2}(1-\frac{v-y_1}{h})\right)}{px}{\dif v}. 
    \end{aligned}
\end{equation*}
Note that \cite[Claim 5.19]{BMR23}
\begin{equation}    
\label{JlogcxEq}
J(x,v)\defeq \frac{\sinh\left(px\frac{x_1-x_3}{2}(1-\frac{v-y_1}{h})\right)}{px}
\quad \hbox{is log-convex in $x$},
\end{equation}
and that
\begin{equation}
\label{hpDeltaEq}
h_{p,\Delta(x_2)}(x,y)= \frac1p\log\int_{v=y_1}^{y_2} e^{pyv} e^{px x_2\frac{v-y_1}{h}} e^{px\frac{x_1+x_3}{2}(1-\frac{v-y_1}{h})}  J(x,v)
{\dif v}.
\end{equation}

We can now prove \eqref{Claim4.5Eq}.
Let $x_2, x_2', y, y' \in \R$. Note that 
$p,x_1,x_3,y_1,y_2,h$ are all fixed parameters and 
\begin{equation}
\label{rstEq}
    r= \frac{2ts}{t+ s}= \frac{s}{t+ s}t+ \frac{t}{t+ s} s. 
\end{equation}
By \eqref{hpDeltaEq}, \eqref{rstEq}, and \eqref{JlogcxEq}, 
{\allowdisplaybreaks
\begin{align*}
    &h_{p,\Delta((x_2+x_2')/2)}\left(r\Big(x, \frac{y+y'}{2}\Big)\right)\\
    &=\frac1p\log\int_{v=y_1}^{y_2} e^{pv r\frac{y+y'}{2}} e^{prx \frac{x_2+x_2'}{2}\frac{v-y_1}{h}} e^{prx\frac{x_1+x_3}{2}(1-\frac{v-y_1}{h})} J(rx,v)
    {\dif v}
    \\
    &=\frac1p\log\int_{v=y_1}^{y_2} e^{pv (\frac{s}{t+ s}ty+ \frac{t}{t+ s} sy')} e^{px (\frac{s}{t+ s} tx_2+\frac{t}{t+ s} s x_2')\frac{v-y_1}{h}}  \\
    &\hspace{2.4cm} e^{p(\frac{s}{t + s}t+ \frac{t}{t+ s} s)x\frac{x_1+x_3}{2}(1-\frac{v-y_1}{h})} J\left(\frac{s}{t+ s}t x + \frac{t}{t+ s} sx, v\right)
    {\dif v}
    \\
    &\leq \frac1p\log\int_{v=y_1}^{y_2} e^{pv (\frac{s}{t+ s}ty+ \frac{t}{t+ s} sy')} e^{px (\frac{s}{t+ s} t x_2+\frac{t}{t+ s} s x_2')\frac{v-y_1}{h}}  \\
    &\hspace{2.4cm} e^{p(\frac{s}{t+ s}t+ \frac{t}{t+ s} s)x\frac{x_1+x_3}{2}(1-\frac{v-y_1}{h})} J(tx, v)^{\frac{s}{t + s}} J(sx,v)^{\frac{t}{t+s}}
    {\dif v}
    \\
    &= \frac1p\log\int_{v=y_1}^{y_2} \left( e^{pvty} e^{pxt x_2\frac{v-y_1}{h}} e^{ptx\frac{x_1+x_3}{2}(1-\frac{v-y_1}{h})} J(tx, v)\right)^{\frac{s}{t+ s}} \\
    &\hspace{2.35cm}\left( e^{pvsy'} e^{pxs x_2'\frac{v-y_1}{h}} e^{psx\frac{x_1+x_3}{2}(1-\frac{v-y_1}{h})} J(sx,v)\right)^{\frac{t}{t+s}} 
    {\dif v}.
\end{align*}
}
that by H\"older's inequality satisfies,  
{\allowdisplaybreaks
\begin{align*}
 &\leq \frac1p \log\left[\left(\int_{v=y_1}^{y_2} e^{pvty} e^{pxt x_2 \frac{v-y_1}{h}} e^{ptx \frac{x_1+x_3}{2}\big( 1-\frac{v-y_1}{h}\big)} J(tx,v) 
 {dv}
 \right)^{\frac{s}{t+ s}}\right.\\
 &\left.
\hspace{1.6cm}
 \left( \int_{v= y_1}^{y_2} e^{pvsy'} e^{pxs x_2' \frac{v-y_1}{h}} e^{psx\frac{x_1+x_3}{2}\big( 1-\frac{v-y_1}{h}\big)} J(sx, v) 
 {dv}
 \right)^{\frac{t}{t + s}}
 \right]\\
 &= \frac{s}{t+ s} \frac1p \log \left(\int_{v=y_1}^{y_2} e^{pvty} e^{pxt x_2 \frac{v-y_1}{h}} e^{ptx \frac{x_1+x_3}{2}\big( 1-\frac{v-y_1}{h}\big)} J(tx,v) 
 {dv}
 \right) \\
 &+ \frac{t}{t+s} \frac1p \log \left( \int_{v= y_1}^{y_2} e^{pvsy'} e^{pxs x_2' \frac{v-y_1}{h}} e^{psx\frac{x_1+x_3}{2}\big( 1-\frac{v-y_1}{h}\big)} J(sx, v)
 {dv}
 \right) \\
 &= \frac{s}{t+ s} h_{p, \Delta(x_2)}(t(x,y))+ \frac{t}{t + s} h_{p,\Delta(x_2')}(s(x, y')),
\end{align*}
}
i.e., \eqref{Claim4.5Eq} holds.
\end{proof}

\subsection{\texorpdfstring{Joint `convexity' of the $L^p$-support for 
sliding polytopes}{Joint convexity of the Lp-support for 
sliding polytopes}}
\label{SymEntirePolytopeSection}

A consequence of Claim \ref{TriangleClaim} is the
same joint `convexity' for the $L^p$-support function of
the entire 1-parameter family of sliding polytopes $S(x_2)$.

\begin{lemma}
\label{hpSymCor}
    Let $p>0$ and $S(x_2)\in\mathcal{S}$. 
    For $r,t, s>0$ with $\frac{2}{r}= \frac{1}{t}+ \frac{1}{s}$, 
\begin{equation*}
\begin{aligned}
    &h_{p,S\big(\frac{x_2+x_2'}{2}\big)}\left(r\Big(x, \frac{y+y'}{2}\Big)\right)\leq \frac{s}{t+ s} h_{p, S(x_2)}(t(x,y))+ \frac{t}{t+ s} h_{p, S(x_2')}(s(x,y')),
\end{aligned}
\end{equation*}
\end{lemma}

\begin{proof}
Let $\Delta(x_2)$ and $\hat S$ be as in \eqref{Deltax2Def}--\eqref{hatS},
with 
$S(x_2)= \widehat{S}\cup \pm\Delta(x_2), 
$
and each having constant volume (recall \eqref{Sx2constEq}).
Write $|S|=|S(x_2)|,\, |\Delta|=|\Delta(x_2)|$ to emphasize the independence from $x_2$. By Corollary \ref{hpMultipleDecomp}, 
\begin{equation}
\label{Sx2DecompEq}
    \begin{aligned}
        &e^{p h_{p,S(x_2)}(x,y)}
        = \frac{|\widehat{S}|}{|S|} e^{ph_{p,\widehat{S}}(x,y)}+ \frac{|\Delta|}{|S|} e^{ph_{p,\Delta(x_2)}(x,y)}+ \frac{|\Delta|}{|S|} e^{p h_{p, -\Delta(x_2)}(x,y)}.
    \end{aligned}
\end{equation}
By \eqref{rstEq},
\begin{equation*}
\begin{aligned}
    r\Big(x, \frac{y+ y'}2\Big) &= \frac{2ts}{t+ s}\frac{(x,y)+ (x, y')}{2} 
    = \frac{s}{t + s} (tx, ty) + \frac{t}{t + s} (sx, sy).
\end{aligned}
\end{equation*}
Thus, by convexity of $h_{p, \widehat{S}}$, 
\begin{equation}
\label{ShatConvexity}
    h_{p,\widehat{S}}\left(r\Big(x, \frac{y+y'}{2}\Big)\right)\leq \frac{s}{t+ s} h_{p,\widehat{S}}(t(x,y))+ \frac{t}{t+ s} h_{p,\widehat{S}}(s(x,y')). 
\end{equation}
Now we can conclude from Claim \ref{TriangleClaim} and the auxiliary Lemma \ref{auxLemma} below. Indeed, by \eqref{Claim4.5Eq}, \eqref{Sx2DecompEq} (used several times), \eqref{ShatConvexity}, and \eqref{3termIneq},
{\allowdisplaybreaks
\begin{align*}
        &e^{p h_{p, S((x_2 + x_2')/2)}(r(x, \frac{y+y'}{2}))}\\
        &=  \frac{|\widehat{S}|}{|S|} e^{p h_{p, \widehat{S}}(r(x, \frac{y+y'}{2})} 
        + \frac{|\Delta|}{|S|} e^{ph_{p,\Delta(x_2/2+ x_2'/2)}(r(x, \frac{y+y'}{2}))} 
        + \frac{|\Delta|}{|S|} e^{ph_{p,-\Delta(x_2/2+ x_2'/2)}(r(x, \frac{y+y'}{2}))} \\
        &\leq \frac{|\widehat{S}|}{|S|} \left( e^{p h_{p,\widehat{S}}(t(x,y))}\right)^{\frac{s}{t+s}} \left( e^{p h_{p, \widehat{S}}(s(x,y'))} \right)^{\frac{t}{t+s}} \\
        & + \frac{|\Delta|}{|S|} \left( e^{p h_{p,\Delta(x_2)}(t(x,y))}\right)^{\frac{s}{t+s}} \left( e^{ph_{p,\Delta(x_2')(s(x,y'))}}\right)^{\frac{t}{t+s}} \\
        & + \frac{|\Delta|}{|S|} \left( e^{p h_{p,-\Delta(x_2)}(t(x,y))}\right)^{\frac{s}{t+s}} \left( e^{ph_{p,-\Delta(x_2')(s(x,y'))}}\right)^{\frac{t}{t+s}} \\
        &=\left(\frac{|\widehat{S}|}{|S|} e^{p h_{p,\widehat{S}}(t(x,y))}\right)^{\frac{s}{t+s}} \left(\frac{|\widehat{S}|}{|S|} e^{p h_{p, \widehat{S}}(s(x,y'))} \right)^{\frac{t}{t+s}} \\
        & + \left(\frac{|\Delta|}{|S|} e^{p h_{p,\Delta(x_2)}(t(x,y))}\right)^{\frac{s}{t+s}} \left(\frac{|\Delta|}{|S|} e^{ph_{p,\Delta(x_2')(s(x,y'))}}\right)^{\frac{t}{t+s}} \\
        & + \left(\frac{|\Delta|}{|S|} e^{p h_{p,-\Delta(x_2)}(t(x,y))}\right)^{\frac{s}{t+s}} \left(\frac{|\Delta|}{|S|} e^{ph_{p,-\Delta(x_2')(s(x,y'))}}\right)^{\frac{t}{t+s}} \\
        &\leq \left( \frac{|\widehat{S}|}{|S|} e^{p h_{p, \widehat{S}}(t(x,y))}+ \frac{|\Delta|}{|S|} e^{ph_{p, \Delta(x_2)}(t(x,y))}+ \frac{|\Delta|}{|S|} e^{p h_{p, -\Delta(x_2)}(t(x,y))} \right)^{\frac{s}{t+s}} \\
        & \hspace{.45cm} \left( \frac{|\widehat{S}|}{|S|} e^{p h_{p, \widehat{S}}(s(x,y'))}+ \frac{|\Delta|}{|S|} e^{ph_{p, \Delta(x_2')}(s(x,y'))}+ \frac{|\Delta|}{|S|} e^{p h_{p, -\Delta(x_2')}(s(x,y'))} \right)^{\frac{t}{t+s}} \\
        &= e^{\frac{s}{t+s} p h_{p, S(x_2)}(t(x,y))} e^{\frac{t}{t+s} ph_{p, S(x_2')}(s(x,y'))}, 
\end{align*}
}
hence the claim.
\end{proof}

\begin{lemma}
\label{auxLemma}
    For $a,b,c,d,f,g>0$ and $\lambda\in(0,1)$, 
    \begin{equation}
    \label{2termIneq}
        a^{1-\lambda} b^\lambda+ c^{1-\lambda}d^\lambda\leq (a+c)^{1-\lambda} (b+d)^\lambda, 
    \end{equation}
    and
    \begin{equation}
    \label{3termIneq}
        a^{1-\lambda}b^\lambda+ c^{1-\lambda}d^\lambda+ f^{1-\lambda}g^\lambda\leq (a+c+f)^{1-\lambda}(b+d+g)^{\lambda}.
    \end{equation}
\end{lemma}
\begin{proof}
    By the weighted Arithmetic Mean--Geometric Mean applied twice, 
    \begin{equation*}
        \begin{aligned}
            & \left( \frac{a}{a+c}\right)^{1-\lambda}\left( \frac{b}{b+d}\right)^{\lambda}+ \left( \frac{c}{a+c}\right)^{1-\lambda}\left( \frac{d}{b+d}\right)^\lambda \\
            &\leq (1-\lambda)\frac{a}{a+c}+ \lambda\frac{b}{b+d}+ (1-\lambda)\frac{c}{a+c}+ \lambda\frac{d}{b+d} \\
            &= (1-\lambda) + \lambda =1, 
        \end{aligned}
    \end{equation*}
    hence \eqref{2termIneq}. 

    For \eqref{3termIneq} it is enough to apply \eqref{2termIneq} twice, 
    \begin{equation*}
        a^{1-\lambda}b^\lambda + c^{1-\lambda} d^\lambda+ f^{1-\lambda} g^\lambda \leq (a+c)^{1-\lambda}(b+d)^\lambda + f^{1-\lambda} g^\lambda\leq (a+c+f)^{1-\lambda}(b+d+g)^\lambda. 
    \end{equation*}
\end{proof}

\begin{proof}[Proof of Lemma \ref{symNormConvexity}]
    By Remark \ref{convexityRemark} and continuity
    properties of convex functions \cite[Theorem 10.1]{Rock70} it is enough to prove midpoint convexity. 
    Fix $y, y'\in \R$. Let,
    \begin{equation*}
    \begin{aligned}
        &H(r)\defeq \exp(-h_{p, S((x_2+x_2')/2)}(r(1, 
        {y/2 + y'/2}))), \quad r>0,\\
        &F(t)\defeq \exp(-h_{p,S(x_2)}(t(1,y))), \quad t>0,\\
        &G(s)\defeq \exp(-h_{p,S(x_2')}(s(1,y'))), \quad s>0. 
    \end{aligned}
    \end{equation*}
    For $r,t,s>0$ with $\frac{2}{r}= \frac{1}{t}+ \frac{1}{s}$, by Lemma \ref{hpSymCor}, 
    \begin{equation*}
        H(r)\geq F(t)^{\frac{s}{t+s}} G(s)^{\frac{t}{t+s}}. 
    \end{equation*}
    Thus, by Theorem \ref{BallIneq} (with $q=2$), 
    \begin{equation*}
        \begin{aligned}
           \left(\int_0^\infty r H(r)dr\right)^{-\frac12}\leq \frac12\left( \int_0^\infty t F(t)dt\right)^{-\frac12} + \frac12\left( \int_0^\infty s G(s)ds\right)^{-\frac12}. 
        \end{aligned}
    \end{equation*}
    Since by \eqref{KcircpNormEq} (with $n=2$), 
    \begin{equation*}
    \begin{aligned}
        &\left(\int_0^\infty rH(r)dr\right)^{-\frac12}= \|(1, y/2+ y'/2))\|_{S((x_2+x_2')/2)^{\circ,p}}, \\
        & \left( \int_0^\infty tF(t)dt\right)^{-\frac12}= \|(1, y)\|_{S(x_2)^{\circ,p}}, \\
        & \left( \int_0^\infty sG(s)ds\right)^{-\frac12}= \|(1, y')\|_{S(x_2')^{\circ,p}},
    \end{aligned}
    \end{equation*}
    the proof of Lemma \ref{symNormConvexity} is complete.
\end{proof}

\subsection{Proof of Theorem \ref{2DimpMahlerSym}}
\label{SymFinishSection}
\label{SymFinishingSection}
    By continuity of the $L^p$-Mahler volume under the Hausdorff topology \cite[Lemma 5.11]{BMR23}, it suffices to work with symmetric polytopes, and as explained in the proof of Proposition \ref{symSlidingProp} it further suffices
    to restrict to polytopes in $\mathcal{S}$ \eqref{SClassEq}. 
    Let $S\in\mathcal{S}$ be a symmetric polytope with $2m\geq 6$ vertices. Write $S_{2m} = S$. By repeated application of Proposition \ref{symSlidingProp}, there exist symmetric polytopes $S_{2m-2}, \ldots, S_{6}, S_4$ with $2m-2, \ldots, 6, 4$ vertices respectively such that 
    \begin{equation*}
        \M_p(S_{2m})\geq \M_p(S_{2m-2})\geq \cdots\geq \M_p(S_6) \geq \M_p(S_4). 
    \end{equation*}
    Since $S_4$ is a symmetric polytope with four vertices, it lies in the $GL(2,\R)$ orbit of $[-1,1]^2$ (to see that,
    note $S_4=\co\{p_1,p_2,-p_1,-p_2\}$ and there is a unique
    $A\in GL(2,\R)$ that sends $p_1$ to $(1,0)$ and $p_2$
    to $(0,1)$), hence $\M_p(S_4)= \M_p([-1,1]^2)$ \cite[Lemma 4.6]{BMR23}. 
    This concludes the proof of Theorem \ref{2DimpMahlerSym}.

\section{\texorpdfstring{Lower bound for $\M_p$}{Lower bound for Mp }}
\label{MpNonSymSection}

In analogy with the symmetric case of \S\ref{MpSymSection}, 
fix a convenient class of (not necessarily symmetric) polytopes as in Figure \ref{SlidingFig}:
\begin{equation}
\label{PclassEq}
\mathcal{P}=
\left\{
P= \co\{(x_1, y_1), \ldots, (x_m, y_m) \}
\;\middle|\;
\begin{alignedat}{2}
  & m\in\mathbb{N}, m>3 \\
  & 0\in\mathrm{int}\,\widehat{P} \text{ }  \eqref{hatP} \\
  & (x_i, y_i) \text{ are extreme points of $P$,} \\
  & \text{counter-clockwise oriented}\\
  & x_3\leq x_2\leq x_1 \\
  & 0\leq y_1= y_3 < y_2
\end{alignedat}
\right\}.
\end{equation}
Unlike in  \S\ref{MpSymSection}, for the proof of Theorem \ref{2DimpMahler} only one vertex will be sliding. On the other hand, a technical issue absent in the symmetric case arises: the necessity to also translate in the direction of the sliding.
Indeed, for symmetric polytopes, the $L^p$-Santal\'o point
\eqref{LpPointDef} remains at the origin throughout the sliding. That is because its $L^p$-support function is an even function and hence its barycenter lies at the origin, which implies the $L^p$-Santal\'o point must be the origin \cite[Proposition 1.5]{BMR23}.  However, for non-symmetric polytopes, this may not be the case
and it is natural to work with 
\begin{equation}
\label{starpnotationEq}
    K^{*,p}\defeq (K-s_p(K))^{\circ,p}.
\end{equation}

The main result of this section is a generalization of
Lemma \ref{SymSlidingLemma}.

\begin{lemma}
\label{SlidingLemma}
Let $p\in (0,\infty]$. For $P(x_2)\in\mathcal{P}$ (recall (\ref{PclassEq})--(\ref{starpnotationEq})),
 $       x_2\mapsto {|P(x_2)^{*,p}|}^{-1}, x_2\in[\xi_\ell, \xi_r], \text{ }
  $
 is convex.
\end{lemma}

In the absence of symmetry, the volume of $|P(x_2)^{\circ,p}|$ splits into two integrals $I_+$ and $I_-$. By Lemma \ref{-1concaveLemma}, the reciprocals of each of $I_+$ and $I_-$ are convex under sliding. However, the reciprocal of their sum, that is the volume of $P(x_2)^{\circ,p}$, may not be. To tackle this difficulty, motivated by the work of Meyer--Reisner \cite[Lemma 5]{MR06}, we simultaneously translate the polytope in the direction of the sliding. This is to ensure that under sliding the ratio $I_+/ I_-$ remains constant, which would ultimately be enough to obtain Lemma \ref{SlidingLemma}. We start by studying the $L^p$-support functions of triangles under sliding and translation (Claim \ref{TriangleSliding}) obtaining an inequality that generalizes Claim \ref{TriangleClaim} (that only dealt
with sliding). Consequently, we obtain an inequality for the $L^p$-support of $P(x_2)$ (Lemma \ref{Px2Slidinghp})
that generalizes Claim \ref{TriangleClaim}. This allows for the application of Ball's theorem (Theorem \ref{BallIneq}) proving the joint convexity of the norm of $P(x_2)^{\circ,p}$ under sliding and translation. Consequently, the reciprocals of $I_+$ and $I_-$ are convex under sliding and translation (Corollary \ref{ConvexityOfI}). We then show that their ratio $I_+/ I_-$ can be kept constant under sliding via appropriate translation (Proposition \ref{BalancingProp}). This turns out to be just enough to obtain Lemma \ref{SlidingProp}. 

\subsection{Triangles revisited}
\label{TrianglesRevisitSection}
Next we state a simple generalization of Claim \ref{TriangleClaim}
that shows the same type of `strong joint convexity' but now 
with respect to 3 variables. 
Apart from $y\in\R$ (the parameter for the boundary of the body) and
the sliding parameter $x_2\in\R$, we introduce a translation parameter $x_0\in\R$. We only consider translation by $(x_0,0)$,
i.e., in the $x$-direction.
This is because we can only translate in the same direction as the sliding to obtain the desired inequality, as it is the only motion for which the numbers workout. 

\begin{claim}
\label{TriangleSliding}
Let $p>0$ and a triangle $\Delta(x_2)\defeq \co\{(x_1, y_1), (x_2, y_2), (x_3, y_3)\}$. Let $r,t,s>0$ with $\frac{2}{r}= \frac{1}{t}+ \frac{1}{s}$, 
\begin{equation*}
\begin{aligned}
    &h_{p,\Delta((x_2+x_2')/2)-(x_0,0)/2-(x_0',0)/2}\left(r\Big(x, \frac{y+y'}{2}\Big)\right)\\
    &\leq \frac{s}{t+s} h_{p, \Delta(x_2)-(x_0,0)}(t(x,y))+ \frac{t}{t+s} h_{p, \Delta(x_2')-(x_0',0)}(s(x,y')),
\end{aligned}
\end{equation*}
for all $x_0, x_0', 
x_2, x_2', x, y, y'\in \R$.
\end{claim}

\begin{proof}
Using \eqref{rstEq},
    \begin{equation}
    \label{TranslationPart}
        \begin{aligned}
            &\left\langle r\Big(x, \frac{y+y'}{2}\Big), \frac{(x_0,0)+  (x_0',0)}{2}\right\rangle 
            &= \frac{s}{t+s} \langle t(x,y), (x_0,0)\rangle + \frac{t}{t+s} \langle s(x,y'), (x_0',0)\rangle. 
        \end{aligned}
    \end{equation}
    Recall the identity \cite[Lemma 2.2 (ii)]{BMR23}
\begin{equation}
\label{hpTranslate}
    h_{p,K-x}(y) = h_{p,K}(y)- \langle x,y\rangle, \quad x,y\in \R^n, 
\end{equation}
for the $L^p$-support of the translated body.
By \eqref{TranslationPart}, Claim \ref{TriangleClaim}, and \eqref{hpTranslate} (twice), 
    \begin{equation*}
        \begin{aligned}
            &h_{p,\Delta((x_2+x_2')/2)-(x_0, y_0)/2- (x_0', y_0)/2}\left(r\Big(x, \frac{y+y'}{2}\Big)\right) \\
            &= h_{p,\Delta((x_2+x_2')/2)}\left(r\Big(x, \frac{y+y'}{2}\Big)\right) - \left\langle r\Big(x, \frac{y+y'}{2}\Big), \frac{(x_0, y_0)+ (x_0', y_0)}{2}\right\rangle \\
            &\leq \frac{s}{t+s} \left( h_{p, \Delta(x_2)}(t(x,y)) -\langle t(x,y), (x_0, y_0)\rangle\right) \\
            &+ \frac{t}{t+s} \left( h_{p,\Delta(x_2')}(s(x,y'))- \langle s(x,y'), (x_0', y_0)\rangle\right) \\
            &= \frac{s}{t+s} h_{p,\Delta(x_2)-(x_0, y_0)}(t(x, y))+ \frac{t}{t+s} h_{p,\Delta(x_2')- (x_0',y_0)}(s(x,y')). 
        \end{aligned}
    \end{equation*}
\end{proof}

\subsection{3-parameter joint `convexity' of the near norm for sliding
and translating polytopes}
\label{ConvNormSection}
The next result is the simple generalization of Lemma \ref{hpSymCor} (concerning sliding alone) 
to the non-symmetric
setting.
Namely, it extends the type of convexity obtained
in Claim \ref{TriangleSliding} for the sliding and translating 
1-parameter family of triangles $\Delta(x_2)-(x_0,0)$ to the sliding and translating
1-parameter family of polytopes $P(x_2)-(x_0,0)$. 
Once again the proof reduces to that of Lemma \ref{hpSymCor}
by \eqref{hpTranslate}. We give the full details
for completeness.

\begin{lemma}
\label{Px2Slidinghp}
    Let $p>0$ and $P(x_2)\in\mathcal{P}$. For $x_0, x_0', 
    x, y, y'\in \R, x_2, x_2'\in [\xi_\ell, \xi_r]$, and $r,t,s>0$ with $\frac2r = \frac{1}{t}+ \frac{1}{s}$, 
    \begin{equation*}
        \begin{aligned}
            &h_{p, P((x_2+x_2')/2)-((x_0+x_0')/2,0)}\left(r\Big(x, \frac{y+y'}{2}\Big)\right) \\
            &\leq \frac{s}{t+s} h_{p, P(x_2)- (x_0, 0)}(t(x,y))+ \frac{t}{t+s} h_{p, P(x_2')- (x_0', 0)}(s(x,y')).
        \end{aligned}
    \end{equation*}
\end{lemma}
\begin{proof}
    Let $\Delta(x_2)$ and $\widehat{P}$ as in \eqref{Deltax2Def} and \eqref{hatP},
    so  that $P(x_2)= \Delta(x_2)\cup \widehat{P}$,
    $|\Delta(x_2)|=|\Delta|$, and $|P|= |\Delta|+ |\widehat{P}|$. 
    By Claim \ref{hpDecomp}, 
    \begin{equation}
    \label{PPhatDeltaEq}
        e^{h_{p, P(x_2)}(x,y)}= \frac{|\widehat{P}|}{|P|} e^{h_{p,\widehat{P}}(x,y)} + \frac{|\Delta|}{|P|} e^{h_{p, \Delta(x_2)}(x,y)}.
    \end{equation}
By \eqref{rstEq}
and convexity of the $L^p$-support function \cite[Lemma 2.1]{BMR23},
    \begin{equation}
    \label{PhatConvexity}
        h_{p,\widehat{P}}\left(r\Big(x, \frac{y+y'}{2}\Big)\right)\leq \frac{s}{t+s} h_{p,\widehat{P}}(t(x,y))+ \frac{t}{t+s} h_{p,\widehat{P}}(s(x,y')). 
    \end{equation}
By \eqref{PPhatDeltaEq}, \eqref{PhatConvexity}, \eqref{2termIneq}, and Claim \ref{TriangleSliding}, 
    {\allowdisplaybreaks
    \begin{align*}
            &e^{h_{p, P((x_2+x_2')/2)}\left(r\Big(x,\frac{y+y'}{2}\Big)\right)} \\
            &= \frac{|\widehat{P}|}{|P|} e^{h_{p,\widehat{P}}(r(x,\frac{y+y'}{2})} + \frac{|\Delta|}{|P|} e^{h_{p,\Delta((x_2+x_2')/2)}(r(x,\frac{y+y'}{2})} \\
            &\leq \frac{|\widehat{P}|}{|P|}\left(e^{h_{p,\widehat{P}}(t(x,y))}\right)^{\frac{s}{t+s}} \left( e^{h_{p,\widehat{P}}(s(x,y'))}\right)^{\frac{t}{t+s}}
            +  \frac{|\Delta|}{|P|}\left( e^{h_{p,\Delta(x_2)}(t(x,y))}\right)^{\frac{s}{t+ s}} 
            \left( e^{p,\Delta(x_2')(s(x,y'))}\right)^{\frac{t}{t+s}} \\
            &= \left( \frac{|\widehat{P}|}{|P|} e^{h_{p,\widehat{P}}(t(x,y))}\right)^{\frac{s}{t+s}} \left( \frac{|\widehat{P}|}{|P|} e^{h_{p,\widehat{P}}(s(x,y'))}\right)^{\frac{t}{t+s}}
            + \left( \frac{|\Delta|}{|P|} e^{h_{p,\Delta(x_2)}(t(x,y))}\right)^{\frac{s}{t+ s}} 
            \left(\frac{|\Delta|}{|P|} e^{p,\Delta(x_2')(s(x,y'))}\right)^{\frac{t}{t+s}} \\
            &\leq \left(\frac{|\widehat{P}|}{|P|} e^{h_{p,\widehat{P}}(t(x,y))}+ \frac{|\Delta|}{|P|} e^{h_{p,\Delta(x_2)} (t(x,y))} \right)^{\frac{s}{t+s}} 
            \left(\frac{|\widehat{P}|}{|P|} e^{h_{p,\widehat{P}}(s(x,y'))}+ \frac{|\Delta|}{|P|} e^{h_{p,\Delta(x_2')}(s(x,y'))} \right)^{\frac{t}{t+s}} \\
            &= \left( e^{h_{p,P(x_2)}(t(x,y))}\right)^{\frac{s}{t+s}} \left( e^{h_{p,P(x_2')}(s(x,y'))}\right)^{\frac{t}{t+s}},
    \end{align*}
    } 
    proving the claim when $x_0= x_0'=0$. In general, the claim follows from 
    this inequality combined with \eqref{TranslationPart} and \eqref{hpTranslate}.     
\end{proof}

\begin{lemma}
\label{normTripleConvexity}
    $(y, x_2, x_0)\mapsto \|(x,y)\|_{P(x_2)-(x_0, 0)}$ is convex.
\end{lemma}
\begin{proof}
By Remark \ref{convexityRemark} and \cite[Theorem 10.1]{Rock70}, it is enough to prove midpoint convexity. Let 
    \begin{equation*}
    \begin{aligned}
        &H(r)\defeq \exp(-h_{p, P((x_2+x_2')/2)- (x_0,0)/2-(x_0', 0)/2}(r(x, 
        {y/2+y'/2}))), \quad r>0,\\
        &F(t)\defeq \exp(-h_{p,P(x_2)-(x_0, 0)}(t(x,y))), \quad t>0,\\
        &G(s)\defeq \exp(-h_{p,P(x_2')-(x_0', 0)}(s(x,y'))), \quad s>0. 
    \end{aligned}
    \end{equation*}
    For $r,t,s>0$ with $\frac2r= \frac{1}{t}+ \frac{1}{s}$,  by Lemma \ref{Px2Slidinghp}, 
  $      H(r)\geq F(t)^{\frac{s}{t+s}} G(s)^{\frac{t}{t+s}}. $
    Therefore, by Theorem \ref{BallIneq} for $q=2$, 
    \begin{equation*}
        \left( \int_0^\infty r H(r)dr \right)^{-\frac12}\leq \frac12\left( \int_0^\infty t F(t)dt\right)^{-\frac12} + \frac12 \left( \int_0^\infty sG(s)ds\right)^{-\frac12}. 
    \end{equation*}
    Finally, by \eqref{KcircpNormEq}, 
    \begin{equation*}
        \begin{aligned}
            & \left( \int_0^\infty rH(r)dr\right)^{-\frac12}= \| (1,y)/2+ (1,y')/2\|_{(P((x_2+x_2')/2)- (x_0,0)/2- (x_0', 0)/2)^{\circ,p}}, \\
            & \left( \int_0^\infty tF(t)dt\right)^{-\frac12} = \|(1,y)\|_{(P(x_2)-(x_0, 0))^{\circ,p}}, \\
            & \left( \int_0^\infty sG(s)ds\right)^{-\frac12} = \|(1,y')\|_{(P(x_2')-(x_0',0))^{\circ,p}}, 
        \end{aligned}
    \end{equation*}
    hence the claim. 
\end{proof}

\subsection{\texorpdfstring{Balancing the left and right areas of the $L^p$-polar}{Balancing the left and right areas of the Lp-polar}}
\label{DecompSection}
Let
\begin{equation}
\label{IpmDefEq}
    I_\pm(x_2, x_0)\defeq \frac12 \int_{\R} \frac{dy}{\|(\pm1,\pm y)\|^2_{(P(x_2)-(x_0,0))^{\circ,p}}}= |(P(x_2)-(x_0,0))^{\circ,p}\cap \{\pm x>0\}|, 
\end{equation}
such that by Lemma \ref{suitableVolumeLemma},
\begin{equation*}
    |(P(x_2)-(x_0,0))^{\circ,p}|= I_+(x_2,x_0) + I_-(x_2,x_0).
\end{equation*}

\begin{corollary}
\label{ConvexityOfI}
    Let $P\in \mathcal{P}$. The functions 
    \begin{equation*}
        (x_2, x_0)\mapsto \frac{1}{I_\pm(x_2,x_0)}
    \end{equation*}
are convex
for $x_2\in [\xi_\ell, \xi_r]$ (recall Figure \ref{SlidingFig})
 and $(x_0,0)\in\mathrm{int}\,P$.
\end{corollary}
\begin{proof}
By Lemma \ref{normTripleConvexity} for $x=1$, $(y,x_2, x_0)\mapsto \|(1,y)\|_{P(x_2)-(x_0,0)}$ is convex, and hence by Lemma \ref{-1concaveLemma}, 
\begin{equation*}
    (x_2, x_0)\mapsto \bigg[{\frac12\int_\R \frac{dy}{\|(1,y)\|^2_{P(x_2)-(x_0,0)}}}\bigg]^{-1}= \frac{1}{|I_+(x_2,x_0)|}
\end{equation*}
is convex. Similarly, by Lemma \ref{normTripleConvexity} for $x=-1$, $(y,x_2, x_0)\mapsto \|(-1,y)\|_{P(x_2)-(x_0,0)}$ is convex, thus $(y,x_2, x_0)\mapsto \|(-1,-y)\|_{P(x_2)-(x_0,0)}$ is also convex (in general, if $f(x,y)$ is convex then $f(x,-y)$ is also convex). As a result, by Lemma \ref{-1concaveLemma}, 
\begin{equation*}
    (x_2, x_0)\mapsto \bigg[\frac12{\int_{\R}\frac{dy}{\|(-1, -y)\|^2_{P(x_2)-(x_0,0)}}}
    \bigg]^{-1}= \frac{1}{I_-(x_2,x_0)}
\end{equation*}
is convex. 
\end{proof}

To facilitate the algebraic steps required to derive Lemma \ref{SlidingLemma} from Corollary \ref{ConvexityOfI}, it is essential to keep the ratio $I_+/ I_-$ constant throughout the sliding coupled with translation. 

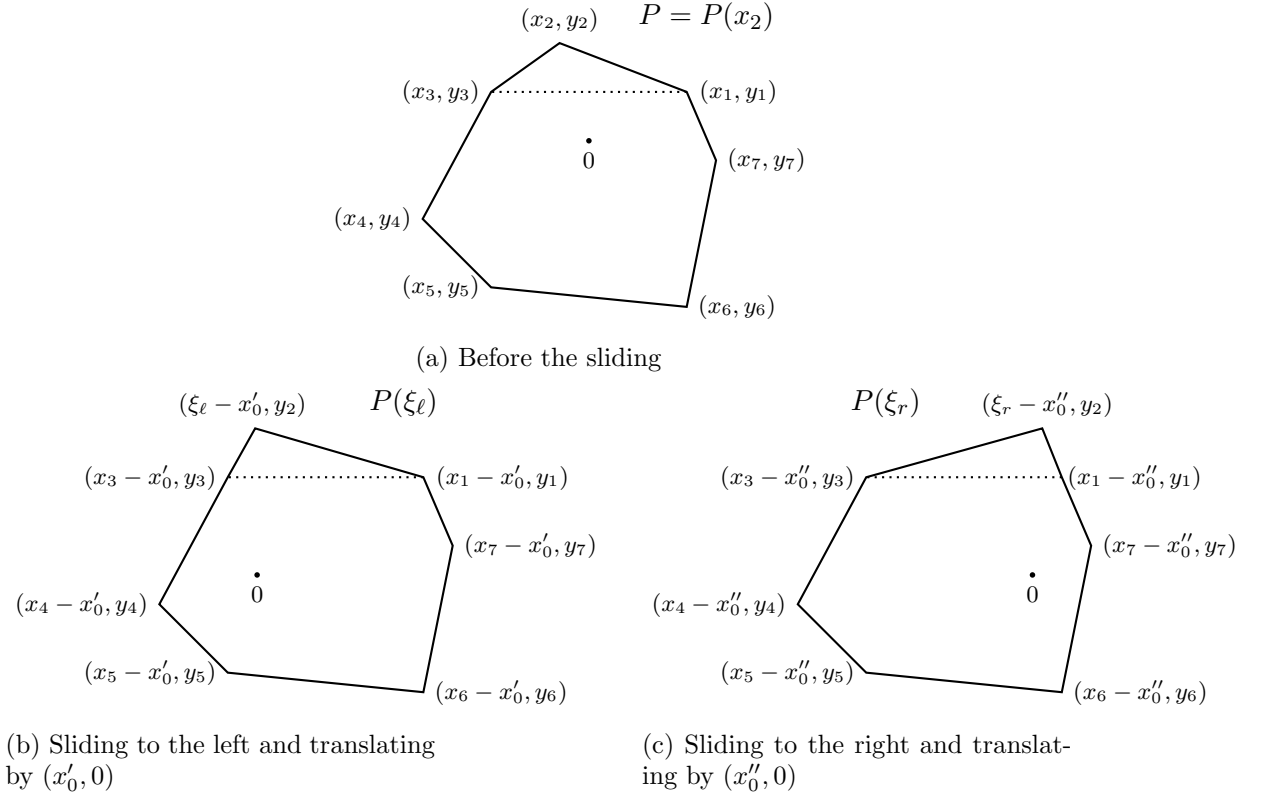
\begin{figure}[H]
\centering
\begin{subfigure}[b]{0.35\textwidth}
\centering
\begin{tikzpicture}[scale=1.3]
\draw[thick] (-1,1) -- (-1.7, -0.3) -- (-1,-1) -- (1,-1.2) -- (1.3, 0.3) -- (1, 1) -- (-.3,1.5) -- cycle;
\draw[thick, dotted] (1,1) -- (-1,1);
\node[above=.5pt] at (1.2, 1.5){$P= P(x_2)$};
\node[right=1pt] at (1, 1){\footnotesize $(x_1, y_1)$};
\node[above=.5pt] at (-.3, 1.5){\footnotesize $(x_2, y_2)$};
\node[left=.5pt] at (-1, 1){\footnotesize $(x_3, y_3)$};
\node[left=.5pt] at (-1.7, -0.3){\footnotesize $(x_4, y_4)$};
\node[left=.5pt] at (-1, -1){\footnotesize $(x_5, y_5)$};
\node[right=.5pt] at (1, -1.2){\footnotesize $(x_6, y_6)$};
\node[right=.5pt] at (1.3, 0.3){\footnotesize $(x_7, y_7)$};
\fill (0, 0.5)  circle[radius=0.8pt];
\node[below=.5pt] at (0, 0.5){\footnotesize $0$};
\end{tikzpicture}
\caption{\small Before the sliding}
\end{subfigure}
\\
\begin{subfigure}[b]{0.35\textwidth}
\centering
\begin{tikzpicture}[scale=1.3]
\draw[thick] (-1,1) -- (-1.7, -0.3) -- (-1,-1) -- (1,-1.2) -- (1.3, 0.3) -- (1, 1) -- (-.72,1.5) -- cycle;
\draw[thick, dotted] (1,1) -- (-1,1);
\node[above=.5pt] at (.8, 1.5){$P(\xi_\ell)$};
\node[right=1pt] at (1, 1){\footnotesize $(x_1-x_0', y_1)$};
\node[above=.5pt] at (-.85, 1.5){\footnotesize $(\xi_\ell-x_0', y_2)$};
\node[left=.5pt] at (-1, 1){\footnotesize $(x_3-x_0', y_3)$};
\node[left=.5pt] at (-1.7, -0.3){\footnotesize $(x_4-x_0', y_4)$};
\node[left=.5pt] at (-1, -1){\footnotesize $(x_5-x_0', y_5)$};
\node[right=.5pt] at (1, -1.2){\footnotesize $(x_6-x_0', y_6)$};
\node[right=.5pt] at (1.3, 0.3){\footnotesize $(x_7-x_0', y_7)$};
\fill (-.7, 0)  circle[radius=0.8pt];
\node[below=.5pt] at (-.7, 0){\footnotesize $0$};
\end{tikzpicture}
\caption{\small Sliding to the left and translating by $(x_0', 0)$}
\end{subfigure}
\hspace{2.5cm}
\begin{subfigure}[b]{0.35\textwidth}
\centering
\begin{tikzpicture}[scale=1.3]
\draw[thick] (-1,1) -- (-1.7, -0.3) -- (-1,-1) -- (1,-1.2) -- (1.3, 0.3) -- (1, 1) -- (.8,1.5) -- cycle;
\draw[thick, dotted] (1,1) -- (-1,1);
\node[above=.5pt] at (-.8, 1.5){$P(\xi_r)$};
\node[right= 2pt] at (.9, 1){\footnotesize $(x_1-x_0'', y_1)$};
\node[above=.5pt] at (.85, 1.5){\footnotesize $(\xi_r-x_0'', y_2)$};
\node[left=.5pt] at (-1, 1){\footnotesize $(x_3-x_0'', y_3)$};
\node[left=.5pt] at (-1.7, -0.3){\footnotesize $(x_4-x_0'', y_4)$};
\node[left=.5pt] at (-1, -1){\footnotesize $(x_5-x_0'', y_5)$};
\node[right=.5pt] at (1, -1.2){\footnotesize $(x_6-x_0'', y_6)$};
\node[right=.5pt] at (1.3, 0.3){\footnotesize $(x_7-x_0'', y_7)$};
\fill (.7, 0)  circle[radius=0.8pt];
\node[below=.5pt] at (.7, 0){\footnotesize $0$};
\end{tikzpicture}
\caption{\small Sliding to the right and translating by $(x_0'',0)$}
\end{subfigure}
    \caption{\small Sliding with translation}
\label{translatingFigure}
\end{figure}

\begin{proposition}
\label{BalancingProp}
    Let $P\in\mathcal{P}$ with $0\in\mathrm{int}\,P$. For any $x_2, x_2'\in [\xi_\ell, \xi_r]$
    (recall Figure \ref{SlidingFig}), there exists $x_0=x_0(x_2, x_2')\in\R$ such that $(x_0,0)\in\mathrm{int}\,P(x_2)$ and $(-x_0,0)\in\mathrm{int}\,P(x_2')$ so that
    \begin{equation}
    \label{ratioProp5.6Eq}
        \frac{|(P(x_2)-(x_0,0))^{\circ,p}\cap\{x>0\}|}
        {|(P(x_2)-(x_0,0))^{\circ,p}\cap\{x<0\}|}
        =\frac{|(P(x_2')-(-x_0,0))^{\circ,p}\cap\{x>0\}|}
        {|(P(x_2')-(-x_0,0))^{\circ,p}\cap\{x<0\}|}.
    \end{equation}
\end{proposition}

The key to proving Proposition \ref{BalancingProp} lies in the following intuition: if a convex body mostly lies on one side of a hyperplane, then its classical polar mostly lies on the other side (Figure \ref{fig:PPpolarHyperplane}). The hope
is that something roughly similar happens for the $L^p$-polar with $p$ finite.
Namely, that if we take $x_0$ very close to the left-hand side boundary of $P$ the difference between the two ratios in Proposition \ref{BalancingProp} is negative, and if we take $x_0$ to the right-hand side boundary, the difference between the two ratios is positive. If that is the case, by the intermediate value theorem, there must be a point where the two are equal. 

\begin{figure}[H]
\centering
\begin{tikzpicture}[scale=1.1]
\draw[->] (-1,0) -- (3,0) node[anchor=north west] {x};
\draw[->] (0,-1.5) -- (0,2) node[anchor=north west] {y};
\draw[thick]  (-.5, 0) -- (0,-1) -- (1.5,-1) -- (2.5, 0) -- (2, 1) -- (0,1.5) -- cycle;
\node[above right=1pt] at (1.7, 1.2){$P$};
\end{tikzpicture}
\quad\quad
\begin{tikzpicture}[scale=1.3]
\draw[->] (-2.5,0) -- (1,0) node[anchor=north west] {x};
\draw[->] (0,-1.3) -- (0,1.5) node[anchor=north west] {y};
\draw[thick]  (-2, 1/3) -- (-2,-1) -- (0, -1) -- (.4, -.4) -- (.4, .2) -- (1/6, 2/3) -- cycle;
\node[above left=1pt] at (-.8, .8){$P^\circ$};
\end{tikzpicture}
    \caption{When $P$ mostly lies in $\{x\geq 0\}$ then $P^{\circ}$ mostly lies in $\{x\leq 0\}$.}
    \label{fig:PPpolarHyperplane}
\end{figure}
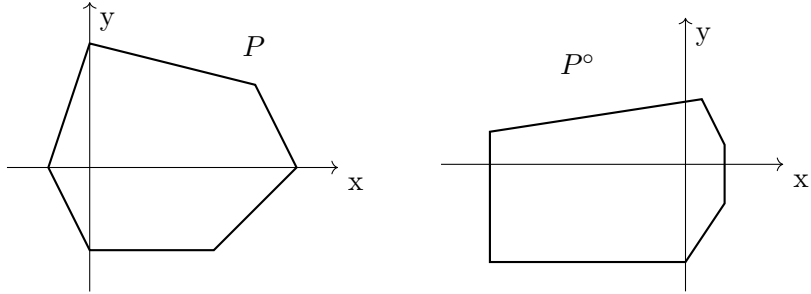

Several basic facts in this direction are the following:
\begin{claim}{\rm\cite[Claim 4.8]{BMR23}}
\label{MpInfinite}
    Let $p\in (0,\infty]$. For a convex body $K\subset \R^n$, $|K^{\circ,p}|<\infty$ if and only if $0\in\mathrm{int}\,K$. 
\end{claim}

\begin{claim}
\label{finiteClaim1}
{\rm \cite[Claims 2.18, 2.21]{Mastr23}}
    Let $p\in (0,\infty]$. For a convex body $K\subset \R^n$ with $0\in\mathrm{int}\,K$ and $t\in \R$ such that $(K-(t,0,\ldots,0))\subset \{x\in\R^n\,:\, x_1\ge0\}$
    (respectively, $(K-(t,0,\ldots,0))\subset \{x\in\R^n\,:\, x_1\le0\}$), 
    \begin{equation*}
       |(K-(t,0,\ldots,0))^{\circ,p}\cap 
        \{x\in\R^n\,:\,  x_1\ge0\}|
        <\infty, 
    \end{equation*}
    (respectively, $|(K-(t,0,\ldots,0))^{\circ,p}\cap 
        \{x\in\R^n\,:\, x_1\le0\}
        <\infty$).
\end{claim}
The next fact implies that 
none of the terms appearing in 
\eqref{ratioProp5.6Eq} vanish.

\begin{claim}
\label{NonzeroLpPolarVolClaim}
Let $p\in (0,\infty]$.    For a convex body $K\subset \R^n$ and $u\in\partial B_2^n$, $|K^{\circ,p}\cap \{x: \langle x,u\rangle\geq 0\}|>0$. 
\end{claim}
\begin{proof}
    The origin lies in the interior of $K^\circ$. By definition, a body is bounded, hence there is $M>0$ such that $|x|\leq M$ for all $x\in K$. Therefore, $\langle x,y\rangle\leq M|y|\leq 1$, for all $|y|\leq 1/M$, i.e., $M^{-1} B_2^n\subset K^\circ$.
    But $K^\circ\subset K^{\circ,p}$ \cite[Lemma 3.1]{BMR23}, so the latter contains a neighborhood of the origin and the claim follows.
\end{proof}

\begin{figure}[ht]
    \centering
    \begin{tikzpicture}[scale=1.1]
\draw[->] (-1,0) -- (3,0) node[anchor=north west] { };
\draw[->] (0,-1.5) -- (0,2) node[anchor=north west] { };
\draw[thick]  (-.5, .5) -- (0,-1) -- (1.5,-1) -- (2, 1) -- (0,1.5) -- cycle;
\draw[dotted, thick]  (-.5, .5) -- (-.5, 0);
\draw[dotted, thick]  (2, 1) -- (2, 0);
\node[above right=1pt] at (1.7, 1.2){$K$};
\node[below=1pt] at (-.5, 0){$\alpha$};
\node[below=1pt] at (2, 0){$\beta$};
\end{tikzpicture}
    \caption{$\alpha$ and $\beta$}
\end{figure}
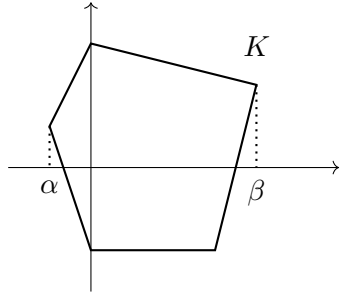

\begin{lemma}
\label{IinfiniteLemma}
Let $p\in (0,\infty]$.
Let $K\subset \R^2$ be a convex body with $0\in\mathrm{int}\,K$. For 
$
    \alpha\defeq \min\{x\in \R: (x,y)\in K\}
$ 
and 
$
\beta \defeq \max\{x\in \R:  (x,y)\in K\},
$
    \begin{equation*}
    \begin{aligned}
       |(K-(\alpha,0))^{\circ,p}\cap \{x> 0\}|,\quad
       |(K-(\beta,0))^{\circ,p}\cap\{x<0\}|<\infty,
    \cr
     |(K-(\alpha,0))^{\circ,p}\cap\{x<0\}|
     =
     |(K-(\beta,0))^{\circ,p}\cap\{x>0\}|= \infty,
     \end{aligned}
    \end{equation*}
\end{lemma}
\begin{proof}
    Since $K-(\alpha,0)\subset \{x\geq 0\}$,
    Claim \ref{MpInfinite} implies $|(K-(\alpha,0))^{\circ,p}|=\infty$, while Claim \ref{finiteClaim1} implies
     $   |(K-(\alpha,0))^{\circ,p}\cap \{x\geq 0\}|<\infty,
    $
    so necessarily $|(K-(\alpha,0))^{\circ,p}\cap\{x<0\}|=\infty$. 
The remaining claim follows in the same way. 
\end{proof}

\begin{lemma}
\label{PartialSantaloLemma}
Let $p\in (0,\infty]$.
    Let $K\subset \R^n$ be a convex body. For $u\in\partial B_2^n$, 
    \begin{equation*}
        x\mapsto |(K-x)^{\circ,p}\cap\{y: \langle y,u\rangle>0\}|
    \end{equation*}
    is log-convex in $x\in\R^n$. 
\end{lemma}

First, we need a convenient formula for $|K^{\circ,p}\cap \{y: \langle y,u\rangle\geq 0\}|$.
\begin{claim}
Let $p\in (0,\infty]$.
    For a convex body $K\subset \R^n$ and $u\in \partial B_2^n$, 
    \begin{equation}
    \label{ppolarIntEq}
        |K^{\circ,p}\cap \{y: \langle y,u\rangle\geq 0\}|= \frac{1}{n!} \int_{\{y: \langle y,u\rangle\geq 0\}} e^{-h_{p,K}(y)}dy.
    \end{equation}
\end{claim}
\begin{proof}
    In polar coordinates $y= rv$, where $r= |y|$, 
 \begin{equation*}
 \begin{aligned}
    K^{\circ,p}\cap\{y: \langle y,u\rangle\geq 0\} &= \{y\in \R^n: \|y\|_{K^{\circ,p}}\leq 1 \text{ and } \langle y,u\rangle\geq 0\} 
    \cr
    &= \{(r,v)\in (0,\infty)\cap \partial B_2^n: r\|v\|_{K^{\circ,p}}\leq 1 \text{ and } \langle v,u\rangle\geq 0\}, 
 \end{aligned}
 \end{equation*}
thus
\begin{equation}
\label{ppolarIntersectionEq1}
    \begin{aligned}
        |K^{\circ,p}\cap\{y: \langle y,u\rangle\geq 0\}|&= \int_{K^{\circ,p}\cap \{y: \langle y,u\rangle\geq 0\}}dy = \int_{\{v\in \partial B_2^n: \langle v,u\rangle\geq 0\}} \int_0^{1/\|v\|_{K^{\circ,p}}} r^{n-1} dr dv\\
        &= \frac1n \int_{\{v\in\partial B_2^n: \langle v,u\rangle\geq 0\}} \frac{dv}{\|u\|_{K^{\circ,p}}^n}. 
    \end{aligned}
\end{equation}
In addition,  
\begin{equation}
\label{ppolarIntersectionEq2}
\begin{aligned}
    \int_{\{y: \langle y,u\rangle\geq 0\}} e^{-h_{p,K}(y)}dy&= \int_{\{v\in\partial B_2^n: \langle v,u\rangle\geq 0\}} \int_0^\infty e^{-h_{p,K}(rv)} r^{n-1}drdv\\
    &= (n-1)! \int_{\{v\in\partial B_2^n: \langle u,v\rangle\geq 0\}} \frac{dv}{\|v\|_{K^{\circ,p}}^n}. 
\end{aligned}
\end{equation}
Combining \eqref{ppolarIntersectionEq1} and \eqref{ppolarIntersectionEq2} proves the claim. 
\end{proof}

\begin{proof}[Proof of Lemma \ref{PartialSantaloLemma}]
    Since $h_{p,K-x}(y)= h_{p,K}(y)-\langle x,y\rangle$ \cite[Lemma 2.2 (ii)]{BMR23}, by \eqref{ppolarIntEq}, 
    \begin{equation}
    \label{KxpPolarEq}
        |(K-x)^{\circ,p}\cap \{y: \langle y,u\rangle\geq 0\}| = \frac{1}{n!} \int_{\{y: \langle y,u\rangle\geq 0\}} e^{-h_{p,K}(y)} e^{\langle x,y\rangle} dy.
    \end{equation}
    By H\"older's inequality, 
    \begin{equation*}
        \begin{aligned}
            &|(K-(1-\lambda)x_0-\lambda x_1)^{\circ,p}\cap \{y: \langle y,u\rangle\geq 0\}| \\
            &= \frac{1}{n!}\int_{\{y: \langle y,u\rangle\geq 0\}} e^{-h_{p,K}(y)} e^{\langle  (1-\lambda)x_0+ \lambda x_1, y\rangle} dy \\
            &\leq \left( \frac{1}{n!}\int_{\{y: \langle y,u\rangle\geq 0\}} e^{-h_{p,K}(y)} e^{\langle  x_0, y\rangle} dy\right)^{1-\lambda} \left( \frac{1}{n!}\int_{\{y: \langle y,u\rangle\geq 0\}} e^{-h_{p,K}(y)} e^{\langle  x_1, y\rangle} dy\right)^\lambda \\
            &= |(K-x_0)^{\circ,p}\cap \{y: \langle y,u\rangle\geq 0\}|^{1-\lambda}  |(K- x_1)^{\circ,p}\cap \{y: \langle y,u\rangle\geq 0\}|^\lambda.
        \end{aligned}
    \end{equation*}
\end{proof}

\begin{remark}
    We may use
    dominated convergence \cite[\S 2.24]{Folland99} on \eqref{KxpPolarEq} to differentiate twice under the integral sign as in \cite[pp. 34--35]{BMR23} to get
    \begin{equation*}
        \frac{\partial}{\partial x_i} |(K-x)^{\circ,p}\cap \{y: \langle y,u\rangle\geq 0\}|= \frac{1}{n!}\int_{\{y: \langle y,u\rangle \geq 0\}} y_i e^{-h_{p,K}(y)} e^{\langle x,y\rangle}dy, 
    \end{equation*}
    and 
    \begin{equation*}
        \frac{\partial^2}{\partial x_i\partial x_j} |(K-x)^{\circ,p}\cap \{y: \langle y,u\rangle\geq 0\}|= \frac{1}{n!}\int_{\{y: \langle y,u\rangle \geq 0\}} y_i y_j e^{-h_{p,K}(y)} e^{\langle x,y\rangle}dy.
    \end{equation*}
    This is positive definite, yielding convexity (weaker than log-convexity). This then indicates the existence of a `Santal\'o point' for $K^{\circ,p}\cap\{y: \langle y,u\rangle\geq 0\}$.
\end{remark}

\begin{proof}[Proof of Proposition \ref{BalancingProp}]
   Fix $x_2$ and $x_2'$. Let 
    \begin{equation*}
        \begin{aligned}
            & \alpha \defeq \min\{x\in \R: (x,y)\in P(x_2)\}, \quad
             &\beta\defeq \max\{x\in \R: (x,y)\in P(x_2)\}, \\
            & \alpha' \defeq \min\{x\in \R: (x,y)\in P(x_2')\}, \quad
             &\beta' \defeq \max\{x\in \R: (x,y)\in P(x_2')\}. 
        \end{aligned}
    \end{equation*}
Note that $P\in\mathcal{P}$ \eqref{PclassEq} implies that the origin is an interior point of 
$\widehat P$, hence also of $P(x_2)$ and $P(x_2')$.
Thus $\alpha,\alpha'<0$ and $\beta,\beta'>0$. 
In particular, 
    \begin{equation}\label{alphabetasetEq}
(\alpha,\beta)\cap (-\beta',-\alpha')\not=\emptyset.
    \end{equation}
    There are four cases to consider:
    \begin{equation}\label{alphabetaset2Eq}
(\alpha,\beta)\cap (-\beta',-\alpha')=
\begin{cases}
(\alpha,\beta)\cr
(\alpha,-\alpha')\cr
(-\beta',-\alpha')\cr
(-\beta',\beta)\cr
\end{cases}
.
    \end{equation}
    Consider, 
    \begin{equation*}
        F(x_0)\defeq 
        G_1(x_0)-G_2(x_0),
    \end{equation*}
    where
    \begin{equation*}
    G_1(x_0):=\frac{|(P(x_2)-(x_0,0))^{\circ,p}\cap\{x>0\}|}
        {|(P(x_2)-(x_0,0))^{\circ,p}\cap\{x<0\}|},
     \qquad
        G_2(x_0):=
        \frac{|(P(x_2')-(-x_0,0))^{\circ,p}\cap\{x>0\}|}
        {|(P(x_2')-(-x_0,0))^{\circ,p}\cap\{x<0\}|}.
    \end{equation*}
We want to find a zero $x_0$ of $F$ satisfying $(x_0,0)\in\hbox{\rm int}\, P(x_2)$
and $(-x_0,0)\in\hbox{\rm int}\, P(x_2')$. In other words, we seek a solution of (recall \eqref{alphabetasetEq})
$$
F(x_0)=0, \hbox{\ \ with \ \  $x_0\in (\alpha,\beta)\cap (-\beta',-\alpha')$}.
$$
    Number the cases in \eqref{alphabetaset2Eq}
    (i)--(iv).

    \smallskip
    \noindent
    {\it Case (i):} 
    By Lemma \ref{IinfiniteLemma}, $G_1(\alpha)=0$ and $G_1(\beta)=\infty$, 
    i.e.,
    \begin{equation}
    \label{FalphaEq}
        F(\alpha)\leq 0, \qquad F(\beta)+G_2(\beta)=\infty.
    \end{equation}
    We claim that $G_2(\beta)<\infty$.
    By Lemma \ref{IinfiniteLemma}, 
    $|(P(x_2')- (\alpha',0))^{\circ,p}\cap\{x>0\}|<\infty$. 
    In addition, since $0\in\mathrm{int}\,P(x_2')$, by Claim \ref{MpInfinite}, $|P(x_2')^{\circ,p}\cap\{x>0\}|<\infty$.  
    In this case, $\alpha'\leq-\beta<0$ thus by Lemma \ref{PartialSantaloLemma} and $-\beta= \frac{-\beta}{\alpha'} \alpha' + (1-\frac{\beta}{\alpha'})0$, 
    \begin{equation*}
        \begin{aligned}
            |(P(x_2')-(-\beta,0))^{\circ,p}\cap\{x>0\}|\leq |(P(x_2')- (\alpha',0))^{\circ,p}\cap \{x>0\}|^{\frac{-\beta}{\alpha'}} |P(x_2')^{\circ,p}\cap\{x>0\}|^{1-\frac{\beta}{\alpha'}}
        \end{aligned}
    \end{equation*}
    which is finite, hence $G_2(\beta)<\infty$.
    In conclusion,
    \begin{equation}
    \label{FbetaEq}
        F(\beta)= 
        \infty
    \end{equation}
    \smallskip
    \noindent
    {\it Case (ii):} 
As in case (i), $F(\alpha)\leq 0$. 
By Lemma \ref{IinfiniteLemma}, 
$|(P(x_2')-(\alpha',0))^{\circ,p}\cap\{x<0\}|=\infty$
and 
$|(P(x_2')-(\alpha',0))^{\circ,p}\cap\{x>0\}|<\infty$.
Thus, $G_2(-\alpha')=0$, and 
    \begin{equation}
    \label{FalphaprimeEq}
        F(-\alpha')= G_1(-\alpha')\ge 0.
    \end{equation}
        
    \smallskip
    \noindent
    {\it Case (iii):} 
    Once again \eqref{FalphaprimeEq} holds. We claim that
    $G_1(-\beta')<\infty$. 
    By Lemma \ref{IinfiniteLemma}, $|(P(x_2)-(\alpha,0))^{\circ,p}\cap\{x>0\}|<\infty$. In addition, since $0\in\mathrm{int}\,P(x_2)$, by Claim \ref{MpInfinite}, $|P(x_2)^{\circ,p}\cap \{x>0\}|<\infty$. Since, in this case, $\alpha\leq -\beta'$, using Lemma \ref{PartialSantaloLemma} and $-\beta'= \frac{-\beta'}{\alpha}\alpha+ (1-\frac{-\beta'}{\alpha})0$, 
    \begin{equation*}
        |(P(x_2)-(-\beta',0))^{\circ,p}\cap\{x>0\}|\leq |(P(x_2)-(\alpha,0))^{\circ,p}\cap\{x>0\}|^{\frac{-\beta'}{\alpha}} |P(x_2)^{\circ,p}\cap\{x>0\}|^{1-\frac{\beta'}{\alpha}}
    \end{equation*}
    is finite, hence $G_1(-\beta')<\infty$.
    By Lemma \ref{IinfiniteLemma}, $G_2(-\beta')= \infty$, hence
    \begin{equation}
    \label{FminusbetaprimeEq}
        F(-\beta')=-\infty.
    \end{equation}

    \smallskip
    \noindent
    {\it Case (iv):} 
    Both \eqref{FminusbetaprimeEq} and \eqref{FbetaEq}
        hold.
    
    Proposition \ref{BalancingProp} now follows from the intermediate value theorem in each case. 
\end{proof}

\begin{remark}
\label{MRRemark}
There are other ways to prove Proposition \ref{BalancingProp}, without using Lemma \ref{PartialSantaloLemma}. One way is to
position the body in such a way that 
\begin{equation}
\label{alphaequalalphaprimeEq}
\hbox{$\alpha=\alpha'$ and $\beta=\beta'$.} 
\end{equation}
To do that one would need to argue that after
applying an affine transformation there is a vertex $p=(x_p,y_p)\in P\in\calP$ that
upon sliding does not achieve the extremal values $\alpha$
and $\beta$. That would mean these are attained at vertices
$p'$ and $p''$ other than the sliding one (the extremal values
of a linear function on a polytope are always attained on vertices) hence \eqref{alphaequalalphaprimeEq} must hold. 
For instance, for the case of 
four vertices one needs to affine transform to a cube (recall proof of Theorem \ref{2DimpMahlerSym} earlier) to avoid the situation of Figure
\ref{inverted-triangles-figure}. 
\begin{figure}[ht]
    \centering
\begin{tikzpicture}[scale=1]
\draw[->] (-2,0) -- (2,0) node[anchor=north west] { };
\draw[->] (0,-2) -- (0,2) node[anchor=north west] { };
\draw[thick]  (0,-1.5) -- (-1,1) -- (-1.25,1.5) -- (1,1) -- cycle;
\node[above right=1pt] at (1.5, 1.2){$P(x_2')$};
\end{tikzpicture}
\begin{tikzpicture}[scale=1]
\draw[->] (-2,0) -- (2,0) node[anchor=north west] { };
\draw[->] (0,-2) -- (0,2) node[anchor=north west] { };
\draw[thick]  (0,-1.5) -- (-1,1) -- (-.5,1.5) -- (1,1) -- cycle;
\node[above right=1pt] at (1.5, 1.2){$P(x_2)$};
\end{tikzpicture}
\begin{tikzpicture}[scale=1]
\draw[->] (-2,0) -- (2,0) node[anchor=north west] { };
\draw[->] (0,-2) -- (0,2) node[anchor=north west] { };
\draw[thick]  (0,-1.5) -- (-1,1) -- (1.25,1.5) -- (1,1) -- cycle;
\node[above right=1pt] at (1.5, 1.2){$P(x_2'')$};
\end{tikzpicture}
    \caption{$\alpha'<\alpha$}
    \label{inverted-triangles-figure}
\end{figure}
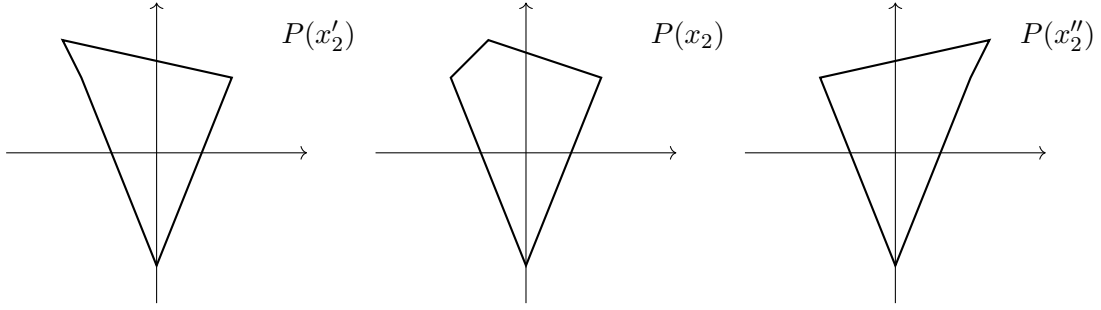
Lemma \ref{PartialSantaloLemma} seemed like a good alternative as it relies on a Santal\'o-type argument not specific to planar bodies.
It also could be used to prove \cite[Lemma 5]{MR06}.  
\end{remark}

\subsection{Proof of Theorem \ref{2DimpMahler}}
\label{nonSymFinishSection}
\begin{proof}[Proof of Lemma \ref{SlidingLemma}]
Let $x_2, x_2'\in[\xi_\ell,\xi_r]$ and recall
the notation \eqref{IpmDefEq}. By Proposition \ref{BalancingProp}, there exists $x_0\in \R$ with $(x_0, 0), (-x_0,0)\in\mathrm{int}\,\widehat{P}$, 
and
\begin{equation}
\label{rhoDef}
    \frac{I_+(x_2,x_0)}{I_-(x_2,x_0)}= 
    \frac{I_+(x_2',-x_0)}{I_-(x_2',-x_0)}=:\rho.
\end{equation}
Thus, 
\begin{equation}
\label{rhoconseqEq}
    |P((x_2)-(x_0, 0))^{\circ,p}|=
    (1+\rho){I_-(x_2,x_0)},\quad
    |P((x_2')-(-x_0, 0))^{\circ,p}|=
    (1+\rho){I_-(x_2',-x_0)}.
    \end{equation}
By Corollary \ref{ConvexityOfI}, 
\begin{equation}
\label{I+Conv}
    \frac{2}{I_\pm((x_2+x_2')/2,0)}\leq 
    \frac{1}{I_\pm(x_2,x_0)}+ \frac{1}{I_\pm(x_2',-x_0)}, 
\end{equation}
By \eqref{IpmDefEq}, \eqref{rhoDef}, \eqref{rhoconseqEq}, and \eqref{I+Conv}, 
\begin{equation}
\label{FinalPEq}
    \begin{aligned}
        |P((x_2+x_2')/2)^{\circ,p}|
        &= I_+((x_2+x_2')/2,0)+ I_-((x_2+x_2')/2,0)
        \\
        &\geq \frac{2 I_+(x_2,x_0) I_+(x_2',-x_0)}{I_+(x_2,x_0)+ I_+(x_2',-x_0)} 
        + \frac{2 I_-(x_2,x_0) I_-(x_2',-x_0)}{I_-(x_2,x_0)+ I_-(x_2',-x_0)}
        \\
        &= (1+\rho)\frac{2 I_-(x_2,x_0) I_-(x_2',-x_0)}{I_-(x_2,x_0)+ I_-(x_2',-x_0)}  \\
        &= (1+\rho)\frac{2 \frac{1}{1+\rho}|P((x_2)-(x_0, 0))^{\circ,p}| \frac{1}{1+\rho}|(P(x_2')-(-x_0, 0))^{\circ,p}|}{\frac{1}{1+\rho} |(P(x_2)-(x_0, 0))^{\circ,p}|+ \frac{1}{1+\rho}|(P(x_2')-(-x_0, 0))^{\circ,p}|}  \\
        &= \frac{2|P((x_2)-(x_0, 0))^{\circ,p}| |(P(x_2')-(-x_0, 0))^{\circ,p}|}{|(P(x_2)-(x_0, 0))^{\circ,p}|+ |(P(x_2')-(-x_0, 0))^{\circ,p}|}. 
    \end{aligned}
\end{equation}
Translating $P((x_2+x_2')/2)$ by $s_p(P((x_2+x_2')/2))$ (always an interior
point by \cite[Proposition 1.5]{BMR23}) we may assume that $s_p(P((x_2+x_2')/2))=0$. 
Thus, 
$P((x_2+x_2')/2)^{*,p}= P((x_2+x_2')/2)^{\circ,p}$ (recall the notation
\eqref{starpnotationEq}), hence \eqref{FinalPEq} gives
\begin{equation*}
\begin{aligned}
    \frac{2}{|P((x_2+x_2')/2)^{*,p}|}&= 
    \frac{2}{|P((x_2+x_2')/2)^{\circ,p}|}\\
    &\leq \frac{1}{|(P(x_2)-(x_0,0))^{\circ,p}|}+ \frac{1}{|(P(x_2)-(-x_0,0))^{\circ,p}|}\\
    &\leq \frac{1}{|P(x_2)^{*,p}|}+\frac{1}{|P(x_2')^{*,p}|},
\end{aligned}
\end{equation*}
by the minimizing property of the $L^p$-Santal\'o point
\cite[Proposition 1.5]{BMR23}.
\end{proof}

\begin{proof}[Proof of Proposition \ref{SlidingProp}]
        $\M_p$ is invariant under the action of $GL(n,\R)$ \cite[Lemma 4.6]{BMR23}, and hence, in particular, invariant under rotation.  We may therefore assume that $P(x_2)\in\mathcal{P}$ (recall \eqref{PclassEq}). Let 
        $\xi_\ell, \xi_r\in \R$, be as in Figure \ref{SlidingFig}, i.e., such that 
    $(x_3, y_3)\in [(\xi_\ell, y_2), (x_4, y_4)]$ and $(x_1, y_1)\in [(\xi_r, y_2), (y_m, y_m)]$. 
    Let also $\lambda\in (0,1)$ such that $x_2= (1-\lambda)\xi_\ell+ \lambda \xi_r$. By Lemma \ref{SlidingLemma}, 
    \begin{equation*}
        \frac{1}{|P(x_2)^{*,p}|}\leq \frac{1-\lambda}{|P(\xi_\ell)^{*, p}|}+ \frac{\lambda}{|P(\xi_r)^{*,p}|}\leq \max\left\{\frac{1}{|P(\xi_\ell)^{*,p}|}, \frac{1}{|P(\xi_r)^{*,p}|}\right\}, 
    \end{equation*}
    or $|P(x_2)^{*,p}|\geq \min\{|P(\xi_\ell)^{*,p}|, |P(\xi_r)^{*,p}|\}$. Finally, $|P(x_2)|= |P(\xi_\ell)|= |P(\xi_r)|= |P|$ hence
    \begin{equation*}
    \begin{aligned}
        \inf_{x\in \R^2}\M_p(P(x_2)-x)&= 2|P||P(x_2)^{*,p}|\geq \min\{2|P||P(\xi_\ell)^{*,p}|, 2|P||P(\xi_r)^{*,p}|\}\\
        &= \min\left\{\inf_{x\in \R^2}\M_p(P(\xi_\ell)-x), \inf_{x\in\R^2}\M_p(P(\xi_r)-x)\right\}
    \end{aligned}
    \end{equation*}
    Proposition \ref{SlidingProp} follows since $P(\xi_\ell)$ and $P(\xi_r)$ are polytopes with $m-1$ vertices.
\end{proof}
\begin{proof}[Proof of Theorem \ref{2DimpMahler}]
 By continuity of the $L^p$-Mahler volume under the Hausdorff topology \cite[Lemma 5.11]{BMR23}, it suffices to work with polytopes. 
Let $P$ be a polytope with $m\geq 4$ vertices. Write $P_m = P$. By repeated applications of Proposition \ref{SlidingProp}, there exist polytopes $P_{m-1}, \ldots, P_4, P_3$ with $m-1, \ldots, 4, 3$ vertices respectively, such that
\begin{equation*}
    \inf_{x\in\R^2}\M_p(P_m-x)\geq \inf_{x\in\R^2}\M_p(P_{m-1}-x)\geq \cdots\geq \inf_{x\in \R^2} \M_p(P_3 - x).
\end{equation*}
Since $P_3$ is a triangle, some translation of $P_3$ is in the $GL(2,\R)$-orbit
of $\Delta_{2,0}$, so Theorem \ref{2DimpMahler} follows.  
\end{proof}

\section{The isotropic constant and Mahler sliding}
\label{BourgainSection}

This section presents the proofs of Theorems \ref{plane_slicing_nonsym} and \ref{plane_slicing_sym}.  
In contrast to the $\M_p$ that are only invariant under the action of $GL(n,\R)$, 
$\calC$ (and the isotropic constant) is an affine invariant, i.e., is in addition unchanged by translations (Lemma \ref{CaffineInvariance}). This fact simplifies many of the calculations. In particular, there are no significant differences between the symmetric and non-symmetric cases (as for $\M$).

\begin{proposition}
\label{nonsym_slicing_sliding_prop}
    For any polytope $P\subset \R^2$ with $m\geq 4$ vertices, there exists a polytope $P'$ with $m-1$ vertices satisfying $\calC(P')<\calC(P)$. 
\end{proposition}

\begin{proposition}
\label{sym_slicing_sliding_prop}
For any symmetric polytope $S\subset \R^2$ with $2m\geq 6$ vertices, there exists a symmetric polytope $S'$ with $2m-2$ vertices satisfying $\calC(S')<\calC(S)$. 
\end{proposition}

These follow directly, as before, from
the next two lemmas.

\begin{lemma}
\label{Cconvexity}
    For $P\in\mathcal{P}$, 
     $   x_2\mapsto {1}/{\calC(P(x_2))}
    $ 
    is a convex quadratic polynomial. 
\end{lemma}

\begin{lemma}
\label{CsymConvexity}
    For $S\in\mathcal{S}$, 
     $   x_2 \mapsto 1/{\calC(S(x_2))}
    $ 
    is a convex quadratic polynomial.
\end{lemma}

The remainder of this section unfolds as follows: we start by quickly verifying the affine invariance of $\calC$ (Corollary \ref{CaffineInvariance}). To establish Lemmas \ref{Cconvexity} and \ref{CsymConvexity}, it becomes necessary to partition the polytope into its moving and non-moving components. Consequently, we study the covariance matrix of a disjoint union (Lemma \ref{iso2_lem1}), and, by extension, of $\calC$ (Lemma \ref{iso2_lem2} and Corollary \ref{2dCFormula}). The moving part is a triangle or the union of two antipodal triangles. Hence, in \S \ref{Ctriangle} we compute the covariance matrix (Lemma \ref{iso_sliding_lem1}) and $\calC$ of triangles (Corollary \ref{TriangleC}). Finally, we combine the aforementioned results to prove Lemmas \ref{Cconvexity} and \ref{CsymConvexity}, Propositions \ref{nonsym_slicing_sliding_prop} and \ref{sym_slicing_sliding_prop}, and Theorems \ref{plane_slicing_nonsym} and \ref{plane_slicing_sym}.

\subsection{\texorpdfstring{Affine invariance of $\calC$}{Affine invariance of C}}
The isotropic constant is an affine invariant (e.g. \cite[p. 77]{BGV+14}), hence so is $\calC$ (for proof, see \cite[Lemma 6.1, Corollary 6.2]{BMR23}). 
\label{CAffineInvSection}
    \begin{lemma}
    \label{cov_lemAT}  For a convex body $K\subset \R^n$ and $T(x)= Ax+b, A\in GL(n,\R), b\in \R^n$ an affine transformation, 
    \begin{equation}
    \label{TCovEq}
    \mathrm{Cov}(TK)= A\mathrm{Cov}(K)A^T. 
    \end{equation}
\end{lemma}

\begin{corollary}
\label{CaffineInvariance}
      For a convex body $K\subset \R^n$ and $T(x)= Ax+b, A\in GL(n,\R), b\in \R^n$ an affine transformation, 
      \begin{equation*}
          \calC(TK)= \calC(K). 
      \end{equation*}
\end{corollary}

The following summarizes the relations between $\calC(K), L_K,$ and 
$\mathrm{Cov}(K)$ (in particular proving \eqref{CLKEq}, i.e.,
the existence of isotropic position):

\begin{corollary}
\label{isotropicpositioncor}
Let $K\subset \R^n$ be a convex body. 

\smallskip
\noindent
(i) There exists a unique
affine transformation $T$ such that $TK$ is in isotropic position, i.e.,
\begin{equation}
\label{isotropicpositionEq}
\hbox{$|TK|=1, \, b(TK)=0$, and $\mathrm{Cov}(TK)$ is a multiple
of the identity matrix, denoted $L_K^2 I_n$}.
\end{equation}

\smallskip
\noindent
(ii) 
One has,
\begin{equation}
\label{CKCTKEq}
    \frac{|K|^2}{\det\mathrm{Cov}(K)}
    =:\calC(K)
    =
    \calC(TK):= \frac{1}{\det\mathrm{Cov}(TK)}= \frac{1}{\det(L_K^2 I_n)}= \frac{1}{L_K^{2n}}.
\end{equation}
\end{corollary}

\bpf
(i)
The covariance matrix is positive definite: For $y\in \R^n$, 
\begin{equation*}
    y^T \mathrm{Cov}(K) y = \sum_{i,j=1}^n y_i \int_K x_ix_j\frac{dx}{|K|} y_j = \int_{K}\langle x,y\rangle^2 \frac{dx}{|K|}\geq 0, 
\end{equation*}
with equality if and only if $y=0$. 
Therefore, it has a square root by the spectral theorem. 
Take $A$ to be the inverse of that and apply 
Lemma \ref{cov_lemAT}, possibly taking a multiple to ensure
the resulting volume is $1$. The covariance matrix of the resulting body is then 
a multiple of the identity. Finally, translate the resulting body so its barycenter
is at the origin.

(ii)
Equation \eqref{CKCTKEq} is by
Corollary \ref{CaffineInvariance}.
\epf

\subsection{\texorpdfstring{Continuity of $\mathcal{C}$ in the Hausdorff metric}{Continuity of C in the Hausdorff metric}}
\label{CHausdorffSection}
Theorems \ref{plane_slicing_nonsym} and \ref{plane_slicing_sym} are proven by 
working with polytopes. Therefore, it is necessary to demonstrate the continuity of $\calC$ under approximations. For two sets $A,B\subset \R^n$, their \textit{Hausdorff distance} is given by 
\begin{equation*}
    d_{H}(A,B)\defeq \inf\{r>0: A\subset B+ rB_2^n \text{ and } B\subset A+ rB_2^n\}. 
\end{equation*}
\begin{claim}
\label{SymDifferenceLimit}
    For a sequence $\{K_m\}_{m\geq 1}$ converging in the Hausdorff metric to a convex body $K$, 
    \begin{equation*}
        \lim_{m\to\infty} |K_m\setminus K|= \lim_{m\to\infty} |K\setminus K_m|=0.
    \end{equation*}
\end{claim}
\begin{proof}
    First, 
    \begin{equation*}
        \lim_{m\to\infty}\bm{1}_{K_m\setminus K}(x)= 0, \quad \text{ for all } x\in (\R^n\setminus K)\cup \mathrm{int}\,K. 
    \end{equation*}
    Indeed, let $x\in \R^n\setminus K$. There is $\e>0$ such that $x+ \e B_2^n\subset \R^n\setminus K$. Therefore, for large enough $m$, $x\notin K_m$, hence $\bm{1}_{K_m\setminus K}(x)=0$. Similarly, for $x\in \mathrm{int}\,K$ and $\e>0$ such that $x+ \e B_2^n\subset K$, let $m_0\in\mathbb{N}$ such that $K\subset K_m+\e B_2^n$. Therefore, $x+ \e B_2^n\subset K_m+ \e B_2^n$ for all $m\geq m_0$, and by the cancellation principle for the Minkowski sum, $x\in K_m$ for all $m\geq m_0$.

    Second, 
    \begin{equation*}
        \lim_{m\to\infty}\bm{1}_{K\setminus K_m}(x)= 0, \quad \text{ for all } x\in (\R^n\setminus K)\cup \mathrm{int}\,K,
    \end{equation*}
    which is proved similarly. 

    The claim follows from dominated convergence \cite[\S 2.24]{Folland99}. 

\end{proof}
We start with the continuity of the first moments. 
\begin{lemma}
\label{firsMomentsLemma}
    For a sequence of convex bodies $\{K_m\}_{m\geq 1}$ converging in the Hausdorff metric to a convex body $K$, 
    \begin{equation*}
        \lim_{m\to\infty} \int_{K_m} x_i \frac{dx}{|K_m|}= \int_{K} x_i \frac{dx}{|K|}. 
    \end{equation*}
\end{lemma}
\begin{proof}
Since $K$ is compact and $K_m\to K$ in the Hausdorff sense, there is $R>0$ with $K_m\subset RB_2^n$ for all $m\in \mathbb{N}$. Hence, 
        \begin{equation*}
        \begin{aligned}
            \left| \int_{K_m}x_i dx - \int_K x_i dx \right| &\leq \int_{(K_m\setminus K)\cup (K\setminus K_m)} |x_i|dx \leq R\big( |K_m\setminus K|+ |K\setminus K_m|\big), 
        \end{aligned}
    \end{equation*}
    which converges to zero by Claim \ref{SymDifferenceLimit}. 
\end{proof}

\begin{lemma}
\label{secondMomentsLemma}
    For a sequence of convex bodies $\{K_m\}_{m\geq 1}$ converging in the Hausdorff metric to a convex body $K$, 
    \begin{equation*}
        \lim_{m\to\infty} \int_{K_m} x_i x_j \frac{dx}{|K_m|}= \int_{K} x_i x_j \frac{dx}{|K|}. 
    \end{equation*}
\end{lemma}
\begin{proof}
    Since $K$ is compact and $K_m\to K$ in the Hausdorff sense, there is $R>0$ with $K_m\subset RB_2^n$ for all $m\in \mathbb{N}$. Hence, 
        \begin{equation*}
        \begin{aligned}
            \left| \int_{K_m}x_ix_j dx - \int_K x_ix_j dx \right| &\leq \int_{(K_m\setminus K)\cup (K\setminus K_m)} |x_i x_j|dx \leq R^2\big( |K_m\setminus K|+ |K\setminus K_m|\big), 
        \end{aligned}
    \end{equation*}
    which converges to zero by Claim \ref{SymDifferenceLimit}. 
\end{proof}

By Claim \ref{SymDifferenceLimit} and
    Lemmas \ref{firsMomentsLemma}--\ref{secondMomentsLemma}:

\begin{corollary}
\label{CHcontinuity}
       For a sequence of convex bodies $\{K_m\}_{m\geq 1}$ converging in the Hausdorff metric to a convex body $K$, 
       \begin{equation*}
           \lim_{m\to\infty} \mathcal{C}(K_m)= \mathcal{C}(K). 
       \end{equation*}
\end{corollary}

\subsection{The covariance matrix of a disjoint union}
\label{UnionSection}
\begin{lemma}
\label{iso2_lem1}
    Let $K\subset \R^n$ be a convex body. For compact bodies (possibly not connected) $L,C\subset K$ with $K= L\cup C$  and $\mathrm{int}\,L\cap\mathrm{int}\,C= \emptyset$, 
    \begin{equation*}
        \mathrm{Cov}(K)= \frac{|L|}{|K|} \mathrm{Cov}(L)+ \frac{|C|}{|K|}\mathrm{Cov}(C)+ \frac{|L||C|}{|K|^2} (b(L)-b(C)) (b(L)-b(C))^T. 
    \end{equation*}
\end{lemma}
\begin{proof}
Since $K= L\cup C$ and $\mathrm{int}\,L\cap \mathrm{int}\,C=\emptyset$, $|L|+ |C|= |K|$ and $\int_K= \int_L+ \int_C$. Thus,
    \begin{equation}
    \label{xixjDecomp}
        \int_{K}x_jx_j\frac{dx}{|K|}= \int_{L}x_ix_j\frac{dx}{|K|} + \int_{C}x_i x_j \frac{dx}{|K|}= \frac{|L|}{|K|}\int_{L}x_i x_j\frac{dx}{|L|}+ \frac{|C|}{|K|}\int_C x_{i}x_j\frac{dx}{|C|}
    \end{equation}
    and, similarly, 
    \begin{equation}
    \label{xiDecomp}
        \int_K x_i\frac{dx}{|K|} =\int_L x_i\frac{dx}{|K|} + \int_{C}x_i\frac{dx}{|K|}= \frac{|L|}{|K|} \int_L x_i\frac{dx}{|L|}+ \frac{|C|}{|K|}\int_{C}x_i\frac{dx}{|C|}. 
    \end{equation}
    In other words, if
    \begin{equation*}
        I(K)\defeq \left( \int_{K}x_ix_j\frac{dx}{|K|}\right)_{i,j=1}^n
    \end{equation*}
    is the matrix of the second moments of $K$, \eqref{xixjDecomp} reads
    \begin{equation}
    \label{IDecomp}
        I(K)=\frac{|L|}{|K|} I(L) +\frac{|C|}{|K|}I(C), 
    \end{equation}
    and \eqref{xiDecomp},
    \begin{equation}
    \label{bDecomp}
        b(K) = \frac{|L|}{|K|}b(L)+ \frac{|C|}{|K|}b(C).
    \end{equation}

    By \eqref{bDecomp}, 
    \begin{equation}
    \begin{aligned}
    \label{bbEq}
        &b(K)b(K)^T 
        \\
        &= \left(\frac{|L|}{|K|}b(L)+\frac{|C|}{|K|}b(C)\right) \left( \frac{|L|}{|K|}b(L)^T+ \frac{|C|}{|K|}b(C)^T\right) \\
        &= \frac{|L|^2}{|K|^2}b(L)b(L)^T+ \frac{|C|^2}{|K|^2} b(C)b(C)^T+ \frac{|L||C|}{|K|^2}\left( b(L)b(C)^T+ b(C)b(L)^T\right) \\
        &= \frac{|L|}{|K|} \left(1-\frac{|C|}{|K|} \right)b(L)b(L)^T+ \frac{|C|}{|K|}\left(1-\frac{|L|}{|K|}\right) b(C)b(C)^T+ \frac{|L||C|}{|K|^2}\left( b(L)b(C)^T+ b(C)b(L)^T\right) \\
        &= \frac{|L|}{|K|}b(L)b(L)^T+ \frac{|C|}{|K|}b(C) b(C)^T+ \frac{|L||C|}{|K|^2}\left( b(L)b(C)^T+ b(C)b(L)^T- b(L)b(L)^T - {b(C)}{b(C)^T}\right) \\
        &= \frac{|L|}{|K|} b(L) b(L)^T+ \frac{|C|}{|K|}b(C)b(C)^T - \frac{|L||C|}{|K|^2} (b(L)-b(C))(b(L)^T- b(C)^T). 
    \end{aligned}
    \end{equation}
Finally, by \eqref{IDecomp} and \eqref{bbEq},
\begin{equation*}
    \begin{aligned}
        &\mathrm{Cov}(K)
        \\
        &= I(K) - b(K)b(K)^T
        \\
        &= \frac{|L|}{|K|}\left(I(L)-b(L)b(L)^T \right)+\frac{|C|}{|K|}\left( I(C)-b(C)b(C)^T\right)+ \frac{|L||C|}{|K|^2}(b(L)-b(C))(b(L)^T-b(C)^T) \\
        &= \frac{|L|}{|K|}\mathrm{Cov}(L)+ \frac{|C|}{|K|}\mathrm{Cov}(C)+\frac{|L||C|}{|K|^2} (b(L)-b(C)) (b(L)^T- b(C)^T),
    \end{aligned}
\end{equation*}
as claimed.
\end{proof}

\begin{lemma}\label{iso2_lem2}
    For a convex body $K\subset \R^n$ with $K= L\cup C, \mathrm{int}\,L\cap\mathrm{int}\,C=\emptyset$, and $L$ isotropic, 
    \begin{equation*}
        \frac{1}{\mathcal{C}(K)}= \frac{\det(\calC(L)^{-\frac1n} I_n + |C|\mathrm{Cov}(C)+ \frac{|C|}{|K|} b(C)b(C)^T)}{|K|^{n+2}}
    \end{equation*}
\end{lemma}
\begin{proof}
    By assumption on $L$ and Corollary \ref{isotropicpositioncor}, $b(L)=0$, $|L|=1$, 
    and $\mathrm{Cov}(L)= \calC(L)^{-\frac1n} I_n$.  By Lemma \ref{iso2_lem1}, 
    \begin{equation*}
        \mathrm{Cov}(K)= \frac{\calC(L)^{-\frac1n}}{|K|} I_n+ \frac{|C|}{|K|}\mathrm{Cov}(C)+ \frac{|C|}{|K|^2} b(C)b(C)^T. 
    \end{equation*}
    Therefore,
    \begin{equation*}
    \begin{aligned}
        \frac{1}{\mathcal{C}(K)}&= \frac{\det\mathrm{Cov}(K)}{|K|^2}
        \\
        &= \frac{1}{|K|^2}\det\left( \frac{1}{|K|}\calC(L)^{-\frac1n} I_n+ \frac{|C|}{|K|}\mathrm{Cov}(C)+ \frac{|C|}{|K|^2} b(C) b(C)^T\right)\\
        &= \frac{\det(\calC(L)^{-\frac1n} I_n + |C|\mathrm{Cov}(C)+ \frac{|C|}{|K|} b(C)b(C)^T)}{|K|^{n+2}}. 
    \end{aligned}
    \end{equation*}
    \end{proof}

For a two-dimensional convex body $K$ 
denote the coordinates of the barycenter
\begin{equation*}
    b(K)\defeq (b(K)_1, b(K)_2), 
\end{equation*}
and set
\begin{equation*}
    b^N(K)\defeq (b(K)_2, -b(K)_1), 
\end{equation*}
which is normal to the barycenter, i.e.,
$\langle b(K), b^N(K)\rangle =0$.

    \begin{corollary}
    \label{2dCFormula}
    For a convex body $K\subset \R^2$ with $K= L\cup C, \mathrm{int}\,L\cap\mathrm{int}\,C=\emptyset$, and $L$ isotropic, 
    \begin{equation*}
    \begin{aligned}
        \frac{|K|^4}{\mathcal{C}(K)}&= \calC(L)^{-1} + \calC(L)^{-\frac12}\Big( |C|\mathrm{tr}\,\mathrm{Cov}(C)+ \frac{|C|}{|K|} \mathrm{tr}\,b(C)b(C)^T\Big)
        \\
        &+ \frac{|C|^4}{\mathcal{C}(C)} + \frac{|C|^2}{|K|} b^N(C)^T\mathrm{Cov}(C) b^N(C).     
    \end{aligned}
    \end{equation*} 
    \end{corollary}
    \begin{proof}
    In dimension $n=2$, 
    $$
    b(C) b(C)^T= \begin{bmatrix} b(C)^2_1 & b(C)_1 b(C)_2 \\ b(C)_1 b(C)_2 & b(C)_2^2\end{bmatrix},
    $$
    thus 
    {\allowdisplaybreaks
    \begin{align*}
        &\det\left( \calC(L)^{-\frac12} I_2+ |C|\mathrm{Cov}(C)+ |C|/|K| b(C) b(C)^T\right)=
        \\
        &\det\begin{bmatrix}
            \calC(L)^{-\frac12}+ |C|\mathrm{Cov}(C)_{11} +\frac{|C|}{|K|} b(C)_1^2 & |C|\mathrm{Cov}(C)_{12}+ \frac{|C|}{|K|} b(C)_1 b(C)_2 \\
            |C|\mathrm{Cov}(C)_{21} +\frac{|C|}{|K|} b(C)_1 b(C)_2 &  
            \calC(L)^{-\frac12}+ |C|\mathrm{Cov}(C)_{22} + \frac{|C|}{|K|} b(C)_2^2
        \end{bmatrix}
        \\
        &= \calC(L)^{-1}+ \calC(L)^{-\frac12} |C| \mathrm{Cov}(C)_{22}+ \calC(L)^{-\frac12} \frac{|C|}{|K|}b(C)_2^2
        \\
        &\quad+ \calC(L)^{-\frac12} |C|\mathrm{Cov}(C)_{11}+ |C|^2 \mathrm{Cov}(C)_{11}\mathrm{Cov}(C)_{22}+ \frac{|C|^2}{|K|} \mathrm{Cov}(C)_{11} b(C)_2^2
        \\ 
        &\quad+\calC(L)^{-\frac12}\frac{|C|}{|K|}  b(C)_1^2+ \frac{|C|^2}{|K|}\mathrm{Cov}(C)_{22} b(C)_1^2+ \frac{|C|^2}{|K|^2} b(C)_1^2 b(C)_2^2
        \\
        &\quad-|C|^2 \mathrm{Cov}(C)_{12}\mathrm{Cov}(C)_{21}- \frac{|C|^2}{|K|} \mathrm{Cov}(C)_{21} b(C)_1 b(C)_2 
        \\
        &\quad- \frac{|C|^2}{|K|} \mathrm{Cov}(C)_{12} b(C)_1 b(C)_2- \frac{|C|^2}{|K|^2} b(C)_1^2 b(C)_2^2 
        \\
        &= \calC(L)^{-1}        
        + |C|^2 \left(\mathrm{Cov}(C)_{11}\mathrm{Cov}(C)_{22}-\mathrm{Cov}(C)_{12}\mathrm{Cov}(C)_{21} \right)
        \\
        &\quad+ \calC(L)^{-\frac12}|C| \left(\mathrm{Cov}(C)_{22}+ \mathrm{Cov}(C)_{11}\right)
        + \calC(L)^{-\frac12}\frac{|C|}{|K|} \left(b(C)_2^2+ b(C)_1^2 \right)
        \\
        &\quad+ \frac{|C|^2}{|K|}\left(\mathrm{Cov}(C)_{11} b(C)_2^2+ \mathrm{Cov}(C)_{22} b(C)_1^2 - \mathrm{Cov}(C)_{21} b(C)_1 b(C)_2- \mathrm{Cov}(C)_{12} b(C)_1 b(C)_2\right)
        \\
        &= \calC(L)^{-1}+ |C|^2 \det\mathrm{Cov}(C)+ \calC(L)^{-\frac12} |C|\mathrm{tr}\,\mathrm{Cov}(C)+ \calC(L)^{-\frac12}\frac{|C|}{|K|} (b(C)_1^2+ b(C)_2^2)
        \\
        &\quad+ \frac{|C|^2}{|K|} \left( \mathrm{Cov}(C)_{11} b(C)_2^2+ \mathrm{Cov}(C)_{22} b(C)_1^2 - \mathrm{Cov}(C)_{21} b(C)_1 b(C)_2- \mathrm{Cov}(C)_{12} b(C)_1 b(C)_2\right) 
        \\
        &= \mathcal{C}(L)^{-1}+ \frac{|C|^4}{\mathcal{C}(C)}+ \calC(L)^{-\frac12} |C|\mathrm{tr}\,\mathrm{Cov}(C)+ \calC(L)^{-\frac12} \frac{|C|}{|K|} (b(C)_1^2+ b(C)_2^2)\\
        &\quad+ \frac{|C|^2}{|K|} b^N(C)^T\mathrm{Cov}(C) b^N(C). 
    \end{align*}
    }
\end{proof}

\subsection{Isotropic constant of a sliding triangle}
\label{Ctriangle}
As in the previous sections, we aim to decompose the polytope of $m$ vertices $P\in \mathcal{P}$ into $\widehat{P}$ and $\Delta(x_2)$ as in \eqref{Deltax2Def} and \eqref{hatP}. The only moving vertex lies in the triangle, so we start by studying that. The upside in the study of $\mathrm{Cov}$ and $\calC$ is they are also invariant under translations (Lemma \ref{cov_lemAT} and Corollary \ref{CaffineInvariance}). Therefore, to further simplify calculations, let 
\begin{equation}
\label{Deltalh_def}
    \Delta_{l,h}(x_2)\defeq \co\{(-l/2, 0), (l/2,0), (x_2,h)\}, \quad l>0, h\neq 0, x_2\in \R,
\end{equation}
the triangle with basis on the $x$-axis centered at the origin and basis $l$, height $h$, and sliding parameter $x_2$. For triangles, the barycenter is the average over the vertices, hence
\begin{equation}
\label{DeltalhBar}
    b(\Delta_{l,h}(x_2))= \left( \frac{x_2}{3}, \frac{h}{3}\right).
\end{equation}

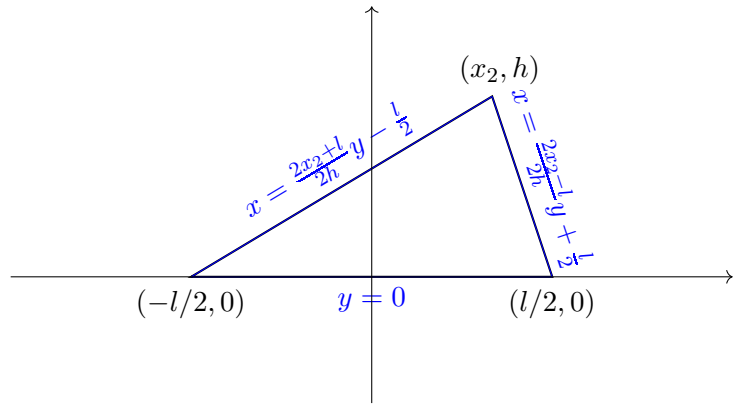
\begin{figure}[H]
    \centering
    \begin{tikzpicture}[scale=2.4]
    \draw[thick] (-1,0) -- (1,0) -- (.667,1) -- cycle;
    \draw[-, blue] (1,0)--(.667,1) node[pos=.5,sloped,above] {$x= \frac{2x_2-l}{2h}y+\frac{l}{2}$};
    \draw[-, blue] (.667,1)--(-1,0) node[pos=.5,sloped,above] {$x= \frac{2x_2+l}{2h}y-\frac{l}{2}$};
    \draw[-, blue] (-1,0)--(1,0) node[pos=.5,sloped,below] {$y=0$};
\draw[->] (-2,0) -- (2,0) node[anchor=north west] {};
\draw[->] (0,-.7) -- (0,1.5) node[anchor=north west] {};
\coordinate (O) at (0,0);
\coordinate (p1) at (1,1);
\coordinate (p2) at (-.5,1.5);
\coordinate (p3) at (-1,1);
\node[above left=1pt] at (1, 1){$(x_2,h)$};
\node[below=1pt] at (1,0){$(l/2,0)$};
\node[below=1pt] at (-1,0){$(-l/2,0)$};
\end{tikzpicture}
        \caption{\small $\Delta_{l,h}(x_2)$.}
        \label{fig:Delta_lh}
\end{figure}

\begin{lemma}
\label{iso_sliding_lem1}
    For $\Delta_{l,h}(x_2)$ as in \eqref{Deltalh_def},
    \begin{equation*}
        \mathrm{Cov}(\Delta_{l,h}(x_2))= \frac{1}{18}\begin{bmatrix}
        x_2^2+ \frac{3l^2}{4} & x_2 h\\
        x_2 h & h^2
        \end{bmatrix}. 
    \end{equation*}
\end{lemma}
\begin{proof}
Note $|\Delta_{l,h}(x_2)|= lh/2$. Moreover, $\Delta_{l,h}(x_2)$ is bounded by the lines
\begin{equation*}
        y=0, \quad 
        x= \frac{2x_2+l}{2h}y-\frac{l}{2}, \quad 
        \text{ and }
        x= \frac{2x_2-l}{2h}y+ \frac{l}{2}. 
\end{equation*}
Therefore, to compute the second moments
{\allowdisplaybreaks
\begin{align*}
        &\int_{\Delta_{l,h}(x_2)}x^2 \frac{\dif x\dif y}{|\Delta_{l,h}(x_2)|}\\
        &= \frac{2}{lh}\int_{y=0}^h \int_{x=\frac{2x_2+l}{2h}y-\frac{l}{2}}^{\frac{2x_2-l}{2h}y+ \frac{l}{2}} x^2\dif x\dif y \\
        &= \frac{2}{3lh} \int_{y=0}^h \left( \frac{2x_2-l}{2h}y+\frac{l}{2}\right)^3- \left( \frac{2x_2+l}{2h}y-\frac{l}{2}\right)^3 \dif y\\
        &= \frac{2}{3lh} \int_{0}^h\left[\left(\frac{2x_2-l}{2h}\right)^3 y^3+ 3\left(\frac{2x_2-l}{2h}\right)^2 \frac{l}{2} y^2+ 3\left( \frac{2x_2-l}{2h}\right)\frac{l^2}{4}y+ \frac{l^3}{8} \right.\\
        &\hspace{1.65cm}\left. -\left(\frac{2x_2+l}{2h} \right)^3 y^3+ 3\left( \frac{2x_2+l}{2h}\right)^2 \frac{l}{2}y^2 - 3\left(\frac{2x_2+l}{2h} \right)\frac{l^2}{4} y+ \frac{l^3}{8}\right]\dif y\\
        &= \frac{2}{3lh}\int_{0}^h -\frac{24x_2^2l+2l^3}{8h^3} y^3+ 3\frac{8x_2^2+2l^2}{4h^2} \frac{l}{2} y^2- 3\frac{2l}{2h}\frac{l^2}{4}y+ \frac{l^3}{4}\dif y\\
        &= \frac{2}{3lh}\left( -\frac{12 x_2^2 l+ l^3}{4h^3}\frac{h^4}{4}+ 3\frac{4 x_2^2+ l^2}{4h^2}l \frac{h^3}{3}- 3\frac{l^3}{4h}\frac{h^2}{2}+ \frac{l^3}{4}h\right)\\
        &= \frac{2}{3lh} \left( -\frac{(12 x_2^2l+ l^3)h}{16}+ \frac{(4x_2^2+l^2)lh}{4}-\frac{3l^3 h}{8}+\frac{l^3h}{4}\right)
        \\
        &= \frac{2}{3lh}\left( -\frac{3}{4}x_2^2 l h - \frac{l^3h}{16}+ hl x_2^2+ \frac{hl^3}{4}- \frac{3l^3 h}{8}+ \frac{l^3h}{4}\right) \\
        &= \frac{2}{3lh} \frac{lh(-12x_2^2-l^2+16x_2^2+ 4l^2-6l^2+4l^2)}{16}
        = \frac{4x_2^2+ l^2}{24}.
\end{align*}
}
In addition, 
{\allowdisplaybreaks
\begin{align*}
    \int_{\Delta_{l,h}(x_2)}y^2 \frac{\dif x\dif y}{|\Delta_{l,h}(x_2)|}&= \frac{2}{lh}\int_{y=0}^h \int_{x= \frac{2 x_2+l}{2h}y-\frac{l}{2}}^{\frac{2x_2-l}{2h}y+\frac{l}{2}} y^2 \dif x\dif y \\
    &= \frac{2}{lh} \int_{y=0}^h y^2 \left( \frac{2x_2-l}{2h}y+ \frac{l}{2}-\frac{2x_2+l}{2h}y+\frac{l}{2}\right) \dif y\\
    &= \frac{2}{lh}\int_{y=0}^h y^2 \left(-\frac{l}{h}y+ l\right)\dif y\\
    &= \frac{2}{lh} \left( -\frac{l}{h}\frac{h^4}{4}+l\frac{h^3}{3}\right)\\
    &= \frac{2}{lh} \frac{lh^3}{12}= \frac{h^2}{6}.
\end{align*}
}
Finally, 
{\allowdisplaybreaks
\begin{align*}
    &\int_{\Delta_{l,h}(x_2)} xy\frac{\dif x \dif y}{|\Delta_{l,h}(x_2)|}\\
    &= \frac{2}{lh} \int_{y=0}^h \int_{x= \frac{2x_2+l}{2h}y-\frac{l}{2}}^{\frac{2x_2-l}{2h}y+\frac{l}{2}} xy\dif x\dif y\\
    &= \frac{1}{lh}\int_{y=0}^h y\left( \left(\frac{2x_2-l}{2h}y+\frac{l}{2}\right)^2-\left(\frac{2x_2+l}{2h}y-\frac{l}{2}\right)^2\right)\dif y\\
    &= \frac{1}{lh} \int_{y=0}^h y\left(\frac{2x_2-l}{2h}y+\frac{l}{2}+ \frac{2x_2+l}{2h}y-\frac{l}{2} \right)\left(\frac{2x_2-l}{2h}y+\frac{l}{2} -\frac{2x_2 +l}{2h}y+\frac{l}{2}\right) \dif y\\
    &= \frac{1}{lh}\int_{y=0}^h y\left( \frac{2x_2}{h}y\right) \left( -\frac{l}{h}y+l\right)\dif y\\
    &= \frac{1}{lh}\int_0^h -\frac{2lx_2}{h^2}y^3+ \frac{2lx_2}{h}y^2 \dif y \\
    &= \frac{1}{lh}\left( -\frac{2lx_2}{h^2} \frac{h^4}{4}+ \frac{2lx_2}{h} \frac{h^3}{3}\right) \\
    &= \frac{1}{lh} \frac{lx_2 h^2}{6}= \frac{x_2h}{6}.
\end{align*}
}

Putting everything together, 
\begin{equation*}
    \begin{aligned}
        \mathrm{Cov}(\Delta_{l,h}(x_2))_{11}&\defeq \int_{\Delta_{l,h}(x_2)} x^2 \frac{\dif x\dif y}{|\Delta_{l,h}|}- \left(\int_{\Delta_{l,h}(x_2)} x\frac{\dif x\dif y}{|\Delta_{l,h}|}\right)^2
        = \frac{4x_2^2+ l^2}{24}- \frac{x_2^2}{9}
        = \frac{x_2^2}{18}+ \frac{l^2}{24}, 
    \end{aligned}
\end{equation*}
and
\begin{equation*}
\begin{aligned}
    \mathrm{Cov}(\Delta_{l,h}(x_2))_{22}&\defeq \int_{\Delta_{l,h}(x_2)} y^2\frac{\dif x\dif y}{|\Delta_{l,h}|}- \left( \int_{\Delta_{l,h}(x_2)}y \frac{\dif x\dif y}{|\Delta_{l,h}|}\right)^2
    = \frac{h^2}{6}- \frac{h^2}{9}= \frac{h^2}{18},
\end{aligned}
\end{equation*}
and 
\begin{equation*}
\begin{aligned}
    \mathrm{Cov}(\Delta_{l,h}(x_2))_{12}= \mathrm{Cov}(\Delta_{l,h}(x_2))_{21}&\defeq \int_{\Delta_{l,h}}x y\frac{\dif x\dif y}{|\Delta_{l,h}|}- \int_{\Delta_{l,h}(x_2)} x\frac{\dif x\dif y}{|\Delta_{l,h}|}\int_{\Delta_{l,h}(x_2)}y\frac{\dif x\dif y}{|\Delta_{l,h}|}
    \\
    &= \frac{x_2h}{6}- \frac{x_2}{3}\frac{h}{3}= \frac{x_2 h}{18}.
\end{aligned}
\end{equation*}
\end{proof}

\begin{corollary}
\label{TriangleC}
    For a triangle $\Delta\subset \R^2$, $\mathcal{C}(\Delta)= 108$. 
\end{corollary}
\begin{proof}
    By Corollary \ref{CaffineInvariance}, $\calC$ is invariant under affine transformations, so after translation and rotation 
\begin{equation*}
    \calC(\Delta)= \calC(\Delta_{l,h}(x_2)),
\end{equation*}
for appropriate $l,h>0$ and $x_2\in \R$. By Lemma \ref{iso_sliding_lem1}, 
\begin{equation*}
    \det\mathrm{Cov}(\Delta_{l,h}(x_2))= \frac{x_2^2 h^2 + \frac{3}{4}l^2h^2- x_2^2 h^2}{18^2} = \frac{l^2h^2}{432}. 
\end{equation*}
Since $|\Delta_{l,h}(x_2)|= lh/2$, 
\begin{equation*}
    \calC(\Delta_{l,h}(x_2))= \frac{|\Delta_{l,h}(x_2)|^2}{\det\mathrm{Cov}(\Delta_{l,h}(x_2))}= \frac{\frac{l^2 h^2}{4}}{\frac{l^2h^2}{432}} = 108. 
\end{equation*}
\end{proof}

\subsection{Isotropic constant of a sliding polytope}
\label{isoPolytopeSlidingSection}
For $P(x_2)\in\mathcal{P}$ and $\Delta(x_2), \widehat{P}$ as in \eqref{Deltax2Def} and \eqref{hatP} so that $P(x_2)=\widehat{P}\cup \Delta(x_2)$, let $T$ be an affine transformation so that $T\widehat{P}$ is in isotropic position. There is unique rotation $A\in O(2)$ such that $A(T(P(x_2))\in \mathcal{P}$, meaning that $A(T(x_1,y_1))$ and $A(T(x_3, y_3))$ have the same positive $y$-coordinate and $A(T(x_2,y_2))$ lies above them. Since $A(T\widehat{P})$ is still in isotropic position, it suffices to work with 
\begin{equation*}
    \widetilde{\mathcal{P}}\defeq \{P(x_2)\in\mathcal{P}: \widehat{P} \text{ in isotropic position}\}.
\end{equation*}

    \begin{figure}[H]
        \centering
         \begin{tikzpicture}[scale=1]
\draw[thick] (0, -1.4) -- (-.91, -.91) -- (-1.4,0) -- (0,1.4) -- (.91, .91) -- (1.4, 0) -- (.84,-1.26) -- cycle;
\draw[thick, dotted] (1.4,0) -- (0,-1.4);
\draw[->] (-2.3,0) -- (2.3,0) node[anchor=north west] {};
\draw[->] (0,-2) -- (0,2.5) node[anchor=north west] {};
\coordinate (O) at (0,0);
\coordinate (p1) at (1,1);
\coordinate (p2) at (-.5,1.5);
\coordinate (p3) at (-1,1);
\node[above right=1pt] at (1.4, 0){\footnotesize$(x_1, y_1)$};
\node[below right=.5pt] at (.84, -1.26){\footnotesize $(x_2, y_2)$};
\node[below=.5pt] at (0, -1.4){\footnotesize $(x_3, y_3)$};
\node[below left=.5pt] at (-.91, -.91){\footnotesize$(x_4, y_4)$};
\node[above left=.5pt] at (-1.4, 0){\footnotesize$(x_5, y_5)$};
\node[above=.5pt] at (0, 1.4){\footnotesize$(x_6, y_6)$};
\node[above right=.5pt] at (.91, .91){\footnotesize$(x_7, y_7)$};
\node[below left=.5pt] at (0, 0){\footnotesize$0$};
\end{tikzpicture}
\hspace{2cm}
\begin{tikzpicture}[scale=1]
\draw[thick] (-1,1) -- (-1.3, 0) -- (-1,-1) -- (1,-1) -- (1.3, 0) -- (1, 1) -- (-.3,1.5) -- cycle;
\draw[thick, dotted] (1,1) -- (-1,1);
\draw[->] (-2.3,0) -- (2.3,0) node[anchor=north west] {};
\draw[->] (0,-2) -- (0,2.5) node[anchor=north west] {};
\coordinate (O) at (0,0);
\coordinate (p1) at (1,1);
\coordinate (p2) at (-.5,1.5);
\coordinate (p3) at (-1,1);
\node[above right=1pt] at (1, 1){\footnotesize $(x_1, y_1)$};
\node[above=.5pt] at (-.3, 1.5){\footnotesize $(x_2, y_2)$};
\node[above left=.5pt] at (-1, 1){\footnotesize $(x_3, y_3)$};
\node[above left=.5pt] at (-1.3, 0){\footnotesize $(x_4, y_4)$};
\node[below left=.5pt] at (-1, -1){\footnotesize $(x_5, y_5)$};
\node[below right=.5pt] at (1, -1){\footnotesize $(x_6, y_6)$};
\node[above right=.5pt] at (1.3, 0){\footnotesize $(x_7, y_7)$};
\node[below left=.5pt] at (0, 0){\footnotesize $0$};
\end{tikzpicture}
\caption{\small $\widehat{P}$ remains isotropic after roation.}
\label{fig:IsoRotation}
\end{figure}
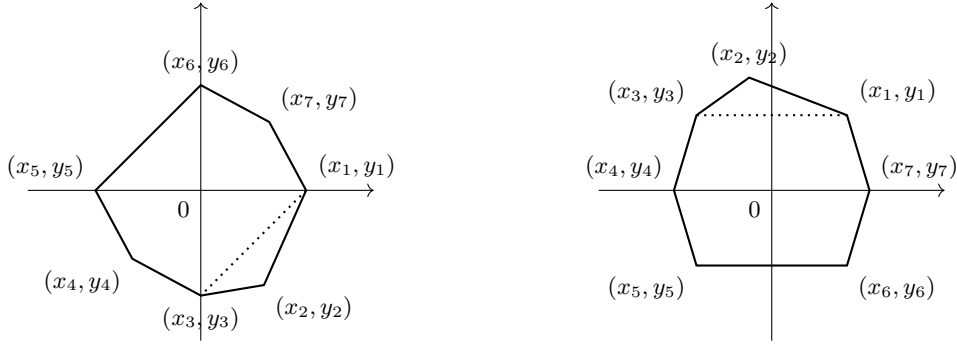

By Corollary \ref{2dCFormula}, for $P\in\widetilde{\mathcal{P}}$,
\begin{equation}
\label{iso_sliding_eq5}
\begin{aligned}
    \frac{|P|^4}{\mathcal{C}(P(x_2))}&= \calC(\widehat{P})^{-1}
    + \calC(\widehat{P})^{-\frac12}\left({|\Delta|}\mathrm{tr}\,\Cov(\Delta(x_2)) +\frac{|\Delta|}{|P|} \mathrm{tr}\, b(\Delta(x_2))b(\Delta(x_2))^T\right)\\
    &+ \frac{|\Delta|^4}{\mathcal{C}(\Delta(x_2))} + \frac{|\Delta|^2}{|P|} b^N(\Delta(x_2))^T \mathrm{Cov}(\Delta(x_2)) b^N(\Delta(x_2)).
\end{aligned}
\end{equation}
Thus, to prove Lemma \ref{Cconvexity} it is enough to compute the components of \eqref{iso_sliding_eq5}. In the notation of \S \ref{Ctriangle}, 
by Lemma \ref{iso_sliding_lem1},
\begin{equation}
\label{trEq1}
    \mathrm{tr}\,\mathrm{Cov}(\Delta_{l,h}(x_2)+(x_0,y_0))= 
    \mathrm{tr}\,\mathrm{Cov}(\Delta_{l,h}(x_2))= \frac{x_2^2}{18}+ \frac{h^2}{18}+ \frac{l^2}{24}. 
\end{equation}
Additionally, 
\begin{equation}
\label{bbTDelta}
\begin{aligned}
       b(\Delta_{l,h}(x_2)+(x_0,y_0)) b(\Delta_{l,h}(x_2)+ (x_0,y_0))^T &= 
       \begin{bmatrix}
        \frac{x_2}{3} + x_0 \\ 
        \frac{h}{3} + y_0
       \end{bmatrix}
       \begin{bmatrix}
           \frac{x_2}{3}+ x_0 & \frac{h}{3}+ y_0
       \end{bmatrix} 
       \\
       &= \begin{bmatrix}
           \left( \frac{x_2}{3}+ x_0\right)^2 & \left( \frac{x_2}{3}+ x_0\right) \left( \frac{h}{3}+ y_0\right) \\
           \left( \frac{x_2}{3}+ x_0\right) \left( \frac{h}{3}+ y_0\right) &
           \left( \frac{h}{3}+y_0\right)^2
       \end{bmatrix}, 
\end{aligned}
\end{equation}
hence
    \begin{equation}
    \label{trEq2}
        \mathrm{tr}\left( b(\Delta_{l,h}(x_2)+(x_0,y_0)) b(\Delta_{l,h}(x_2)+ (x_0,y_0))^T \right)=  \left(\frac{x_2}{3}+ x_0\right)^2+ \left(\frac{h}{3}+ y_0\right)^2. 
    \end{equation}
Finally, 
$b^N(\Delta_{l,h}(x_2)+ (x_0,y_0))= (y_0+\frac{h}{3}, -x_0-\frac{x_2}{3})$ thus, 
\begin{equation}
\label{bNEq}
\begin{aligned}
        &b^N(\Delta_{l,h}(x_2)+(x_0,y_0))^T \mathrm{Cov}(\Delta_{l,h}(x_2)+(x_0,y_0)) b^N(\Delta_{l,h}(x_2)+ (x_0,y_0)) 
        \\
        &= \frac{1}{18}\begin{bmatrix} y_0+ \frac{h}{3} & -x_0-\frac{x_2}{3}\end{bmatrix}
        \begin{bmatrix} x_2^2+ \frac{3l^2}{4} & x_2 h \\ x_2 h & h^2\end{bmatrix}
        \begin{bmatrix} y_0+ \frac{h}{3} \\ -x_0-\frac{x_2}{3}\end{bmatrix}
        \\ 
        &=\frac{1}{18}\begin{bmatrix} y_0+ \frac{h}{3} & -x_0-\frac{x_2}{3}\end{bmatrix}
        \begin{bmatrix}
        x_2^2 y_0+ x_2^2 \frac{h}{3}+ 3l^2\frac{y_0}{4}+ l^2 \frac{h}{4}- x_2 hx_0- x_2^2 \frac{h}{3}\\
        x_2hy_0+ x_2 \frac{h^2}{3}- x_0 h^2- x_2 \frac{h^2}{3}
        \end{bmatrix}
        \\
        &=\frac{1}{18}\begin{bmatrix} y_0+\frac{h}{3} & -x_0-\frac{x_2}{3}\end{bmatrix}
        \begin{bmatrix}
        x_2^2 y_0 + 3l^2\frac{y_0}{4}+ l^2 \frac{h}{4}- x_2 hx_0\\
        x_2 hy_0- x_0 h^2
        \end{bmatrix}
        \\
        &=\frac{1}{18} \left( x_2^2 y_0^2+ 3l^2\frac{y_0^2}{4} + l^2h\frac{y_0}{4}- x_2 hx_0 y_0+ x_2^2y_0\frac{h}{3}+ l^2y_0\frac{h}{4} + l^2\frac{h^2}{12} - x_2 h^2\frac{x_0}{3} \right.
        \\
        &\left. -x_2 hx_0 y_0+ x_0^2h^2- x_2^2h\frac{y_0}{3}+ x_0 h^2\frac{x_2}{3}\right) 
        \\
        &=\frac{1}{18}\left(y_0^2 x_2^2- 2hx_0 y_0 x_2+ x_0^2h^2 + \frac{3l^2 y_0^2}{4}+ \frac{l^2h y_0}{2}+ \frac{l^2h^2}{12}\right).
\end{aligned}
\end{equation}

\begin{proof}[Proof of Lemma \ref{Cconvexity}]
    By the affine invariance of $\calC$ (Corollary \ref{CaffineInvariance}) there is no loss in taking $P\in\widetilde{\mathcal{P}}$, i.e., assuming that $\widehat{P}$ is in isotropic position, as discussed at the beginning of this section.  

    There are $h,l>0$ and $(x_0, y_0)\in \R^2$ so that
    \begin{equation*}
        \Delta(x_2)= \Delta_{l,h}(x_2)+ (x_0,y_0), \quad \text{ for all } x_2\in\R.
    \end{equation*}
    By \eqref{iso_sliding_eq5}--\eqref{bNEq} and Corollary \ref{TriangleC},
    \begin{equation*}
        \begin{aligned}
            \frac{|P|^4}{\calC(P(x_2))}&= \calC(\widehat{P})^{-\frac2n}+ \calC(\widehat{P})^{-\frac1n}
            \frac{lh}{2}\left(\frac{x_2^2}{18}+ \frac{h^2}{18} +\frac{l^2}{24}\right) 
            \\
            &+ \calC(\widehat{P})^{-\frac1n} \frac{lh}{2|P|} \left(\left( \frac{x_2}{3}+ x_0\right)^2 + \left(\frac{h}{3}+ y_0 \right)^2\right)
            + \frac{1}{108} \frac{l^4h^4}{16} 
            \\
            &+ \frac{l^2h^2}{4|P|} \frac{1}{18}\left(y_0^2 x_2^2- 2hx_0 y_0 x_2+ x_0^2h^2 + \frac{3l^2 y_0^2}{4}+ \frac{l^2h y_0}{2}+ \frac{l^2h^2}{12}\right) \\
            &= \left(\calC(\widehat{P})^{-\frac1n}\frac{lh}{36} + \calC(\widehat{P})^{-\frac1n} \frac{lh}{18 |P|}+ \frac{l^2h^2 y_0^2}{72|P|}\right) x_2^2 
            + \left( \calC(\widehat{P})^{-\frac1n} \frac{lh x_0}{3|P|}- \frac{l^2h^3 x_0 y_0}{36|P|}\right) x_2\\
            &+ \calC(\widehat{P})^{-\frac2n} + \calC(\widehat{P})^{-\frac1n}\frac{lh}{2}\left( \frac{h^2}{18}+ \frac{l^2}{24}\right) +  \calC(\widehat{P})^{-\frac1n} \frac{lh}{2|P|} \left( x_0^2 + \left(\frac{h}{3}+ y_0 \right)^2\right)
            \\
            &+ \frac{1}{108} \frac{l^4h^4}{16} 
            + \frac{l^2h^2}{4|P|} \frac{1}{18}\left(x_0^2h^2 + \frac{3l^2 y_0^2}{4}+ \frac{l^2h y_0}{2}+ \frac{l^2h^2}{12}\right),
        \end{aligned}
    \end{equation*}
    is a convex quadratic polynomial in $x_2$. 
\end{proof}

Similarly, let
\begin{equation*}
    \widetilde{\mathcal{S}}\defeq \{S\in\mathcal{S}: \widehat{S} \text{ in isotropic position}.\}
\end{equation*}

\begin{proof}[Proof of Lemma \ref{CsymConvexity}]
    Let $S\in \mathcal{S}$, $\widehat{S}$ as in \eqref{hatS}, and $\Delta(x_2)$ as in \eqref{Deltax2Def} so that $S(x_2)= \widehat{S}\cup \Delta(x_2)\cup (-\Delta(x_2))$. If $T$ is an affine transformation so that $T\widehat{S}$ is in isotropic position, then after rotation (Figure \ref{fig:IsoRotation}), we may assume $TS\in \mathcal{S}$. By Corollary \ref{CaffineInvariance}, $\calC(S)= \calC(TS)$, hence we may assume $S\in \widetilde{\mathcal{S}}$. 

    Note that Corollary \ref{2dCFormula}, does not require either of the bodies to be connected. Consequently, we may take $K= S(x_2), L= \widehat{S}$ and $C(x_2)= \Delta(x_2)\cup (-\Delta(x_2))$. Note $b(C)=0$, hence $b^N(C)=0$ and 
    \begin{equation}
    \label{CsymEq1}
        \frac{|S|}{\calC(S(x_2))}= \calC(\widehat{S})^{-\frac2n}+ \calC(\widehat{S})^{-\frac1n} 2|\Delta| \mathrm{tr}\,\mathrm{Cov}(C(x_2)) + \frac{16|\Delta|^4}{\mathcal{C}(C(x_2))}. 
    \end{equation}
    
    Let $(x_0, y_0)\in \R^2$ and $l,h>0$ so that
    \begin{equation*}
        \Delta(x_2)= \Delta_{l,h}(x_2)+ (x_0, y_0). 
    \end{equation*}
    By Lemmas \ref{iso2_lem1} and \ref{iso_sliding_lem1}, and \eqref{bbTDelta},
    \begin{equation*}
        \begin{aligned}
            \mathrm{Cov}(C(x_2))&= \mathrm{Cov}(\Delta(x_2))+ b(\Delta(x_2)) b(\Delta(x_2))^T \\
            &= \mathrm{Cov}(\Delta_{l,h}(x_2))+ b(\Delta_{l,h}(x_2)+(x_0,y_0)) b(\Delta_{l,h}(x_2)+ (x_0,y_0))^T  \\
            &= \frac{1}{18}\begin{bmatrix}
        x_2^2+ \frac{3l^2}{4} & x_2 h\\
        x_2 h & h^2
        \end{bmatrix}
        + \begin{bmatrix}
           \left( \frac{x_2}{3}+ x_0\right)^2 & \left( \frac{x_2}{3}+ x_0\right) \left( \frac{h}{3}+ y_0\right) \\
           \left( \frac{x_2}{3}+ x_0\right) \left( \frac{h}{3}+ y_0\right) &
           \left( \frac{h}{3}+y_0\right)^2
       \end{bmatrix}
       \\
       &= \begin{bmatrix}
           \frac{x_2^2}{6}+ \frac{2x_0}{3}x_2+ \frac{l^2}{24}+ x_0^2 & \big(\frac{h}{6}+\frac{y_0}{3}\big)x_2 + x_0y_0 + \frac{x_0h}{3} \\
           \big(\frac{h}{6}+\frac{y_0}{3}\big)x_2 + x_0y_0 + \frac{x_0h}{3} &  \frac{h^2}{6}+ \frac{2h}{3}y_0+ y_0^2
       \end{bmatrix}
        \end{aligned}
    \end{equation*}
    because $b(-\Delta(x_2))= -b(\Delta(x_2))$ and $\mathrm{Cov}(-\Delta(x_2))= \mathrm{Cov}(\Delta(x_2))$. Thus, 
    \begin{equation}
    \label{CsymEq2}
        \mathrm{tr}\,\mathrm{Cov}(C(x_2))= \frac{x_2^2}{6}+ \frac{2x_0}{3}x_2+ \frac{l^2}{24}+ x_0^2 +  \frac{h^2}{6}+ \frac{2h}{3}y_0+ y_0^2, 
    \end{equation}
    and 
    \begin{equation}
    \label{CsymEq3}
        \frac{1}{\calC(C(x_2))}= \frac{\det\mathrm{Cov}(C(x_2))}{|C(x_2)|^2}= \frac{2y_0^2}{9|\Delta|^2} x_2^2+ \text{ lower order terms}.
    \end{equation}
    because $|C(x_2)|= 2|\Delta|$. The claim follows from the combination of \eqref{CsymEq1}, \eqref{CsymEq2}, and \eqref{CsymEq3}. 
\end{proof}

\subsection{Proofs of Theorems \ref{plane_slicing_nonsym}--\ref{plane_slicing_sym}
}
\label{isoFinishingSection}
\begin{proof}[Proof of Proposition \ref{nonsym_slicing_sliding_prop}]
                By Lemma \ref{CaffineInvariance}, $\calC$ is affine invariant, so we may assume that $P(x_2)\in\mathcal{P}$. Let $\xi_\ell, \xi_r\in \R$, be such that 
    $(x_3, y_3)\in [(\xi_\ell, y_2), (x_4, y_4)]$ and $(x_1, y_1)\in [(\xi_r, y_2), (y_m, y_m)]$. 
    Let also $\lambda\in (0,1)$ such that $x_2= (1-\lambda)\xi_\ell+ \lambda \xi_r$.
    By Lemma \ref{Cconvexity}, 
    \begin{equation*}
        \frac{1}{|\calC(P(x_2))|}< \frac{1-\lambda}{|\calC(P(\xi_\ell))|}+ \frac{\lambda}{|\calC(P(\xi_r))|}\leq \max\left\{\frac{1}{|\calC(P(\xi_\ell))|}, \frac{1}{|\calC(P(\xi_r))|}\right\}, 
    \end{equation*}
    or $|\calC(P(x_2))|> \min\{|\calC(P(\xi_\ell))|, |\calC(P(\xi_r))|\}$.
    To complete the proof, it is enough to note that $P(\xi_\ell)$ and $P(\xi_r)$ are polytopes with $m-1$ vertices.
\end{proof}

\begin{proof}[Proof of Proposition \ref{sym_slicing_sliding_prop}]
        By Lemma \ref{CaffineInvariance}, $\calC$ is affine invariant, so we can assume that $S=S(x_2)\in\mathcal{S}$. Let $\xi_\ell, \xi_r\in \R$, be such that 
    $(x_3, y_3)\in [(\xi_\ell, y_2), (x_4, y_4)]$ and $(x_1, y_1)\in [(\xi_r, y_2), (y_m, y_m)]$. 
    Let also $\lambda\in (0,1)$ such that $x_2= (1-\lambda)\xi_\ell+ \lambda \xi_r$.
    By Lemma \ref{CsymConvexity}, 
    \begin{equation*}
        \frac{1}{|\calC(S(x_2))|}< \frac{1-\lambda}{|\calC(S(\xi_\ell))|}+ \frac{\lambda}{|\calC(S(\xi_r))|}\leq \max\left\{\frac{1}{|\calC(S(\xi_\ell))|}, \frac{1}{|\calC(S(\xi_r))|}\right\}, 
    \end{equation*}
    or $|\calC(S(x_2))|> \min\{|\calC(S(\xi_\ell))|, |\calC(S(\xi_r))|\}$. To complete the proof, it is enough to note that $S(\xi_\ell)$ and $S(\xi_r)$ are symmetric polytopes with $2m-2$ vertices.
\end{proof}

\begin{proof}[Proof of Theorem \ref{plane_slicing_nonsym}]
    By continuity of $\calC$ under the Hausdorff topology (Corollary \ref{CHcontinuity}), it suffices to work with polytopes. 
Let $P$ be a polytope with $m\geq 4$ vertices. Write $P_m = P$. By repeated applications of Proposition \ref{nonsym_slicing_sliding_prop}, there exist polytopes $\{P_{i}\}_{i=3}^{m-1}$ satisfying
\begin{equation*}
    \calC(P_m)> \calC(P_{m-1})> \cdots> \calC(P_3).
\end{equation*}
Since $P_3$ is a triangle, the claim follows by Corollary \ref{TriangleC}.
\end{proof}

\begin{proof}[Proof of Theorem \ref{plane_slicing_sym}]
    By continuity of $\calC$ under the Hausdorff topology (Corollary \ref{CHcontinuity}), it suffices to work with symmetric polytopes. 
        Let $S$ be a symmetric polytope with $2m\geq 6$ vertices. Write $S_{2m} = S$. By repeated applications of Proposition \ref{sym_slicing_sliding_prop}, there exist symmetric polytopes $S_{2m-2}, \ldots, S_{6}, S_4$ with $2m-2, \ldots, 6, 4$ vertices respectively such that 
    \begin{equation*}
        \calC(S_{2m})> \calC(S_{2m-2})> \cdots> \calC(S_6) > \calC(S_4). 
    \end{equation*}
    Since $S_4$ is a symmetric polytope with four vertices, it lies in the $GL(2,\R)$ orbit of $[-1,1]^2$, hence $\calC(S_4)= \calC([-1,1]^2)$. Finally, 
    \begin{equation*}
        \mathrm{Cov}([-1,1]^2)= \begin{pmatrix}
            \int_{[-1,1]^2} x^2 \frac{dxdy}{4} & \int_{[-1,1]^2}xy \frac{dxdy}{4} \\
            \int_{[-1,1]^2} xy \frac{dx dy}{4} & \int_{[-1,1]^2} y^2 \frac{dy}{4}
        \end{pmatrix}
        = 
        \begin{pmatrix}
            \frac13 & 0 \\
            0 & \frac13, 
        \end{pmatrix}
    \end{equation*}
    hence
    \begin{equation*}
        \calC([-1,1]^2)= \frac{|[-1,1]^2|^2}{\det\mathrm{Cov}([-1,1]^2)}= \frac{16}{\frac19}= 144.
    \end{equation*}
\end{proof}

\appendix
\section{Application of Borell--Brascamp--Lieb}
\label{appendixA}

The next lemma loosely falls under the umbrella of results 
that give convexity of a `marginal type function' of  a convex function
(the protoype being the minimum principle and Pr\'ekopa--Leindler). For
convenience we include
the details in this appendix.

\begin{lemma}
\label{-1concaveLemma}
    Let $m\in\mathbb{N}$. For $f: \R\times \R^m\to (0,\infty)$ a positive and convex function, 
    \begin{equation*}
        x_2\mapsto \bigg({\int_{\R}\frac{\dif z}{f(z,x_2)^2}}\bigg)^{-1}
    \end{equation*}
    is convex function on $\R^m\to (0,\infty)$. 
\end{lemma}

Lemma \ref{-1concaveLemma} is a corollary of the classical Borell--Brascamp--Lieb inequality,
Lemma \ref{BrascampLiebineq} below \cite[Theorem 3.2]{BrascampLieb76b} \cite[Theorem 4.3]{Borell75}, itself a generalization of the Pr\'ekopa--Leinder
inequality for log-concave functions. 
Lemma \ref{BrascampLiebineq} assumes a generalized $q$-concavity on three functions
and concludes $\frac{q}{nq+1}$-concavity on their integrals. 
We now define this notion of concavity (that generalizes log-concavity, the case $q=0$).

For $\lambda\in[0,1], \, q\in [-\infty,\infty],\, a,b>0$, set
\begin{equation*}
    \mathcal{A}_{q}(\lambda, a,b)\defeq \begin{cases}
    \left( (1-\lambda)a^q+ \lambda b^q \right)^{1/q}, & q\in \R\setminus\{0\}, \\
    a^{1-\lambda}b^\lambda, & q=0, \\
    \max\{a,b\}, & q=\infty, \\
    \min\{a,b\}, & q=-\infty. 
    \end{cases}
\end{equation*}
A non-negative function $f$ on $\R^n$ is \textit{$q$-concave} \cite[Definition 1.1]{Borell75} if
\begin{equation*}
    f((1-\lambda)x+ \lambda y)\geq \mathcal{A}_q(\lambda, f(x), f(y)), \quad \text{ for all } x,y\in \R^n, \lambda\in[0,1].
\end{equation*}
For $s<r$, it follows from the convexity of $x^a$ for $a\in (1,\infty)$ that if $f$ is $r$-concave then it is also $s$-concave. 

Note that Lemma \ref{-1concaveLemma} states that a certain function is $-1$-concave. Thus,
it fits naturally to the Borell--Brascamp--Lieb framework \cite[Theorem 3.3]{BrascampLieb76b}.

\begin{lemma}[Borell--Brascamp--Lieb]
\label{BrascampLiebineq}
   Let $n\in\mathbb{N}, -1/n\leq q\leq \infty$ and $\lambda\in(0,1)$. If $f,g$ and $h$ are non-negative measurable functions on $\R^n$ satisfying
   \begin{equation*}
       h((1-\lambda)x+ \lambda y)\geq \mathcal{A}_q(\lambda, f(x), g(y)), 
   \end{equation*}
   then
   \begin{equation*}
       \int_{\R^n} h\geq \mathcal{A}_{\frac{q}{nq+1}}\left( \lambda, \int_{\R^n}f, \int_{\R^n} g\right),
   \end{equation*}
   with the convention $\frac{q}{nq+1}=-\infty$ when $q=-1/n$. 
\end{lemma}

This immediately leads to a marginal-type statement \cite[Corollary 3.5]{BrascampLieb76b} \cite[Corollary]{CampGron06}.
\begin{corollary}
\label{BLcorollary}
    If $F$ is a non-negative measurable function on $\R^n\times \R^m$ that is $q$-concave for $q\geq -1/n$, then
    \begin{equation*}
        y\mapsto \int_{\R^n}F(x,y)\dif x, 
    \end{equation*}
    is $\frac{q}{nq+1}$-concave on $\R^m$.
\end{corollary}
\begin{proof}
    Fix $\lambda\in(0,1)$ and $y,z\in\R^m$. Let, also, 
    \begin{equation*}
        \begin{aligned}
        &f(x)\defeq F(x,y), \\
        &g(x)\defeq F(x,z), \\
        &h(x)\defeq F(x, (1-\lambda)y+ \lambda z). 
        \end{aligned}
    \end{equation*}
By the $q$-concavity of $F$, 
\begin{equation*}
    h((1-\lambda)x+\lambda u)= F((1-\lambda)(x, y)+ \lambda(u,z))\geq \mathcal{A}_q(\lambda, F(x,y), F(u,z))= \mathcal{A}_q(\lambda, f(x), g(u)),  
\end{equation*}
for all $x,u\in\R^n$. Therefore, by Lemma \ref{BrascampLiebineq}, 
\begin{equation*}
\begin{aligned}
    \int_{\R^n}F(x,(1-\lambda)y + \lambda z)\dif x&= \int_{\R^n} h(x)\dif x \\
    &\geq \mathcal{A}_{\frac{q}{nq+1}} \left( \lambda, \int_{\R^n} f(x)\dif x, \int_{\R^n} g(x)\dif x\right) \\
    &= \mathcal{A}_{\frac{q}{nq+1}}\left( \lambda, \int_{\R^n}F(x,y)\dif x, \int_{\R^n}F(x,z)\dif x\right), 
\end{aligned}
\end{equation*}
as claimed. 
\end{proof}
\begin{remark}
\label{convexityRemark}
In our applications, the functions are always continuous.
For continuous functions, midpoint convexity implies convexity, hence it
will be enough to work with $\lambda=1/2$ throughout.
\end{remark}

\begin{proof}[Proof of Lemma \ref{-1concaveLemma}]
    Let
    \begin{equation*}
        F(x,y)\defeq f(x,y)^{-2}. 
    \end{equation*}
    Since $f$ is convex, $F$ is $-1/2$-concave:
    \begin{equation*}
        \begin{aligned}
            F((1-\lambda)(x,y)+ \lambda (x',y'))
            &= 
            {f\Big((1-\lambda)(x,y)+ \lambda (x', y')\Big)^{-2}} \\
            &\geq {\Big[ (1-\lambda) f(x,y)+ \lambda f(x', y')\Big]^{-2}} \\
            &= \Big[ (1-\lambda)F(x,y)^{-\frac12}+ \lambda F(x',y')^{-\frac12}\Big]^{-2} \\
            &= \mathcal{A}_{-1/2}(\lambda, F(x,y), F(x', y')).
        \end{aligned}
    \end{equation*}
    Thus, by Corollary \ref{BLcorollary} for $n=1$ and $q=-1/2$, 
    \begin{equation*}
        y\mapsto \int_{\R}F(x,y)dx= \int_{\R}\frac{dx}{f(x,y)^2}
    \end{equation*}
    is $\frac{-\frac12}{-\frac12+1}= -1$-concave, concluding the proof.
\end{proof}

\section{A theorem of Ball}
\label{BallSection}
Recall a theorem of Ball \cite[Theorem 4.10]{Ball86}; for a detailed proof see \cite[Theorem 5.20]{BMR23}.
\begin{theorem}
\label{BallIneq}
    Let $F,G,H: (0,\infty)\to [0,\infty)$ be continuous
    functions, not identically $0$, satisfying
    \begin{equation}
    \label{HFGrstEq}
        H(r)^{1/r}\geq \sqrt{F(t)^{1/t} G(s)^{1/s}}, \quad 
        \text{ for all $r,s,t>0$ with} \quad \frac{1}{r}=\frac12\bigg(\frac{1}{t}+\frac{1}{s}\bigg).
    \end{equation}
    Then, for $q\geq 1$, 
    \begin{equation*}
        \left(\int_{0}^\infty r^{q-1} H(r) \dif r \right)^{-\frac1q}\leq \frac12\left(\int_0^\infty t^{q-1} F(t) \dif t \right)^{-\frac1q}+\frac12 \left( \int_0^\infty s^{q-1} G(s)\dif s\right)^{-\frac1q}.
    \end{equation*}
\end{theorem}
Oftentimes we will
verify 
    \begin{equation}
\label{HFGrstEquivEq}
        H(r)\geq F(t)^{\frac{s}{t+s}} G(s)^{\frac{t}{t+s}},\text{ for all $r,s,t>0$ with} \quad \frac{1}{r}=\frac12\bigg(\frac{1}{t}+\frac{1}{s}\bigg),
    \end{equation}
instead of \eqref{HFGrstEq}; thanks to $\frac{1}{r}=\frac12\bigg(\frac{1}{t}+\frac{1}{s}\bigg)$, equations
\eqref{HFGrstEquivEq}
and \eqref{HFGrstEq}
are equivalent.

\phantomsection
\addcontentsline{toc}{section}{References}
\printbibliography

{\sc University of Maryland}

{\tt vmastr@umd.edu, yanir@alum.mit.edu}
\end{document}